\documentclass[11pt]{amsart}

\usepackage{amssymb, booktabs, comment, etoolbox, eucal, mathrsfs, mathtools, paralist, stmaryrd, enumitem, mdwlist}
\usepackage[colorlinks=true, pdfstartview=FitV, linkcolor=blue, citecolor=blue, urlcolor=blue, breaklinks=true]{hyperref}
\usepackage[usenames,dvipsnames]{xcolor}

\theoremstyle{plain}
\newtheorem{thm}{Theorem}[section]
\newtheorem{cor}[thm]{Corollary}
\newtheorem{lemma}[thm]{Lemma}
\newtheorem{prop}[thm]{Proposition}
\newtheorem*{thm*}{Theorem}

\theoremstyle{definition}
\newtheorem{defn}[thm]{Definition}
\newtheorem{ex}[thm]{Example}
\newtheorem{rmk}[thm]{Remark}

\def \0{\bar 0}
\def \a{\alpha}
\def \C{\Bbbk}
\def \cal{\mathcal}
\def \dr{\Delta_{\mathfrak r}}
\def \inj{\hookrightarrow}
\def \loc{\textup{loc}}
\def \max{\text{max}}
\def \min{\text{min}}
\def \N{\mathbb N}
\def \R{\mathbb R}
\def \bu{\textup{\bf{U}}}
\def \Z{{\mathbb Z}}

\def \b{\mathfrak{b}}
\def \g{\mathfrak{g}}
\def \ga{\mathfrak{g} \otimes A}
\def \gl{\mathfrak{gl}}
\def \h{\mathfrak{h}}
\def \ha{\mathfrak{h} \otimes A}
\def \n{\mathfrak{n}}
\def \nm{\mathfrak{n}^-}
\def \np{\mathfrak{n}^+}

\def \q{\mathfrak{q}}
\def \lir{\mathfrak{r}}

\def \sl{\mathfrak{sl}}
\def \ft{\mathfrak{t}}

\def \fA{\mathbf{A}}
\def \fB{\mathbf{B}}
\def \fC{\mathbf{C}}
\def \fR{\mathbf{R}}
\def \fW{\mathbf{W}}

\DeclareMathOperator{\Ann}{Ann}

\DeclareMathOperator{\dis}{dis}
\DeclareMathOperator{\dive}{div}
\DeclareMathOperator{\Ext}{Ext}
\DeclareMathOperator{\Hom}{Hom}
\DeclareMathOperator{\id}{id}
\DeclareMathOperator{\im}{im}
\DeclareMathOperator{\MaxSpec}{MaxSpec}

\DeclareMathOperator{\Supp}{Supp}
\DeclareMathOperator{\vspan}{span}

\let\atop\relax 
\newcommand{\atop}[2]{\genfrac{}{}{0pt}{}{#1}{#2}}
\newcommand{\floor}[1]{\left\lfloor{#1}\right\rfloor}
\newcommand{\ind}[2]{\textup{ind}_{#1}^{#2}\,}
\renewcommand{\L}{\textup{L}}
\newcommand{\lie}[1]{\textup{$\mathfrak{#1}$}}
\renewcommand{\mod}[1]{{#1}\textup{-\textbf{mod}}}
\newcommand{\pra}[1]{\xrightarrow{#1}}
\newcommand{\U}[1]{\textup{\bf{U}}(#1)}

\usepackage[margin=1in]{geometry}
\numberwithin{equation}{section}

\newtoggle{details}

\iftoggle{details}{%
  \newcommand{\details}[1]{
      \ \\
      {\color{OliveGreen} #1 }
  }
}{%
  \newcommand{\details}[1]{}
}

\iftoggle{details}{%
  \newcommand{\dproof}[1]{%
      \noindent
      {\color{OliveGreen}
      \textit{Proof.}#1 \qed
      }
  }
}{%
  \newcommand{\dproof}[1]{}
}

\begin{document}
%

\title{Weyl modules and Weyl functors for Lie superalgebras}

\author{Irfan Bagci}
\address{I.~Bagci: Department of Mathematics, University of North Georgia, Oakwood, GA 30566, USA}
\email{irfan.bagci@ung.edu}

\author{Lucas Calixto}
\address{L.~Calixto: Department of Mathematics, Federal University of Minas Gerais, Belo Horizonte, MG 30.123-970, Brazil}
\email{lhcalixto@ufmg.br}

\author{Tiago Macedo}
\address{T.~Macedo: Department of Science and Technology, Federal University of S\~ao Paulo, S\~ao Jos\'e dos Campos, SP 12.247-014, Brazil}
\email{tmacedo@unifesp.br}

\begin{abstract}
Given an algebraically closed field $\C$ of characteristic zero, a Lie superalgebra $\g$ over $\C$ and an associative, commutative $\C$-algebra $A$ with unit, a Lie superalgebra of the form $\g \otimes_\C A$ is known as a map superalgebra. Map superalgebras generalize important classes of Lie superalgebras, such as, loop  superalgebras (where $A=\C[t, t^{-1}]$), and current superalgebras (where $A=\C[t]$).  In this paper, we define Weyl functors, global and local Weyl modules for all map superalgebras where $\g$ is either $\sl (n,n)$ with $n \ge 2$, or a finite-dimensional simple Lie superalgebra not of type $\lie q(n)$.  Under certain conditions on the triangular decomposition of these Lie superalgebras we prove that global and local Weyl modules satisfy certain universal and tensor product decomposition properties.  We also give necessary and sufficient conditions for local (resp. global) Weyl modules to be finite dimensional (resp. finitely generated).
\end{abstract}



\maketitle

\thispagestyle{empty}

\setcounter{tocdepth}{1}
\tableofcontents

%
\section{Introduction}
%

Let $\g$ be a Lie algebra and $X$ be a scheme, both defined over a field $\C$.  Map Lie algebras (also known as generalized current Lie algebras) are Lie algebras of regular maps from $X$ to $\g$.  They form a large class of Lie algebras, whose representation theory is an extremely active area of research.  Map Lie algebras generalize loop algebras and current algebras, which are very important to the theory of affine Kac-Moody Lie algebras.

Given a finite-dimensional, simple Lie algebra $\g$ over $\mathbb C$, (local) Weyl modules for the loop algebra $\g \otimes_{\mathbb C} \mathbb C[t, t^{-1}]$ were introduced by Chari and Pressley in \cite{CP01}.  These modules are indexed by dominant integral weights of $\g$ and are closely related to certain irreducible modules for quantum affine algebras.  In \cite{FL04}, Feigin and Loktev defined local and global Weyl modules for map Lie algebras of the form $\g \otimes_{\mathbb C} A$, where $\g$ is a finite-dimensional semisimple Lie algebra and $A$ is the coordinate ring of an affine variety, both defined over $\mathbb C$.  A more general approach was taken in \cite{CFK10}, where Chari, Fourier and Khandai studied local Weyl modules, global Weyl modules, and Weyl functors for map algebras of the form $\g \otimes_{\mathbb C} A$, where $\g$ is a finite-dimensional simple Lie algebra and $A$ is an associative, commutative algebra with unit, both defined over $\mathbb C$.  In \cite{FKKS12} and \cite{FMS15}, the representation theory of local and global Weyl modules were developed for equivariant map Lie algebras, that is, Lie algebras of $\Gamma$-equivariant regular maps from an affine scheme of finite type $X$ to a finite-dimensional simple Lie algebra $\g$, both defined over an algebraically closed field $\C$ of characteristic zero, on which a finite group $\Gamma$ acts by automorphisms (both on $\g$ and $X$) and freely on the rational points of $X$.

In \cite{CLS}, Calixto, Lemay and Savage initiated the study of Weyl modules for Lie superalgebras by defining local and global Weyl modules for map superalgebras of the form $\g \otimes_{\mathbb C} A$, where $\g$ is either a finite-dimensional basic classical Lie superalgebra, or $\sl(n,n)$ with $n \ge 2$, and $A$ is an associative, commutative algebra with unit, both defined over $\mathbb C$. Weyl modules for Lie superalgebras also appear in \cite{FM} and \cite{Kus17}.

In the current paper we study global and local Weyl modules for a more general class of map superalgebras and initiate the study of Weyl functors in the super setting.  In fact, we consider map superalgebras $\g \otimes_\C A$, where $\g$ is either $\sl (n,n)$ with $n \ge 2$, or any finite-dimensional simple Lie superalgebra not of type $\lie q(n)$, and $A$ is an associative, commutative algebra with unit, both defined over an algebraically closed field $\C$ of characteristic zero.

\subsection{Main results}

Let $\C$ be an algebraically closed field of characteristic zero, $\g$ be either $\sl (n,n)$ with $n \ge 2$, or a finite-dimensional simple Lie superalgebra not of type $\lie q (n)$, and $A$ be an associative, commutative algebra with unit, both defined over $\C$.

Since there are non-conjugate Borel subsuperalgebras of simple Lie superalgebras (see for instance \cite{Ser05, Ser11, Cou14}), results for Lie superalgebras (unlike Lie algebras) may depend on the chosen triangular decomposition $\g = \n^- \oplus \h \oplus \n^+$.  In Section~\ref{S:tri.dec}, Theorem~\ref{thm:gtd-exist}, we prove that $\g$ admits triangular decompositions satisfying two important conditions \hyperref[C1]{($\mathfrak C1$)} and \hyperref[C2]{($\mathfrak C2$)}.  Another important condition on triangular decompositions is what we call \emph{parabolic} in Definition~\ref{defn.parabolic.td}.  Triangular decompositions satisfying these conditions are important because the structure of generalized Kac modules, global and local Weyl modules may change drastically if we choose different triangular decompositions.

In Section~\ref{S:glo.weyl.mod}, Definition~\ref{def:global-Weyl}, we define global Weyl modules $W_A(\lambda)$, one of the main objects of this paper.  In Section~\ref{S:super.weyl.functors}, we introduce $\fA_\lambda$, a commutative algebra that is a quotient of $\U{\ha}$.  The global Weyl module $W_A(\lambda)$ admits a structure of (right) $\fA_\lambda$-module.  Our first main result, Theorem~\ref{fin.dim.odd}, is the following:

\begin{thm*}
If $A$ is finitely generated and $\g$ is a simple Lie superalgebra not of type $\lie q(n)$, with a triangular decomposition satisfying condition \hyperref[C2]{$(\mathfrak C2)$}, then $W_A(\lambda)$ is a finitely-generated right $\fA_\lambda$-module.
\end{thm*}

In Section~\ref{sec:local-Weyl}, Definition~\ref{def:local-Weyl}, we define local Weyl modules $W_A^\loc(\psi)$, the second main object of this paper.  We proceed to give necessary and sufficient conditions (on the choice of triangular decomposition of $\g$) for local Weyl modules to be finite dimensional.  In order to do that, we denote by $w_\psi$ a highest-weight generator of $W_A^\loc(\psi)$ and define the subalgebra
\[
w_\psi^\g = \{ x \in \g \mid (x \otimes a) w_\psi = 0 \textup{ for all $a \in A$}\} \subseteq \g.
\]
Theorem~\ref{thm:loc.weyl.fd}, the other main theorem of this paper, is the following:
\begin{thm*}
Let $\g$ be a simple Lie superalgebra not of type $\q(n)$.

\begin{enumerate}[leftmargin=*]
\item If $\g$ is basic classical of type II, then $W_A^\loc(\psi)$ is finite dimensional (for every triangular decomposition).

\item If $\g$ is basic classical of type I, then $W_A^\loc(\psi)$ is finite dimensional if and only if the triangular decomposition is not a parabolic one.

\item If $\g$ is either of type $\lie{p}(n)$ or of Cartan type and $x^-_\theta$ is in the $\lie{r}$-submodule of $\g$ generated by $w_\psi^\g$, then $W_A^\loc(\psi)$ is finite-dimensional.

\item If $\g$ is either of type $\lie{p}(n)$ or of Cartan type, and the triangular decomposition of $\g$ is parabolic, then $W_A^\loc(\psi)$ is infinite-dimensional.
\end{enumerate}
\end{thm*}

This whole paper is devoted to studying global and local Weyl modules for map Lie superalgebras.  In particular, we prove that the global and local Weyl modules defined here satisfy universal properties analogous to those satisfied by other important modules.  Namely, when $\g$ is a finite-dimensional simple Lie algebra and $A=\C$, global and local Weyl modules are equal to each other, and they are irreducible finite-dimensional $\g$-modules.  When $\g$ is a finite-dimensional simple Lie superalgebra not isomorphic to $\q(n)$ and $A=\C$, global and local Weyl modules are equal to each other, and they are isomorphic to generalized Kac modules.  Finally, when $\g$ is a finite-dimensional simple Lie algebra and $A$ is an associative, commutative, algebra with unit, local and global Weyl modules defined in the current paper are isomorphic to local and global Weyl modules that have been extensively studied in several papers, such as, \cite{CP01, FL04, CFK10}.

\subsection*{Note on the arXiv version} For the interested reader, the tex file of the \href{https://arxiv.org/abs/1611.06349}{arXiv version} of this paper includes hidden details of some arguments that are omitted in the pdf file.  These details can be displayed by switching the \texttt{details} toggle to true in the tex file and recompiling.

\subsection*{Acknowledgments}  We would like to thank M.~Brito, T.~Brons, A.~Moura and A.~Savage for useful discussions.  L.~C. and T.~M. would also like to thank IMPA for the wonderful research environment and CNPq for the support during the Summer of 2017.  L.~C. was also supported by FAPESP grant 2013/08430-4 and PRPq grant ADRC - 05/2016.

%
\section{Preliminaries}\label{sec:prelim}
%

Throughout this paper $\C$ will denote an algebraically closed field of characteristic zero, $\mathbb Z$ will denote the set of integers, $\Z_2 = \{ \bar 0, \bar 1 \}$ will  denote the quotient ring $\Z/2\Z$, $\mathbb N$ will denote the set $\{ 0, 1, \dotsc \}$, and $\mathbb N_+$ will denote the set $\{ 1, 2, \dotsc \}$.  All vector spaces, algebras, and tensor products will be considered over the field $\C$ (unless otherwise specified).

A Lie superalgebra is a  $\Z_2$-graded vector space $\lie g = \lie g_{\bar0} \oplus \lie g_{\bar1}$ with a $\Z_2$-graded linear transformation $[\cdot , \cdot] \colon \g \otimes \g \to \g$ which satisfies $\Z_2$-graded versions of anticommutativity and Jacobi identity.  Given a Lie superalgebra $\lie g $, we will denote by $\U{\lie g}$ its \emph{universal enveloping superalgebra}.  Recall that the superalgebra $\U{\lie g}$ admits a PBW-type basis, that is, if $x_1, \cdots, x_m$ is a basis of $\lie g_{\bar0}$ and $y_1, \dots, y_n$ is a basis of $\lie g_{\bar1}$, then the monomials
	\[
y_1^{j_1}\dots y_n^{j_n}x_1^{i_1} \dots x_m^{i_m}, \quad i_1, \dots, i_m \geq 0 \text{ and } j_1, \dots, j_n \in \{0, 1\},
	\]
form a basis of $\U{\lie g}$.

Let $\lie f = \lie f_{\bar0} \oplus \lie f_{\bar1}$, $\lie g = \lie g_{\bar0} \oplus \lie g_{\bar1}$ be Lie superalgebras, and $M = M_{\bar0} \oplus M_{\bar1}$, $N = N_{\bar0} \oplus N_{\bar1}$ be $\g$-modules.  Throughout this paper we will assume that every \emph{homomorphism of Lie superalgebras} $\phi \colon \lie f \to \lie g$ and every \emph{homomorphism of $\g$-modules} $\psi \colon M \to N$ is even, that is, $\phi (\lie f_{\bar0}) \subseteq \lie g_{\bar0}$, $\phi (\lie f_{\bar1}) \subseteq \lie g_{\bar1}$, $\psi (M_{\bar0}) \subseteq N_{\bar0}$, and $\psi (M_{\bar1}) \subseteq N_{\bar1}$.  Notice that the category of $\g$-modules is equivalent to the category of left $\Z_2$-graded $\U\g$-modules.  In particular, the universal enveloping superalgebra $\U\g$ is a $\g$-module via left multiplication.

\begin{defn}[Finitely-semisimple module]
Let $\g$ be a Lie superalgebra.  A $\g$-module is said to be \emph{finitely semisimple} if it is equal to the direct sum of its finite-dimensional irreducible submodules.  Given a subsuperalgebra $\ft \subseteq \g$, let $\mathcal{C}_{(\g, \ft)}$ denote the full subcategory of the category of all $\g$-modules whose objects are the $\g$-modules which are finitely semisimple as $\ft$-modules.
\end{defn}

The proof of the next result is standard (see, for instance, \cite[Section~3.1 and Appendix~D]{Kum02}).  Since the category of $\g$-modules is abelian, this result implies that $\mathcal{C}_{(\g, \ft)}$ is also an abelian category.

\begin{lemma} \label{lem:I(a,b)closed}
Category $\mathcal{C}_{(\g, \ft)}$ is closed under taking submodules, quotients, arbitrary direct sums, and finite tensor products.
\qed
\end{lemma}

\dproof{
Assume that $M$ is an object of $\mathcal{C}_{(\g, \ft)}$, that is, $M$ is a $\g$-module that is finitely semisimple as an $\ft$-module.  As an $\ft$-module, any submodule of $M$ is a direct summand of $M$ (see, for instance, \cite[Exercise~III.3.1.E(3)]{Kum02}), and thus any quotient of $M$ is also finitely semisimple.

Let $J$ be a set and $M_j$, $j \in J$, be $\g$-modules.  If $M_i$ is an object of $\mathcal{C}_{(\g, \ft)}$ for each $j \in J$, then $M_j$ is a direct sum of its finite-dimensional irreducible $\ft$-submodules for each $j \in J$, and so is $\bigoplus_{j \in J} M_j$.

Now, let $n > 0$ and $M_1, \dotsc, M_n$ be objects of $\mathcal{C}_{(\g, \ft)}$.  Hence $M_i$ is a direct sum of its finite-dimensional irreducible $\ft$-submodules for each $i = 1, \dotsc, n$, and so is the tensor product $M_1 \otimes \dotsb \otimes M_n$.  Thus $M_1 \otimes \dotsb \otimes M_n$ is also an object of $\mathcal{C}_{(\g, \ft)}$.
}

Given a  Lie superalgebra $\g$, a Lie subsuperalgebra $\ft \subseteq \g$ and a $\ft$-module $M$, define the induced module $\ind{\ft}{\g}M$ to be the $\g$-module
\[
\ind{\ft}{\g} M = \U \g \otimes_{\U \ft} M,
\]
with action induced by left multiplication.

\begin{lemma} \label{lem:ind.I(a,b)}
Let $\g$ be a Lie superalgebra, $\ft \subseteq \g$ be a Lie subsuperalgebra and $M$ be a $\ft$-module.  If $\g$ (via the restriction of its adjoint representation) and $M$ are finitely-semisimple $\ft$-modules, then $\ind{\ft}{\g} M$ is an object in $\mathcal{C}_{(\g, \ft)}$.
\end{lemma}

\begin{proof}
Let $\U\g_{\rm ad_\ft}$ denote $\U \g$ regarded as a $\ft$-module via the restriction of the adjoint representation of $\g$. Since $\g$ is assumed to be a finitely-semisimple $\ft$-module via the restriction of its adjoint representation, by Lemma~\ref{lem:I(a,b)closed}, we have that $\g$, its tensor algebra and $\U \g_{\rm ad_\ft}$ are finitely-semisimple $\ft$-modules.  Moreover, since $M$ is assumed to be a finitely-semisimple $\ft$-module, by Lemma~\ref{lem:I(a,b)closed}, $\U\g_{{\rm ad}_\ft}\otimes_\C M$ a finitely-semisimple $\ft$-module.

Now, notice that the map
	\[
\U\g_{{\rm ad}_\ft}\otimes_\C M \to \U\g\otimes_{\U\ft}M, \quad u\otimes m\mapsto u\otimes m.
	\]
is a surjective homomorphism of $\ft$-modules.  In fact, for every $u\otimes m\in \ind{\ft}{\g} M$ and every homogeneous element $x \in \ft$, we have:
	\[
x \cdot (u \otimes m) = xu \otimes m = ([x,u] + (-1)^{p(u)p(x)} ux) \otimes m = [x,u] \otimes m + (-1)^{p(u)p(x)} u \otimes x \cdot m.
	\]
This shows that $\ind{\ft}{\g} M$ is a quotient of $\U\g_{{\rm ad}_\ft}\otimes_\C M$.  The result follows from Lemma~\ref{lem:I(a,b)closed}.
\end{proof}

\begin{lemma} \label{lem:ind.cyclic=cyclic}
Let $\g$ be a Lie superalgebra, $\ft \subseteq \g$ be a Lie subsuperalgebra.  If $M$ is a cyclic $\ft$-module given as the quotient of $\U\ft$ by a left ideal $J \subsetneq \U\ft$, then $\ind{\ft}{\g} M$ is a cyclic $\g$-module given as the quotient of $\U\g$ by the left ideal generated by $J$ in $\U\g$.
\end{lemma}

\begin{proof}
Using the short exact sequence $0 \to J \to \U\ft \to M \to 0$, this proof is straightforward.
\details{\textbf{Details:}
Since $M$ is assumed to be a cyclic $\ft$-module given as the quotient of $\U\ft$ by a left ideal $J \subsetneq \U\ft$, we have a short exact sequence of left $\U\ft$-modules
\begin{equation} \label{eq:ses1}
0 \to J \to \U\ft \to M \to 0.
\end{equation}
Tensoring \eqref{eq:ses1} on the left by $\U\g$, we obtain the following exact sequence of left $\U\g$-modules:
\[
\U\g \otimes_{\U\ft} J \pra{m} \U\g \to \ind{\ft}{\g} M \to 0,
\]
where $m \colon \U\g \otimes_{\U\ft} J \to \U\g$ is the unique left $\U\g$-module homomorphism satisfying
\[
m(u \otimes j) = uj \quad
\textup{for all $u \in \U\g$ and $j \in J$}.
\]
Thus, $\ind{\g}{\ft} M$ is isomorphic as a left $\U\g$-module, to the quotient of $\U\g$ by the image of $m$.  Since $m \left( \U\g \otimes_{\U\ft} J \right)$ is the left ideal generated by $J$ in $\U\g$, the result follows.
\qedhere
}
\end{proof}

\begin{table}
\begin{tabular}{lll}
\toprule
$\g$ & $\lie r$ & Type \\
\midrule
    $A(m,n)$, \ $m > n \ge 0$ & $A_m \oplus A_n \oplus \C$ & Basic, type I \\
    $A(n,n)$, \ $n \ge 1$ & $A_n \oplus A_n$ & Basic, type I \\
    $\sl(n,n)$, \ $n \ge 2$ & $A_{n-1} \oplus A_{n-1} \oplus \C$ & N/A \\
    $B(m,n)$, \ $m \ge 0$, $n \ge 1$ & $B_m \oplus C_n$ & Basic, type II \\
    $C(n+1)$, \ $n \ge 1$ & $C_n \oplus \C$ & Basic, type I \\
    $D(m,n)$, \ $m \ge 2$, $n \ge 1$ & $D_m \oplus C_n$ & Basic, type II \\
    $D(2,1;\alpha)$, \ $\alpha \ne 0,-1$ & $A_1 \oplus A_1 \oplus A_1$ & Basic, type II \\
    $F(4)$ & $A_1 \oplus B_3$ & Basic, type II \\
    $G(3)$ & $A_1 \oplus G_2$ & Basic, type II \\
    $H(n)$, \ $n \ge 4$ & $B_n$ or $D_n$ & Cartan \\
    $S(n)$, \ $n \ge 3$ & $A_{n-1}$ & Cartan \\
    $\tilde S(n)$, \ $n = 2m, m \ge 2$ & $A_{n-1}$ & Cartan \\
    $W(n)$, \ $n \ge 2$ & $A_{n-1}\oplus \C$ & Cartan \\
    $\lie p(n)$, \ $n \ge 2$ & $A_n$ & Strange \\
    $\lie q(n)$, \ $n \ge 2$ & $A_n$ & Strange \\
\bottomrule
\end{tabular}
\bigskip
\caption{}
\label{table}
\end{table}

Finite-dimensional simple Lie superalgebras over an algebraically closed field of characteristic zero were classified by V. Kac in \cite{Kac77}, and they can be divided into three groups: \emph{basic classical}, \emph{strange}, and \emph{Cartan type} (see Table~\ref{table}).  Let $\g$ be either $\sl (n,n)$ with $n \ge 2$, or a finite-di\-men\-sion\-al simple Lie superalgebra.  These Lie superalgebras admit a $\Z$-grading $\g = \bigoplus_{i \ge -2} \g_i$ (see Section~\ref{S:tri.dec} for more details). Let
\begin{equation}\label{eq:reductive.part}
  \lie r =
  \begin{cases}
    \g_0, & \text{if $\g$ is of Cartan type}, \\
    \g_{\bar 0}, & \text{otherwise}.
  \end{cases}
\end{equation}
For every $\g$, the Lie subsuperalgebra $\lie r$ is a reductive Lie algebra (see Table~\ref{table}).  Denote by $\lie r'$ the semisimple part of $\lie r$ and by $\lie z$ its center.  Fix a Cartan subalgebra $\h \subseteq \g$, (in particular, if $\g \not\cong \q(n)$, then $\h=\h_{\bar 0}$ is a Cartan subalgebra of $\lie r$, and, if $\g \cong \q(n)$, then $\h=\h_{\bar 0}\oplus \h_{\bar 1}$, $[\h_{\bar 1},\h_{\bar 1}]=\h_{\bar 0}$ and $\h_{\bar 0}$ is a Cartan subalgebra of $\lie r$), consider a triangular decomposition $\g = \n^- \oplus \h \oplus \n^+$, and let $\b={\lie h}\oplus {\lie n}^+$ be the Borel subsuperalgebra associated to this decomposition.  Notice that a triangular decomposition $\g = \n^- \oplus \h \oplus \n^+$ induces a triangular decomposition $\lie r = \n_0^- \oplus \h_{\bar 0} \oplus \n_0^+$, where $\n_0^\pm = \n^\pm \cap \lie r = \n^\pm \cap \lie r'$ and $\lie z \subseteq \lie h_{\bar 0}$.

A $\g$-module $V$ is said to be a \emph{weight module} when
\[
V = \bigoplus_{\mu \in \h_{\bar0}^*} V_\mu,
\quad \textup{where }
V_\mu = \{ v \in V \mid hv = \mu(h)v \textup{ for all } h\in \h_{\bar 0} \}.
\]
An element $\mu\in \h_{\bar 0}^*$ is said to be a \emph{weight} of $V$ when $V_\mu\neq \{0\}$, and, in this case, $V_\mu$ is said to be a \emph{weight space} of $V$, and the elements of $V_\mu$ are said to be \emph{weight vectors}.  A vector $v \in V_\mu \setminus \{0\}$ is said to be a \emph{highest-weight vector} (with respect to the fixed triangular decomposition) if $\n^+ v=0$.  Similarly, $\lambda \in \h_{\bar 0}^*$ is said to be the \emph{lowest weight} of a weight $\g$-module $V$ if $V_\lambda \ne \{0\}$ and $\nm \, V_\lambda = \{0\}$.  A $\g$-module $V$ is said to be a \emph{highest-weight module} of highest weight $\lambda \in \h_{\bar0}^*$ if $V$ is generated by a highest-weight vector $v \in V_\lambda \setminus \{0\}$.  Every irreducible finite-dimensional $\g$-module is a highest-weight module.  Denote by $L_\b(\lambda)$ the unique \emph{irreducible $\g$-module} of highest weight $\lambda$ (with respect to $\b$), and set
\begin{equation} \label{defn:X+}
X^+ = X^+(\b) = \{ \lambda \in \h_{\bar 0}^* \mid L_\b(\lambda) \textup{ is finite dimensional} \}.
\end{equation}

%
\section{Triangular decompositions} \label{S:tri.dec}
%

If $\g$ is a finite-dimensional simple Lie superalgebra, with respect to each choice of triangular decomposition $\g = \nm \oplus \h \oplus \np$, we have that, as a $\g$-module, $\g$ has a lowest weight, which we will denote by $-\theta\in \h_{\bar 0}^*$.  Let $\n_{\bar z}^\pm$ denote $\n^\pm \cap \g_{\bar z}$, $z \in \{0, 1\}$.  In this paper we will be interested in triangular decompositions $\g = \n^- \oplus \h \oplus \n^+$ satisfying the following conditions:
\begin{enumerate}
\item[({$\mathfrak C$}1)] $\nm_{\bar0} \subseteq \lie r$.  \phantomsection\label{C1}
\item[({$\mathfrak C$}2)] $\nm_{\bar0} \subseteq \lie r$ and $-\theta$ is also a root of $\lie r$. \phantomsection\label{C2}
\end{enumerate}
In fact, triangular decompositions satisfying \hyperref[C1]{$(\mathfrak C1)$} will be important to show that generalized Kac modules are finite dimensional (see Proposition~\ref{prop:Vbar-fd}), and triangular decompositions satisfying \hyperref[C2]{$(\mathfrak C2)$} will be crucial in the proof that global Weyl modules are finitely-generated right $\fA_\lambda$-modules (see Theorem~\ref{fin.dim.odd}).  This  section is devoted to constructing triangular decompositions satisfying these conditions.

\subsection{Basic classical Lie superalgebras and $\sl(n,n)$ with $n \ge 2$}

In this subsection we assume that $\g$ is a basic classical Lie superalgebra, unless otherwise specified.  In these cases, $\g_{\bar0}$ is a reductive Lie algebra.  A basic classical Lie superalgebra $\g$ is said to be of \emph{type II} if $\g_{\bar1}$ is an irreducible $\g_{\bar 0}$-module, and $\g$ is said to be of \emph{type I} if $\g_{\bar 1}$ is a direct sum of two irreducible $\g_{\bar 0}$-modules (see Table~\ref{table}).

A \emph{Cartan subalgebra} $\h \subseteq \g$ is defined to be a Cartan subalgebra of $\g_{\bar 0}$.  Under the adjoint action of $\h$, we have a root space decomposition:
\[
\g=\h \oplus \bigoplus_{\alpha \in \h^* \setminus \{0\}} \g_{\alpha},
\quad \textup{where} \quad
\g_\alpha = \{ x \in \g \mid [h, x] = \alpha(h) x \textup{ for all } h \in \h \}.
\]
Denote by $R$ the \emph{set of roots}, $\{ \alpha \in \h^* \setminus \{0\} \mid \g_\alpha \neq \{ 0 \} \}$.  For $\alpha \in R$, $\g_\alpha$ is either purely even, that is, $\g_\alpha \subseteq \g_{\bar0}$, or $\g_\alpha$ is purely odd, that is, $\g_\alpha \subseteq \g_{\bar1}$.  Let $R_{\bar0} = \{ \alpha \in R \mid \g_\alpha \subseteq \g_{\bar0} \}$ be the set of even roots and $R_{\bar1} = \{ \alpha \in R \mid \g_\alpha \subseteq \g_{\bar1} \}$ be the set of odd roots.

It is known that $\g$ can be realized as a contragradient Lie superalgebra  (for details, see \cite[Chapter~5]{Mus12}). Recall that $\Delta \subseteq R$ is a set of simple roots if, for each $\alpha \in \Delta$, there exist elements $x_\alpha\in \g_\alpha$, $y_\alpha\in \g_{-\alpha}$, such that: $\{x_\alpha, \ y_\alpha\mid \alpha\in \Delta\}\cup \h$ generates $\g$, and $[x_\alpha, y_\beta]=0$ for $\alpha \neq \beta \in \Delta$.  Denote $h_\alpha:=[x_\alpha, y_\alpha]$.  Every choice of a set of simple roots $\Delta \subseteq R$ yields a decomposition $R = R^+(\Delta) \sqcup R^-(\Delta)$, where $R^+ (\Delta)$ (resp. $R^-(\Delta)$) denotes the set of positive (resp. negative) roots (defined in the usual way).  Define
\[
\Delta_{\bar0} = \Delta \cap R_{\bar0}, \quad
\Delta_{\bar1} = \Delta \cap R_{\bar1}, \quad
R^\pm_{\bar0} = R_{\bar0} \cap R^\pm
\quad \textup{and} \quad
R^\pm_{\bar1} = R_{\bar1} \cap R^\pm.
\]
A choice of simple roots $\Delta \subseteq R$ also induces a triangular decomposition $\g = \nm (\Delta) \oplus \h \oplus \np (\Delta)$, where $\n^\pm (\Delta) = \bigoplus_{\alpha \in R^\pm (\Delta)} \g_\alpha$.  The subsuperalgebra $\b(\Delta) = \h \oplus \n^+(\Delta)$ is said to be the \emph{Borel subsuperalgebra} of $\g$ corresponding to $\Delta$.

In order to construct a triangular decomposition satisfying \hyperref[C2]{$(\mathfrak C2)$}, for each simple odd isotropic root (that is, $\beta\in\Delta_{\bar 1}$ such that $\beta(h_\beta)=0$), define the \emph{odd reflection} with respect to $\beta$ to be the map
\[
r_\beta \colon \Delta \to R,
\quad
r_\beta(\beta') =
\begin{cases}
-\beta,
& \text{ if } \beta' = \beta, \\
\beta',
& \text{ if } \beta' \in \Delta,\ \beta' \neq \beta, \textup{ and } \beta(h_{\beta'})=\beta'(h_\beta)=0, \\
\beta + \beta',
& \text{ if } \beta' \in \Delta,\ \beta' \neq \beta,\ \beta(h_{\beta'}) \neq 0 \text{ or } \beta'(h_\beta) \neq 0.
\end{cases}
\]
By \cite[Lemma~1.30]{CW12}, the set $r_\beta(\Delta)$ is a set of simple roots in $R$, and
\[
R^+(r_\beta(\Delta)) \setminus \{-\beta\} = R^+(\Delta) \setminus \{\beta\}.
\]

Now, let $\Delta_{\dis} = \{\gamma_1,\dotsc,\gamma_n\}$ be a distinguished set of simple roots of $\g$ (that is, a set of simple roots that has only one odd root), and let $\gamma_s$ denote the unique odd root in $\Delta_{\dis}$ (see \cite[Tables~3.54, 3.57-3.60]{FSS00}).  The choice of $\Delta_{\dis}$ induces a $\Z$-grading $\g=\bigoplus_{i\in\Z}\g_i$ that is compatible with the $\Z_2$-grading;  namely, $\g_{i}=\bigoplus_{{\rm ht}(\alpha)=i} \g_\alpha$, where ${\rm ht}(\alpha) = \sum_{j=1}^k {\rm ht}(\gamma_{i_j})$ if $\alpha = \gamma_{i_1} + \dotsb + \gamma_{i_k}$, and ${\rm ht}(\gamma_i)=\delta_{i,s}$ for all $i \in \{1, \dotsc, n\}$.  Explicitly:
\begin{equation} \label{eq:dis.Z.grad}
\g_{\bar 0}=\g_0
\quad \textup{and} \quad
\g_{\bar 1}=\g_{-1}\oplus \g_{1},
\quad \textup{if $\g$ is of type I},
\end{equation}
\begin{equation}
\g_{\bar 0}=\g_{-2}\oplus \g_0\oplus \g_2
\quad \textup{and} \quad
\g_{\bar 1}=\g_{-1}\oplus \g_{1},
\quad \textup{if $\g$ is of type II}.
\end{equation}
Moreover, $\Delta_{\dis}$ induces a triangular decomposition
\begin{equation}\label{eq:dis.tri.dec}
\g = \nm(\Delta_{\dis}) \oplus \h \oplus \np(\Delta_{\dis}),
\quad \textup{where} \quad
\n^\pm (\Delta_{\dis}) = \n_0^\pm \oplus \left( \bigoplus_{i>0} \g _{\pm i} \right).
\end{equation}
A subsuperalgebra $\b_{\dis}=\b(\Delta_{\dis})$ is called a distinguished Borel subsuperalgebra of $\g$, and the triangular decomposition given in \eqref{eq:dis.tri.dec} is called a distinguished triangular decomposition of $\g$.

Recall that $A(n,n)=\sl(n,n)/\Bbbk I_{n,n}$, where $I_{n,n}$ is the identity matrix in $\sl(n,n)$. Hence, the preimage of the canonical projection $\sl(n,n)\twoheadrightarrow A(n,n)$ induces decompositions as in \eqref{eq:dis.Z.grad} and \eqref{eq:dis.tri.dec} on $\sl(n,n)$.

\begin{prop} \label{prop:nice-system-exists-basic}
  Let $\g$ be a either a basic classical Lie superalgebra, or $\sl(n,n)$ with $n \ge 2$, and let $\Delta_{\dis}$ be a distinguished system of simple roots for $\g$.
\begin{enumerate}[leftmargin=*]
    \item \label{lem-item:key-system-type-II} If $\g$ is a basic classical Lie superalgebra of type II, then the triangular decomposition of $\g$ induced by $\Delta_{\dis}$ satisfies \hyperref[C2]{$(\mathfrak C2)$}.

    \item \label{lem-item:key-system-reflected} If $\g$ is $\sl(n,n)$, $n \ge 2$, or a basic classical Lie superalgebra of type I, then there exist a chain of odd reflections $\varphi$ such that the triangular decomposition of $\g$ induced by $\varphi(\Delta_{\dis})$ satisfies \hyperref[C2]{$(\mathfrak C2)$}.
  \end{enumerate}
In particular, $\g$ admits at least one triangular decomposition satisfying \hyperref[C2]{$(\mathfrak C2)$}.
\end{prop}

\begin{proof}
Notice that, since $\lie r$ is defined to be $\g_{\bar0}$, we have $\nm_{\bar0} = \nm \cap \g_{\bar0} \subseteq \lie r$ for every triangular decomposition $\g = \nm \oplus \h \oplus \np$.  That is, we only have to prove that there exists a triangular decomposition such that the associated lowest root of $\g$ is also a root of $\lie r$.

\eqref{lem-item:key-system-type-II} Recall from \eqref{eq:dis.Z.grad} that, if $\g$ is of type II, the $\Z$-grading associated to $\Delta_{\dis}$ is given by $\g = \g_{-2} \oplus \g_{-1} \oplus \g_0 \oplus \g_1 \oplus \g_2$.  Since $\nm(\Delta_{\dis})=\g_{-2}\oplus \g_{-1}\oplus \nm_0$, and $[\g_{-2}, \g_{-2}\oplus \g_{-1}]=0$, it follows that the lowest weight of $\g_{-2}$ as a $\g_0$-module is also the lowest weight of $\g$, in other words $\g_{-\theta} \subseteq \g_{-2}$.  Since $\g_{-2} \subseteq \g_{\bar0} = \lie r$, we obtain that $-\theta$ is also a root of $\lie r$.

\eqref{lem-item:key-system-reflected} Let $\lambda$ denote the highest root of $\g$ with respect to $\Delta_{\dis}$, and let $\Delta$ denote $r_{\gamma_s}(\Delta_{\dis})$.  If $\g=C(n+1)$, then one can check that $\lambda (h_{\gamma_s}) \ne 0$.  (In fact, $\lambda (h_{\gamma_s})$ is a purely positive integer linear combination of the elements of the $s$-th row of the Cartan matrix of $\g$.)  Thus, it follows from \cite[Lemma~10.2]{Ser11} that $L_{\b_{\dis}} (\lambda) \cong L_{\b(\Delta)} (\lambda-\gamma_s)$.  That is, $\theta = \lambda - \gamma_s$ is the highest root of $\g$ relative to $\lie b(\Delta)$.  Since both $\lambda$ and $\gamma_s$ are odd roots, $\theta$ is an even root.  Thus, since $\lie r$ is defined to be $\g_{\bar0}$, the lowest root $-\theta$ is a root of $\lie r$. For the cases where $\g=A(m,n)$ or ${\lie {sl}}(n,n)$, we need to apply a chain of odd reflections to obtain the result. Indeed, consider the following chain of odd reflections:
	\[
\Delta_{\dis}=\Delta_0 \overset{r_{\varepsilon_{m+1}-\varepsilon_{m+2}}}{\rightarrow} \Delta_1 \overset{r_{\varepsilon_m-\varepsilon_{m+2}}}{\rightarrow} \Delta_2 \overset{r_{\varepsilon_{m-1}-\varepsilon_{m+2}}}{\rightarrow} \cdots \overset{r_{\varepsilon_{2}-\varepsilon_{m+2}}}{\rightarrow} \Delta_m \overset{r_{\varepsilon_{1}-\varepsilon_{m+2}}}{\rightarrow}\Delta_{m+1}=\Delta',
	\]
where, following \cite[\S~2.5.4]{Kac77}, $\{\varepsilon_i\mid i=1,\ldots, m+n+2\}$ is the standard basis of the dual space of the diagonal matrices.  Since $\lambda=\varepsilon_1-\varepsilon_{m+n+2}$, we have that  $\lambda(h_{\varepsilon_k-\varepsilon_{m+2}})=0$, for all $k\geq 2$, and $\lambda(h_{\varepsilon_1-\varepsilon_{m+2}})=1$. Thus, by \cite[Lemma~10.2]{Ser11}, $L_{\b_{\dis}}(\lambda)\cong L_{\b(\Delta')}(\varepsilon_{m+2}-\varepsilon_{m+n+2})$, and $\varepsilon_{m+2}-\varepsilon_{m+n+2}$ is a root of $\lie r$.
\end{proof}

\subsection{Cartan type Lie superalgebras} \label{tri.dec.cartan}

 In this subsection $\g$ will denote a Lie superalgebra of Cartan type (see Table~\ref{table}).  We will now briefly describe each one of these Lie superalgebras.

Fix $n\geq 2$ and let $\Lambda(n)$ denote the exterior algebra with generators $\xi_1, \dotsc, \xi_n$.  The algebra $\Lambda(n)$ is a $2^n$-dimensional associative anticommutative algebra, which admits a $\Z \times \Z_2$-grading by setting the degree of $\xi_1, \dotsc, \xi_n$ to be $(1, \bar1)$.  Thus, with respect to the $\Z$-grading,
\[
\Lambda(n) = \bigoplus_{k=0}^n \Lambda^k (n),
\quad \textup{where  } \
\Lambda^k (n) = \vspan_\C \left\{ \xi_{i_1} \dotsm \xi_{i_k} \mid 1 \le i_1 < \dotsb < i_k \le n \right\}, \ k \in \{ 0, \dotsc, n \},
\]
and with respect to the $\Z_2$-grading,
\[
\Lambda(n) = \Lambda(n)_{\bar 0} \oplus \Lambda(n)_{\bar1},
\quad \textup{where} \
\Lambda(n)_{{\bar 0}} = \bigoplus_{k=0}^{\floor{n/2}} \Lambda^{2k}(n)
\ \textup{ and } \
\Lambda(n)_{\bar 1} = \bigoplus_{k=0}^{\floor{n/2}} \Lambda^{2k+1}(n).
\]
If $x \in \Lambda(n)_z$, $z \in \Z_2$, we say that $x$ is homogeneous and define $p(x) = z$.

Given a linear map $D \colon \Lambda(n) \to \Lambda(n)$, define $p(D) = \bar0$, if $D$ is even, and $p(D) = \bar1$, if $D$ is odd.  A \emph{homogeneous superderivation} of $\Lambda(n)$ is defined to be a linear map $D \colon \Lambda(n) \to \Lambda(n)$ that is either even or odd, and satisfies $D(xy) = D(x)y + (-1)^{p(D) p(x)} x D(y)$ for all homogeneous $x,y \in \Lambda(n)$.  A \emph{superderivation} of $\Lambda(n)$ is a linear combination of homogeneous superderivations of $\Lambda(n)$.

Let $W(n)$ be the Lie superalgebra consisting of superderivations of $\Lambda(n)$ endowed with the unique superbracket satisfying
	\[
	[D_1, D_2] = D_1 \circ D_2 - (-1)^{p(D_1) p(D_2)} D_2 \circ D_1,
	\]
for all homogeneous superderivations $D_1, D_2$.  The $\Z$-grading on $\Lambda(n)$ induces a $\Z$-grading:
	\[
	W(n)= W(n)_{-1} \oplus W(n)_0 \oplus \dotsb \oplus W(n)_{n-1},
	\]
where $W(n)_k$ consists of derivations that map $\xi_1, \dotsc, \xi_n$ to $\Lambda(n)_{k+1}$.

For each $i \in \{ 1, \dotsc, n \}$, denote by $\partial_i$ the unique superderivation of $\Lambda(n)$ satisfying $\partial_i (\xi_j) = \delta_{i,j}$ for all $j \in \{ 1, \dotsc, n \}$, and for each $x \in \Lambda(n)$, let $D_x$ denote the superderivation $\sum_{i=1}^n \partial_i(x) \partial_i$.  The subspace
\[
\tilde H (n) = \vspan_\C \{ D_x \mid x \in \Lambda(n) \} \subseteq W(n)
\]
is a subsuperalgebra of $W(n)$ and inherits a $\Z$-grading $\tilde{H}(n)=\tilde{H}(n)_{-1} \oplus \tilde{H}(n)_0 \oplus \dots \oplus \tilde{H}(n)_{n-2}$.  The simple Lie superalgebra $H(n)$ is defined to be the
	\[
H(n)=[\tilde{H}(n), \tilde{H}(n)]=H(n)_{-1} \oplus H(n)_0 \oplus \dots \oplus H(n)_{n-3}.
	\]

Now, let $\dive \colon W(n) \to W(n)$ be the linear transformation given by
\[
\dive (D) = \sum_{i=1}^n \partial_i (D(\xi_i))
\quad \textup{ for all }
D \in W(n).
\]
The superalgebra $S(n)$ is the subsuperalgebra of $W(n)$ consisting of all $D \in W(n)$ such that $\dive(D) = 0$.  The superalgebra $S(n)$ inherits a $\Z$-grading from $W(n)$:
	\[
	S(n)= S(n)_{-1} \oplus S(n)_0 \oplus \dots \oplus S(n)_{n-2}.
	\]

Finally we assume that $n\geq 4$ and $n$ is even. Then we define
	\[
{\tilde S}(n):=\{D\in W(n)\mid (1+\xi_1+\cdots + \xi_n)\dive(D) + D(\xi_1\cdots \xi_n)=0\}.
	\]
This Lie superalgebra does not admit a $\Z$-grading, it admits however a $\Z_n$-grading:
	\[
{\tilde S}(n)= {\tilde S}(n)_{[0]} \oplus \dots \oplus {\tilde S}(n)_{[n-1]},
	\]
where ${\tilde S}(n)_{[z]}=S(n)_z$ for every $0\leq z\leq n-2$, and ${\tilde S}(n)_{[n-1]}=\vspan_\C\{ (\xi_1\cdots\xi_n - 1)-\partial_j)\mid j=1,\ldots, n\}$.

A crucial difference between Cartan type superalgebras and basic classical superalgebras is that for Cartan type superalgebras, $\g_{\bar0}$ is not a reductive Lie algebra.  However, as was described above, if $\g$ is a Cartan type superalgebra, then $\g$ admits a ${\mathbb Z}$-grading (compatible with the $\Z_{2}$-grading, if $\g$ is not of type $\tilde S(n)$) $\g = \g_{-1} \oplus \g_0 \oplus \dotsb \oplus \g_{n-1}$, and moreover, $\g_0$ is a reductive Lie algebra, for all Cartan type Lie superalgebras.

A \emph{Cartan subalgebra} $\h$ of $\g$ is defined to be a Cartan subalgebra of $\g_{0}$.  Fix the Cartan subalgebras $\h$ of $\g_{0}$ that have the following bases:
\begin{eqnarray*}
\left\{ \xi_k \partial_k \mid 1 \le k \le n \right\},
&\textup{if}
& \g \cong W(n); \\
\left\{ \xi_k \partial_k - \xi_{k+1} \partial_{k+1} \mid 1 \le k \le n-1 \right\},
&\textup{if}
& \g \cong S(n),\ {\tilde S}(n); \\
\left\{ \xi_k \partial_k - \xi_{\floor{n/2}+k} \partial_{\floor{n/2}+k} \mid 1 \le k \le \floor{n/2} \right\},
&\textup{if}
&\g \cong H(n).
\end{eqnarray*}

Consider the element
	\[
\mathcal{E}:=\sum_{i=1}^n\xi_i\partial_i \in W(n)_0,
	\]
and define $\overline \g = \g + \C \cal E$ and $\overline \h = \h + \C \cal E$.  Notice that $\cal E$ does not lie in the the Lie superalgebras $S(n),\ {\tilde S}(n)$, or $H(n)$, thus
\[
\g = \overline\g,
\ \textup{ if $\g \cong W(n)$}
\quad \textup{ and } \quad
\g \subsetneq \g \oplus \C \cal E = \overline\g,
\ \textup{ if $\g \cong H(n), S(n), {\tilde S}(n)$.}
\]
Also notice that:
	\[
[\cal{E}, x]=zx,
\quad \text{for all $x \in \g_z$ and $z \in \Z$}.
	\]
Hence the adjoint action of $\overline{\h}$ on $\g$ gives a root space decomposition
\[
\g=\h \oplus \bigoplus_{\alpha \in \overline\h^* \setminus \{0\}} \g_{\alpha},
\quad \textup{where} \quad
\g_\alpha = \{ x \in \g \mid [h, x] = \alpha(h) x \textup{ for all } h \in \overline\h \}.
\]
Denote by $R$ the \emph{set of roots}, $\{\alpha\in \overline{\h}^\ast\setminus\{0\} \mid \g_{\alpha} \neq \{0\}\}$.  For $\alpha \in R$, either $\g_\alpha \subseteq \g_{\bar0}$, or $\g_\alpha \subseteq \g_{\bar1}$.  Let $R_{\bar0} = \{ \alpha \in R \mid \g_\alpha \subseteq \g_{\bar0} \}$ be the set of even roots and $R_{\bar1} = \{ \alpha \in R \mid \g_\alpha \subseteq \g_{\bar1} \}$ be the set of odd roots.  Moreover, for each $\a\in R$, there is $z \in \Z$ such that $\g_\alpha \subseteq \g_z$.  Thus, we can define the \emph{height} of $\alpha \in R$ to be ${\rm ht}(\alpha) = z$, and define $R_z$ to be $\{ \a \in R  \mid {\rm ht}(\a) = z\}$.  Notice that
\[
R = \bigcup_{z\in\Z} R_z, \quad
R_{\bar 0} = \bigcup_{z\in\Z} R_{2z}
\quad \textup{and} \quad
R_{\bar 1} = \bigcup_{z\in\Z} R_{2z+1}.
\]

\begin{rmk}
Notice that $\cal{E}$ captures the height of the roots.  Thus, it helps identifying the simple roots of $\g$.  The addition of $\cal E$ to $\h$ will be used in the sequel to construct triangular decompositions satisfying \hyperref[C2]{$(\mathfrak C2)$}.
\end{rmk}

We will describe now the roots and root spaces of $\g$.  Notice that, if $\g \cong W(n)$, then $\g_0 \cong \gl(n)$, with the basis elements $\xi_i \partial_j \in W(n)$ corresponding to the basis elements $E_{ij} \in \gl(n)$.  Recall that $\h = \vspan_\C \{ \xi_i \partial_i \mid 1 \le i \le n \}$ and, for each $1 \le i \le n$, let $\varepsilon_i$ be the unique linear map in $\h^*$ satisfying $\varepsilon_i (\xi_j \partial_j) = \delta_{i,j}$, for all $1 \leq j \leq n$.  The set of roots of $\g$ is
\[
R = \{ \varepsilon_{i_1} + \dotsb + \varepsilon_{i_k} - \varepsilon _j \mid 1 \leq i_1 < \dotsb < i_k \leq n, \ 0 \le k \le n \textup{ and } 1 \leq j \leq n \} \setminus \{0\},
\]
and the corresponding root spaces are:
\[
\g_\alpha =
\begin{cases}
\C \, \xi_{i_1} \dotsm \xi_{i_k} \partial_j,
& \textup{ if $\alpha = \varepsilon_{i_1} + \dotsb + \varepsilon_{i_k} -\varepsilon_j$ and $j \not\in \{i_1, \dotsc, i_k\}$}, \\
\vspan_\C \{ \xi_{i_1} \dotsm \xi_{i_k} \xi_{j} \partial_j \mid j \not\in \{i_1, \dotsc, i_k\}\},
& \textup{ if $\alpha = \varepsilon_{i_1} + \dotsb + \varepsilon_{i_k}$}.
\end{cases}
\]

If $\g = S(n)$, then $\g_0 \cong \sl(n)$.  The set of roots of $S(n)$ is the subset of the set of roots of $W(n)$ obtained from it by removing the roots $\varepsilon_{1} + \dotsb + \varepsilon_{n} - \varepsilon _j$ for all $1 \le j \le n$, that is,
\[
R =  \{ \varepsilon_{i_1} + \dotsb + \varepsilon_{i_k} - \varepsilon _j \mid 1 \leq i_1 < \dotsb < i_k \leq n, \ 0 \le k \le n-1 \textup{ and }  1 \leq j \leq n \} \setminus \{0\}.
\]
The corresponding root spaces are thus:
\[
\g_\alpha =
\begin{cases}
\C \, \xi_{i_1} \dotsm \xi_{i_k} \partial_j,
\qquad \textup{ if $\alpha = \varepsilon_{i_1} + \dotsb + \varepsilon_{i_k} -\varepsilon_j$ and $j \not\in \{i_1, \dotsc, i_k\}$}, \\
\vspan_\C \{ \xi_{i_1} \dotsm \xi_{i_k} ( \xi_{j} \partial_j - \xi_{j+1} \partial_{j+1}) \mid j, j+1 \not\in \{i_1, \dotsc, i_k\}\},
\qquad \textup{ if $\alpha = \varepsilon_{i_1} + \dotsb + \varepsilon_{i_k}$}.
\end{cases}
\]

If $\g = {\tilde S}(n)$, then (as for $S(n)$) $\g_0 \cong \sl(n)$ and
\[
R =  \{ \varepsilon_{i_1} + \dotsb + \varepsilon_{i_k} - \varepsilon _j \mid 1 \leq i_1 < \dotsb < i_k \leq n, \ 0 \le k \le n-1 \textup{ and }  1 \leq j \leq n \} \setminus \{0\}.
\]
The corresponding root spaces are: $\g_\alpha=S(n)_\alpha$ if ${\rm ht}(\alpha)\geq 0$, and $\g_{-\epsilon_i} = \vspan_\C \{ (\xi_1 \dotsm \xi_n - 1)\partial_i\}$ for all $i \in \{1, \dotsc, n\}$.

Finally, if $\g = H(n)$, then $\g_0 \cong \mathfrak{so}(n)$.  Let $r = \floor{n/2}$, $\{ \varepsilon_1, \dotsc, \varepsilon_r \}$ be the elements in the Cartan subalgebra of $\g_0$ that identify with the standard basis of the Cartan subalgebra of $\mathfrak{so}(n)$, and $\delta \in \overline\h^*$ be the dual of $\mathcal{E}$.  If $n = 2r$, then the set of roots of $\g$ is given by
	\[
R=\{\pm \varepsilon_{i_1}\pm \dotsb  \pm \varepsilon_{i_k}+ m\delta \mid 1\leq i_1<\dotsb <i_k\leq r,\ k-2 \leq m \leq n-2, \ m\geq -1, \ k-m \in 2\mathbb{Z}\}.
	\]
If $n=2r+1$, then the set of roots of $\g$ is the set
	\[
R=\{\pm \varepsilon _{i_1}\pm \dotsb  \pm \varepsilon _{i_k}+ m\delta \mid 1\leq i_1<\dotsb <i_k\leq r,\ k-2 \leq m \leq n-2, \  m\geq -1\}.
	\]
For each root $\alpha = d_1 \varepsilon_1 + \dotsb + d_r \varepsilon_r + m\delta$ with $d_i \in \{ -1, 0, 1 \}$, we have:
\[
\g_{\alpha} = \vspan_\C \left\{ D_x \mid x = \xi_1^{a_1} \dotsm \xi_n^{a_n}, \  a_i \in \{0,1\}, \  a_1 + \dotsb + a_n = m+2, \ a_i - a_{r+i} = d_i \textup{ for all } i \right\}.
\]

\begin{rmk}
Some properties that hold for roots of semisimple Lie algebras do not hold for Cartan type Lie superalgebras.  For instance, a root may have multiplicity greater than 1, and $R^-$ may be different from $-R^+$. (For more details, see \cite{Gav14,Kac77,Sch79,Ser05}.)  Also, notice that the root space decomposition of $\g$ given by the action of $\overline{\h}$ induces the root space decomposition of $\g$ given by the action of $\h$.  The roots with respect to $\h$ are the restrictions of the elements of $R$ to $\h$.  Hence the set of roots of $\g$ with respect to $\h$ will be also denoted by $R$ and their roots will be denoted by the same symbols.
\end{rmk}

An element $h \in \h_\R$ is said to be \emph{regular} if $\alpha(h) \neq 0$ for all $\alpha \in R$.  Every regular element $h \in \h_\R$ induces a decomposition $R = R^+(h) \sqcup R^-(h)$, where
\[
R^+(h) = \{ \alpha \in R \mid \alpha(h) > 0 \}
\quad \textup{ and } \quad
R^-(h) = \{ \alpha \in R \mid \alpha(h) < 0 \}.
\]
The set $R^+(h)$ (resp. $R^-(h)$) is said to be the \emph{set of positive} (resp. \emph{negative}) roots of $\g$ relative to $h$.  A regular element $h \in \h_\R$ also induces a triangular decomposition
\[
\g = \n^-(h) \oplus \h \oplus \n^+(h),
\quad \textup{where} \quad
\n^\pm(h) = \bigoplus_{\alpha \in R^\pm(h)} \g_\alpha.
\]
A Lie subsuperalgebra $\b$ is a \emph{Borel subsuperalgebra} of $\g$ if $\b = \h \oplus \n^+(h)$ for some regular $h \in \h_\R$.

Following \cite{Ser05}, a root $\alpha\in R^+$ is said to be simple for a Borel subsuperalgebra $\b$ if the set
\begin{equation}\label{eq:reflec.Cartan.type}
r_\alpha \left( R^+ \right)
=
\begin{cases}
\left( R^+ \setminus \{\alpha\} \right) \cup \{ -\alpha \},
& \textup{if } -\alpha \in R, \\
R^+\setminus \{\alpha\},
& \textup{otherwise}
\end{cases}
\end{equation}
is a set of positive roots relative to some regular element $h \in \h_\R$.  In this case, the subsuperalgebra
	\[
r_\alpha(\b):=\h\oplus \bigoplus_{\beta\in r_\alpha (R^+)}\g_\beta
	\]
is said to be the Borel subsuperalgebra of $\g$ obtained from $\b$ by the reflection $r_\alpha$.

Choose a Borel subalgebra $\b_0$ of $\g_0$ such that the set of simple roots associated to it is given by:
\begin{align*}
\{ \varepsilon_1-\varepsilon_2, \dotsc, \varepsilon_{n-1}-\varepsilon_n \},
\quad &\textup{if} \quad
\g \cong W(n), S(n), {\tilde S}(n) \\
\{ \varepsilon_1-\varepsilon_2, \dotsc, \varepsilon_{r-1}-\varepsilon_r, \varepsilon_{r-1}+\varepsilon_r \},
\quad &\textup{if} \quad
\g \cong H(2r), \\
\{ \varepsilon_1-\varepsilon_2, \dotsc, \varepsilon_{r-1}-\varepsilon_r,  \varepsilon_r \},
\quad &\textup{if} \quad
\g\cong H(2r+1).
\end{align*}
Let $R_0^+$ (resp. $R_0^-$) denote the set of positive (resp. negative) roots of $\g_0$ associated to the simple roots above.  The subsuperalgebra $\b_\max =\b_0 \oplus \left(\bigoplus_{i>1} \g_i\right)$ is known as the \emph{maximal Borel subsuperalgebra} of $\g$, and $\b_\min = \b_0 \oplus \g_{-1}$ is known as the \emph{minimal Borel subsuperalgebra} of $\g$.

\begin{prop} \label{prop:nice-system-exists-cartan}
If $\g$ is a Cartan type Lie superalgebra, then $\g$ admits at least one triangular decomposition satisfying \hyperref[C2]{$(\mathfrak C2)$}.
\end{prop}
\begin{proof}
We claim that there exists a triangular decomposition of $\g$ for which the highest root of $\g$ is a root of $\g_0$. Indeed,  since $\g$ is simple we have that $\g\cong L_{\b_\min}(\lambda)$ as a $\g$-module, where $\lambda$ is its highest root with respect to $\b_\min$. For every $\alpha\in R$, let $\h_\alpha:=[\g_\alpha, \g_{-\alpha}]$.  We now prove our claim case by case:

Let  $\g= W(n)$. Then  we have that $\lambda=-\varepsilon_n$ and the unique odd simple root for $\b_\min$ is $\alpha=-\varepsilon_1$. Set $\bar{\b}_1=r_\alpha(\b_\min)$. Since
	\[
\h_{-\varepsilon_i}=\{h\in \h\mid \varepsilon_i(h)=0\},\quad (\text{precisely } \h_{-\varepsilon_i}=\vspan_\C\{\xi_j\partial_j\mid j=1,\ldots, n,\ j\neq i\})
	\]
we have that $\varepsilon_n(\h_{-\varepsilon_1})\neq 0$ and hence $L_{\b_\min}(-\varepsilon_n)\cong L_{\bar{\b}_1}(-\varepsilon_n+\varepsilon_1)$ (see \cite[Lemma~5.3]{Ser05}). To conclude this case, we notice that $-\varepsilon_n$ is a root of $\g_{-1}$ and $\varepsilon_1$ is a root of $\g_1$, which implies that $(-\varepsilon_n+\varepsilon_1)$ is a root of $\g_0$, as we want.

The proof for  $W(n)$ also works for $S(n)$. The only difference   is that
	\[
\h_{-\varepsilon_i}=\{h\in \h\mid \varepsilon_i(h)=0,\ (\varepsilon_1+\cdots +\varepsilon_n)(h)=0\},
	\]
but it is still clear that $\varepsilon_n(\h_{-\varepsilon_1})\neq 0$ and hence $L_{\b_\min}(-\varepsilon_n)\cong L_{\bar{\b}_1}(-\varepsilon_n+\varepsilon_1)$.

Now, suppose that $\g=H(n)$. For $n=2k$, we have that $\lambda=\varepsilon_1-\delta$ and the unique odd simple root for $\b_\min$ is $\alpha_1=-\varepsilon_1-\delta$. Set $\bar{\b}_1=r_{\alpha_1}(\b_\min)$. Since
	\[
\h_{\varepsilon_i-\delta}=\h_{-\varepsilon_i-\delta}=\{h\in \h\mid \varepsilon_i(h)=\delta(h)=0\},
	\]
we have that $\lambda(\h_{-\varepsilon_1-\delta})=0$. Hence $L_{\b_\min}(\lambda)\cong L_{\bar{\b}_1}(\lambda)$. Now, $\alpha_2=-\varepsilon_2-\delta$ is an odd simple root for $\bar{\b}_1$. Since $\lambda(\h_{-\varepsilon_2-\delta})\neq 0$, we obtain that $L_{\b_\min}(\lambda)\cong L_{\bar{\b}_2}(\lambda-\alpha_2)$, where $\bar{\b}_2=r_{\alpha_2}(\bar{\b}_1)$. In particular, the highest root of $\g$ with respect to $\bar{\b}_2$ is $\varepsilon_1+\varepsilon_2$, which is clearly a root of $\g_0$. Now, we suppose that $n=2k+1$. Observing that $\h_{\pm\varepsilon_i-\delta}$ are the same as in the case $n=2k$, we have that the proof of the case $n=2k$ also works for $n=2k+1$. This proves the claim.

Notice that for $W(n)$, $S(n)$ and $H(n)$ we have found a triangular decomposition for which the highest root of $\g$ is in $R_0$. Namely, the triangular decomposition induced by $\bar{\b}_1$ (resp. $\bar{\b}_2$) when $\g$ is either $W(n)$ or $S(n)$ (resp. $H(n)$). Let $\b_1$ and $\b_2$ be the opposite Borel subalgebras of $\bar{\b}_1$ and $\bar{\b}_2$, respectively. It is clear that the triangular decomposition induced by $\b_1$ for $\g=W(n)$ and $S(n)$, and by $\b_2$ for $\g=H(n)$, also satisfy the condition of the claim. The set of negative roots with respect to $\b_1$ and $\b_2$ is as follows:
\begin{align*}
R_{\b_1}^- & =R_0^-\cup (R_{-1}\setminus \{-\varepsilon_1\})\cup\{\varepsilon_1\}, \text{ for } \g=W(n)\text{ or } S(n), \\
R_{\b_2}^- & =R_0^-\cup (R_{-1}\setminus \{-\varepsilon_1-\delta,\ -\varepsilon_2-\delta\})\cup\{\varepsilon_1+\delta,\ \varepsilon_2+\delta\}, \text{ for } \g=H(n).
\end{align*}
In particular, $\n^-_{\bar 0}\subseteq \lie r$ and hence such triangular decompositions satisfy \hyperref[C2]{$(\mathfrak C2)$}.

Finaly, consider a triangular decomposition of $S(n)$ satisfying \hyperref[C2]{$(\mathfrak C2)$}, and let $h\in \h_\R$ be a regular element that induces such a decomposition. This implies that $\theta(h)>\alpha(h)$ for any root $\alpha$ of $S(n)$. Since ${\tilde S}(n)_0=S(n)_0$, and every roots of ${\tilde S}(n)$ is also a root of $S(n)$, we have that $\theta(h)>\alpha(h)$ for any root $\alpha$ of ${\tilde S}(n)$. Thus the result follows.
\end{proof}

\subsection{Periplectic Lie superalgebras}

For each $n \ge 2$, let $\lie p(n)$ be the Lie subsuperalgebra of $\lie{gl}(n+1, n+1)$ consisting of all matrices of the form
\begin{equation}\label{eq:matrixper}
  M=\left(\begin{array}{r|r}
    A & B  \\
    \hline
    C & -A^t \\
  \end{array}\right),
  \quad
  \textup{where $A \in \sl (n+1)$, $B = B^t$ and $C^t = -C$}.
\end{equation}
Throughout this subsection,  $\g$ will denote $\lie p(n)$.  Notice that $\g_{\bar0}$ is isomorphic to the Lie algebra $\sl (n+1)$, and as a $\g_{\bar0}$-module, the structure of $\g_{\bar 1}$ is the following.  Let $S^2 (\C^{n+1})$ (resp. $\Lambda^2 (\C^{n+1})^*$) denote the second symmetric (resp. exterior) power of $\C^{n+1}$ (resp. $(\C^{n+1})^*$), with the natural action of $\sl (n+1)$ (by matrix multiplication) in each term.  Denote by $\g_{\bar 1}^+$ (resp. $\g_{\bar 1}^-$) the set of all matrices of the form \eqref{eq:matrixper} such that $A=C=0$ (resp. $A=B=0$), and notice that, as $\g_{\bar 0}$-modules, $\g_{\bar 1} \cong \g_{\bar 1}^+ \oplus \g_{\bar 1}^-$, where $\g_{\bar 1}^+ \cong S^2 (\C^{n+1})$ and $\g_{\bar 1}^- \cong \Lambda^2 (\C^{n+1})^*$.

Consider $\g_{-1} = \g_{\bar 1}^-$, $\g_0 = \g_{\bar 0}$ and $\g_1 = \g_{\bar 1}^+$. Then $\g= \g_{-1} \oplus \g_0 \oplus \g_1$ is a $\Z$-grading of $\g$ that is compatible with the $\Z_2$-grading ($\g_{\bar 0} = \g_0$ and $\g_{\bar 1} = \g_{-1} \oplus \g_1$).  Let $\lie h \subseteq \g_0$ be a Cartan subalgebra of $\g_0$, recall that $\g_0$ is isomorphic to $\lie{sl}(n+1)$, and identify $\lie h$ with $\h^*$ via the Killing form $(A_1, A_2) = {\rm tr} (A_1 A_2)$.  If $\{ \varepsilon_1, \dotsc, \varepsilon_n \}$ is the standard orthogonal basis of $\lie h$, then the roots of $\g$ are described as follows:
\begin{itemize}[leftmargin=*]
\item Roots of $\g_{-1}$: $-\varepsilon_i-\varepsilon_j$, where $1\leq i<j\leq n$.

\item Roots of $\g_0$: $\varepsilon_i-\varepsilon_j$, where $i\neq j$ and $1\leq i, j \leq n$.

\item Roots of $\g_1$: $\varepsilon_i+\varepsilon_j$, where $1\leq i\leq  j \leq n$.
\end{itemize}

Consider the triangular decomposition
\[
\n_0^- \oplus \h \oplus \n_0^+,
\quad \textup{where} \quad
\n_0^\pm = \bigoplus_{1 \le i < j \le n} \g_{\pm(\varepsilon_i-\varepsilon_j)}.
\]
This triangular decomposition induces a triangular decomposition on $\g$ and we have the following result.

\begin{prop} \label{prop:QGTD.p(n)}
If $\g$ is isomorphic to $\lie p(n)$ with $n \ge 2$, then the \emph{distinguished} triangular decomposition
\[
\nm \oplus \h \oplus \np,
\quad \textup{where} \quad
\n^\pm = \g_{\pm 1} \oplus \n_0^\pm
\]
satisfies \hyperref[C1]{$(\mathfrak C1)$}.  In particular, $\nm_{\bar0} = \nm_0$ and all the roots of $\g_1$ are positive.
\qed
\end{prop}

As usual, let $\b$ be the Borel subsuperalgebra $\h \oplus \np \subseteq \g$, $R$ be the set of roots of $\g$, $R_{\dis}^+$ be the positive roots associated to this decomposition, etc.  Notice that $R_{\dis}^- \neq -R_{\dis}^+$, since, for each $i \in \{ 1,\dotsc, n \}$, there exists a positive root of the form $2\varepsilon_i$, such that $-2\varepsilon_i\notin R$.

\begin{prop}\label{prop:nice-system-exists-p(n)}
If $\g$ is isomorphic to $\lie p(n)$ with $n \ge 2$, then $\g$ admits a triangular decomposition satisfying \hyperref[C2]{$(\mathfrak C2)$}.
\end{prop}
\begin{proof}
Recall that $\g_1\cong S^2(\C^{n+1})$ as a $\g_0$-module. In particular, $2\varepsilon_{n+1}$ (resp. $2\varepsilon_1$) is the lowest (resp. highest) weight of $\g_1$ as a $\g_0$-module. This implies that $2\varepsilon_1$ is the highest weight of $\g$ with respect to the distinguished triangular decomposition given in Proposition~\ref{prop:QGTD.p(n)}, and $2\varepsilon_{n+1}$ is a simple root of $R_{\dis}^+$. Now, notice that the chain of reflections (as defined in \eqref{eq:reflec.Cartan.type})
	\[
R_{\dis}^+=R_0^+ \overset{r_{2\varepsilon_{n+1}}}{\rightarrow} R_1^+ \overset{r_{\varepsilon_{n}+\varepsilon_{n+1}}}{\rightarrow} R_2^+ \overset{r_{\varepsilon_{n-1}+\varepsilon_{n+1}}}{\rightarrow} \cdots \overset{r_{\varepsilon_{2}+\varepsilon_{n+1}}}{\rightarrow} R_n^+ \overset{r_{\varepsilon_{1}+\varepsilon_{n+1}}}{\rightarrow}R_{n+1}^+.
	\]
is well defined, since $2\varepsilon_{n+1}$ is simple in $R_{\dis}^+$, and each $\varepsilon_k+\varepsilon_{n+1}$ is simple in $R_{n+1-k}^+$.  Denote $\h_\alpha=[\g_\alpha, \g_{-\alpha}]$ for all $\alpha\in R$.  Since $2\varepsilon_1(\h_{2\varepsilon_{n+1}})=0$, and  $2\varepsilon_1(\h_{\varepsilon_k+\varepsilon_{n+1}})=0$, for all $k>1$, it follows from \cite[Lemma~5.2]{Ser05}, that $2\epsilon_1$, the highest root of $\g$, is invariant under $r_{2\epsilon_{n+1}}, r_{\epsilon_n + \epsilon_{n+1}}, \dotsc, r_{\epsilon_2 + \epsilon_{n+1}}$. Since $2\varepsilon_1(\h_{\varepsilon_1+\varepsilon_{n+1}})\neq 0$, the highest root of $\g$ with respect to $R_{n+1}^+$ is given by $2\varepsilon_1-(\varepsilon_1+\varepsilon_{n+1})=\varepsilon_1-\varepsilon_{n+1}$, which is a root of $\g_0$.
\end{proof}

\subsection{Remarks on triangular decompositions}

Recall conditions \hyperref[C1]{$(\mathfrak C1)$} and \hyperref[C2]{$(\mathfrak C2)$}.  The next result sums up the results that we have obtained regarding triangular decompositions satisfying these conditions.

\begin{thm}\label{thm:gtd-exist}
If $\g$ is either $\sl (n,n)$ with $n \ge 2$, or a finite-dimensional simple Lie superalgebra, then $\g$ admits at least one triangular decomposition satisfying \hyperref[C1]{$(\mathfrak C1)$} and \hyperref[C2]{$(\mathfrak C2)$}.
\end{thm}

\begin{proof}
If $\g\ncong\q(n)$, then the result follows from Propositions \ref{prop:nice-system-exists-basic}, \ref{prop:nice-system-exists-cartan}, \ref{prop:QGTD.p(n)}, and \ref{prop:nice-system-exists-p(n)}. If $\g \cong \q(n)$, then every triangular decomposition satisfies \hyperref[C1]{$(\mathfrak C1)$} and \hyperref[C2]{$(\mathfrak C2)$}, as $\lie r = \q(n)_{\bar 0}$ and every root of $\q(n)$ is both even and odd.
\end{proof}

\begin{rmk} \label{rmk:GTDs}
Conditions \hyperref[C1]{$(\mathfrak C1)$} and \hyperref[C2]{$(\mathfrak C2)$} are crucial for the rest of this paper.  Namely, \hyperref[C1]{$(\mathfrak C1)$} is used to prove finite-dimensionality of generalized Kac modules for Cartan type Lie superalgebras (see Proposition~\ref{prop:Vbar-fd}), and \hyperref[C2]{$(\mathfrak C2)$} is used to prove that the global Weyl module $W_A(\lambda)$ is a finitely-generated $\fA_\lambda$-module (see Theorem~\ref{fin.dim.odd}).

Observe that if $\g$ is either basic classical or strange, then every triangular decomposition satisfies \hyperref[C1]{$(\mathfrak C1)$}, since $\lie r=\g_{\bar 0}$.  Moreover, if $\g$ is of type II or $\lie q(n)$, then every distinguished triangular decomposition also satisfies \hyperref[C2]{$(\mathfrak C2)$}.  However, if $\g$ is either of type I or $\lie p(n)$, then its distinguished triangular decomposition satisfies \hyperref[C1]{$(\mathfrak C1)$} but not \hyperref[C2]{$(\mathfrak C2)$}, as the lowest root of $\g$ is a root of $\g_{-1}$.

If $\g$ is of Cartan type, then there are triangular decompositions that do not satisfy \hyperref[C1]{$(\mathfrak C1)$}. For instance, the triangular decomposition induced by $\b_\min$, as ${\lie r} \subsetneq \g_{\bar 0}$. Moreover, there are triangular decompositions satisfying \hyperref[C1]{$(\mathfrak C1)$} but not \hyperref[C2]{$(\mathfrak C2)$}. For instance, the triangular decomposition induced by $\b_{\max}$, since the lowest root of $\g$ with respect this triangular decomposition is a root of $\g_{-1}\subseteq \g_{\bar 1}$.

If $\g$ is of type $\q(n)$, then $\h$ is not contained in $\g_{\bar0}$.  Thus, the highest-weight space of an irreducible finite-dimensional $\g$-module is not always one dimensional.  (It is an irreducible module for a certain Clifford superalgebra.)  As it requires a different treatment, this case will be considered in a future work.
\end{rmk}

%
\section{Generalized Kac modules}  \label{S:gen.kac.mod}
%

From now on we assume that $\g$ is either $\sl (n,n)$ with $n \ge 2$, or a finite-dimensional simple Lie superalgebra not of type $\q(n)$.  Let $\h$ be a fixed Cartan subalgebra of $\g$ (in particular, since $\g \not\cong \q(n)$, $\h = \h_{\bar0}$ is a Cartan subalgebra of $\lie r$), $R$ be the set of roots of $\g$ (with respect to $\h$), and $Q \subseteq \h^*$ the root lattice $\sum_{\alpha\in R} \Z \alpha$.  Fix a set $R^+ \subseteq R$ of positive roots, $Q^+$ be the positive cone $\sum_{\alpha \in R^+} \N \alpha$ and $\g = \n^- \oplus \h \oplus \n^+$ be the associated triangular decomposition of $\g$.  Notice that $\b = \h \oplus \n^+$ is a solvable subsuperalgebra of $\g$, that $\n^\pm$ are nilpotent ideals of $\b$, and that $\h^*$ admits a partial order given by: $\lambda \le \mu \in \h^*$ if and only if $\mu - \lambda \in Q^+$.

Let $R_{\lie r}$ denote the root system $\{ \alpha \in \h^*\mid \lie{r}_\alpha \ne \{0\}, \ \alpha \ne 0 \}$ of $\lie r$, $R^+_\lie r$ be the positive system $R^+\cap R_{\lie r}$, and $\dr$ be the simple system of $R_{\lie r}$ associated to $R_{\lie r}^+$.  Since $\lie r'$ is a finite-dimensional semisimple Lie algebra, for each $\alpha \in R^+_\lie r$, one can choose elements $x_\alpha \in \lie r_{\alpha}$, $x_\alpha^{-} \in \lie r_{-\alpha}$, and $h_\alpha \in \h$, such that: the subalgebra generated by $\{ x^-_\alpha, h_\alpha, x_\alpha \}$ is isomorphic to $\sl(2)$, $[x_\alpha, x^-_\alpha] = h_\alpha$, $[h_\alpha, x^-_\alpha] = -2 x_\alpha$, and $[h_\alpha, x_\alpha] = 2 x_\alpha$.  The triple $\left( x_\alpha, x_\alpha^{-}, h_\alpha \right)$ is said to be an \emph{$\sl(2)$-triple} and the Lie subalgebra generated by $\{ x_\alpha^{-}, h_\alpha, x_\alpha \}$ will be denoted $\sl_\alpha$.  Recall that
\[
X^+ = X^+(\b)= \{ \lambda \in \h^* \mid L_\b (\lambda) \textup{ is finite dimensional} \},
\]
and notice that, for $\lambda \in X^+$, we have $\lambda(h_\alpha)\in \N$, for all $\alpha \in \dr$ (since $L_\b(\lambda)$ is also a finite-dimensional $\lie r'$-module).

\begin{defn}[Generalized Kac module] \label{def-kac-mod}
Let $\lambda \in X^+$.  The \emph{generalized Kac module} associated to $\lambda$ is defined to be the cyclic $\g$-module $K(\lambda)=K_{\b}(\lambda)$ given as a quotient of $\U\g$ by the left ideal generated by
\[
\lie n^+, \quad
h -\lambda(h), \quad
(x_\alpha^{-})^{\lambda(h_\alpha)+1},\quad
\textup{for all $h \in \h$ and $\alpha \in \dr$}.
\]
\end{defn}

Denote the image of $1 \in \U\g$ in $K(\lambda)$ by $k_\lambda$, and notice that as a $\g$-module, $K(\lambda)$ is generated by the vector $k_\lambda\in K(\lambda)_{\bar0}$, satisfying the following defining relations
\begin{equation} \label{kac-mod-relations}
\lie n^+ k_\lambda = 0, \quad
h k_\lambda = \lambda(h) k_\lambda, \quad
(x_\alpha^{-})^{\lambda(h_\alpha)+1} k_\lambda = 0,\quad
\textup{for all $h \in \h$ and $\alpha \in \dr$}.
\end{equation}

\begin{defn}[Parabolic triangular decomposition] \label{defn.parabolic.td}
A triangular decomposition $\g = \nm \oplus \h \oplus \np$ is said to be parabolic if $(\lir + \np)$ is a Lie subsuperalgebra of $\g$ and there exists a nontrivial subspace $\g^- \subseteq \g$ such that
\begin{equation}\label{eq:dis.td}
\g = \g^-\oplus (\lir + \np).
\end{equation}
\end{defn}

\begin{rmk}
When $\g$ is basic classical of type I, isomorphic to $\sl(n, n)$ with $n \ge 2$, $\lie p(n)$ or $\q(n)$, distinguished triangular decompositions are parabolic.  When $\g$ is basic classical of type II, it does not admit any parabolic triangular decomposition, since $\g_{\bar 1}$ is an irreducible $\g_{\bar 0}$-module.  When $\g$ is of Cartan type, minimal ($\n^\pm = \n_0^\pm  \oplus \left( \oplus_{i > 0} \g_{\pm i} \right)$) and maximal ($\n^\pm = \n_0^\pm  \oplus \left( \oplus_{i < 0} \g_{\pm i} \right)$) triangular decompositions are parabolic.
\end{rmk}

If $\g$ is a basic classical Lie superalgebra or $\g \cong \sl (n,n)$, $n\geq 2$, then $K(\lambda)$ coincides with the generalized Kac module defined in \cite[Definition~2.6]{CLS}.  The next result gives necessary and sufficient conditions for generalized Kac modules to be finite dimensional.

\begin{prop} \label{prop:Vbar-fd}
Let $\lambda \in X^+$.
\begin{enumerate}[leftmargin=*]
\item \label{kac.fd.clscl}
If $\g$ is basic classical, isomorphic to $\sl(n,n)$ with $n \ge 2$, or $\lie{p}(n)$, then $K(\lambda)$ is finite dimensional (for every triangular decomposition).

\item \label{kac.fd.cartan}
If $\g$ is of Cartan type and the triangular decomposition $\g = \nm \oplus \h \oplus \np$ satisfies \hyperref[C1]{$(\mathfrak C1)$}, then $K(\lambda)$ is finite dimensional.

\item \label{kac.inf.dim.cartan}
If $\g$ is of Cartan type, the triangular decomposition $\g = \nm \oplus \h \oplus \np$ does not satisfy \hyperref[C1]{$(\mathfrak C1)$} and is parabolic, then $K(\lambda)$ is infinite dimensional.
\end{enumerate}
\end{prop}

\begin{proof}
\eqref{kac.fd.clscl} In these cases, the proof is similar to that of \cite[Proposition~2.7]{CLS}.

\eqref{kac.fd.cartan} Since $\lambda \in X^+$ (thus $\lambda(h_\alpha) \in \N$ for all $\alpha \in \dr$), we can consider the finite-dimensional irreducible $\lie r$-module of highest weight $\lambda$, $L_0 (\lambda)$.  Recall that $\lie r$ is a reductive Lie algebra and $\lie z$ acts as a scalar on $L_0(\lambda)$.  Hence $L_0(\lambda)$ is isomorphic to the $\lie r$-module generated by a vector $u_\lambda$ with defining relations
\[
x_\alpha u_\lambda = 0, \quad
h u_\lambda = \lambda(h) u_\lambda, \quad
(x_\alpha^{-})^{\lambda(h_\alpha)+1} u_\lambda=0, \quad
\textup{for all $h \in \h$ and $\alpha \in \dr$}.
\]
Let $W = \U{\lie r} k_\lambda$ be the $\lie r$-submodule of $K(\lambda)$ generated by $k_\lambda$.  Since $W$ is cyclic and $k_\lambda$ satisfies \eqref{kac-mod-relations}, there exists a unique (surjective) homomorphism of $\lie r$-modules satisfying
\[
\varphi \colon L_0(\lambda) \twoheadrightarrow W,\quad
u_\lambda \mapsto k_\lambda.
  \]
Since $\varphi$ is surjective and $L_0(\lambda)$ is finite dimensional, this shows that $W$ is finite dimensional.

If the triangular decomposition $\g = \nm \oplus \h \oplus \np$ satisfies \hyperref[C1]{$(\mathfrak C1)$}, that is, $\n^-_{\bar 0} \subseteq \lie{r}$, then by the PBW Theorem, for any given basis $\{x_i \mid 1\leq i\leq \dim \n_{\bar 1}^-\}$ of $\n_{\bar 1}^-$, we have that
\[
K(\lambda)
= \bu(\g) k_\lambda = \vspan_{\Bbbk}\{ x_{j_1} \dotsm x_{j_k} W
\mid 1 \leq j_1 < \dotsb < j_k \leq \dim \n_{\bar1}^- \}.
\]
Since $W$ is finite dimensional, we conclude that $K(\lambda)$ is also finite dimensional.

\eqref{kac.inf.dim.cartan} If the triangular decomposition $\g = \nm \oplus \h \oplus \np$ is parabolic, $(\lir + \np)$ is a Lie subsuperalgebra of $\g$ and there exists a nontrivial subspace $\g^- \subseteq \g$ such that $\g = \g^-\oplus (\lie{r} + \np)$.  Now, consider the $(\lie{r}+\np)$-module $K^{\lir}(\lambda)$ given as the quotient of $\U{\lir + \np}$ by the left ideal generated by
\[
\g_\alpha \quad \textup{for all } \alpha\in R^+, \qquad
h - \lambda(h) \quad \textup{for all } h \in \h, \qquad
(x_\alpha^{-})^{\lambda(h_\alpha)+1} \quad \textup{for all } \alpha \in \dr.
\]
Notice that the image of $1 \in \U{\lir + \np}$, which we will denote by $u_\lambda$, generates $K^{\lir}(\lambda)$.

Now, let
	\[
\overline{K^\lir}(\lambda) = \ind{\lie{r} + \np}{\g} K^\lir (\lambda),
	\]
and notice that $\overline{K^\lir}(\lambda)$ is generated by $1 \otimes u_\lambda$. If the triangular decomposition $\g = \nm \oplus \h \oplus \np$ does not satisfy \hyperref[C1]{$(\mathfrak C1)$}, then $\g^- \cap \g_{\bar0} \ne 0$.  In particular,  $\overline{K^\lir}(\lambda)\cong \bigoplus m \otimes K^\lir (\lambda)$, where $m$ runs over the set of ordered monomials where the variables form a basis of $\g^-$. Since $\g^- \cap \g_{\bar0} \ne 0$, $\overline{K^\lir}(\lambda)$ is infinite dimensional.  Finally, notice that there exists a unique surjective homomorphism of $\g$-modules $K(\lambda) \to \overline{K^\lir}(\lambda)$ satisfying $k_\lambda \mapsto (1 \otimes u_\lambda)$.  Since $\overline{K^\lir}(\lambda)$ is infinite dimensional, we conclude that $K(\lambda)$ is also infinite dimensional.
\end{proof}

\begin{ex}
If $\g$ is of Cartan type, then the minimal triangular decomposition, that is, the one induced by $\np = \n_0^+ \oplus \g_{-1}$ is parabolic and does not satisfy \hyperref[C1]{$(\mathfrak C1)$}.  In fact, in this case, $\nm = \n_0^- \oplus \left( \oplus_{k \ge 1} \g_k \right)$ and $\nm \cap \g_{2k} = \g_{2k} \not\subseteq \lir$ for all $k \ge 1$.  Thus generalized Kac modules associated to minimal triangular decompositions are infinite dimensional.
\end{ex}

\begin{ex}
If $\g$ is of Cartan type, then the maximal triangular decomposition, that is, the one induced by $\n^+ = \n_0^+ \oplus \left( \oplus_{k \ge 1} \g_k \right)$ satisfies \hyperref[C1]{$(\mathfrak C1)$} and is parabolic. Moreover, one can check that triangular decompositions induced by the Borel subalgebras $\b_1$ and $\b_2$ constructed in the proof of Proposition~\ref{prop:nice-system-exists-cartan} satisfy \hyperref[C1]{$(\mathfrak C1)$} but are not parabolic (in fact, $\lie{r}+\n^+$ is not a subsuperalgebra of $\g$). In particular, generalized Kac modules associated to these triangular decompositions are finite-dimensional.
\end{ex}

\begin{ex}
The authors do not know yet any example of a triangular decomposition of a Lie superalgebra of Cartan type that does not satisfy \hyperref[C1]{$(\mathfrak C1)$} and is not parabolic.
\end{ex}

The next result generalizes \cite[Lemma~2.8]{CLS} and their proofs are similar.

\begin{prop} \label{prop:univ-prop-kac-mod}
Let $\lambda\in X^+$.  If $V$ is a finite-dimensional $\g$-module generated by a highest-weight vector of weight $\lambda$, then there exists a surjective homomorphism of $\g$-modules $\pi_V \colon K(\lambda) \twoheadrightarrow V$.  Moreover, there exists a unique $\g$-submodule $W \subseteq K(\lambda)$ such that $V \cong K(\lambda) / W$.
\qed
\end{prop}

\dproof{
If $V = \{0\}$, the result is obvious.  So assume $V \ne \{0\}$, and let $v \in V \setminus \{0\}$ be a highest-weight vector that generates $V$.  By the definition of a highest-weight vector, the first two relations in \eqref{kac-mod-relations} are satisfied by $v$.  Since $\g_0$ is a finite-dimensional reductive Lie algebra, $V$ is finite dimensional and $v$ satisfies the first two relations in \eqref{kac-mod-relations}, $v$ must also satisfy the last relation in \eqref{kac-mod-relations}.  Thus there exists a unique homomorphism of $\g$-modules $\pi_V \colon K(\lambda) \to V$ satisfying $\pi_V (k_\lambda) = v$.  Since $v$ generates $V$, $\pi_V$ is also surjective.

Moreover, since $\pi_V$ is a homomorphism of $\g$-modules, it preserves weight spaces.  Since $V$ and $K(\lambda)$ are highest-weight modules, $\dim V_\lambda = \dim K(\lambda)_\lambda = 1$ and $\pi_V$ is unique up to scalar multiple.  Thus the $\g$-submodule $W = \ker \pi_V$ is unique.
}

Since every irreducible finite-dimensional $\g$-module is generated by a highest-weight vector of weight $\lambda\in X^+$,  Proposition~\ref{prop:univ-prop-kac-mod} applies, in particular, to all irreducible finite-dimensional $\g$-modules.

%
\section{Global Weyl modules} \label{S:glo.weyl.mod}
%

Let $\g$ be a Lie superalgebra and consider an associative commutative $\C$-algebra $A$ with unit. The vector space $\g \otimes A$ is a Lie superalgebra when endowed with the $\Z_2$-grading given by $(\g\otimes A)_j=\g_j\otimes A$, $j\in\Z_{2}$, and the Lie superbracket extending bilinearly
	\[
[x_{1}\otimes a_{1},x_{2}\otimes a_{2}] = [x_{1},x_{2}]\otimes a_{1}a_{2},
\quad \textup{for all $x_1, x_2 \in \g$ and $a_1, a_2 \in A$}.
	\]
We refer to a Lie superalgebra of this form as a \emph{map Lie superalgebra}.  From now on, we identify $\g$ with a subsuperalgebra of $\g \otimes A$ via the isomorphism $\g \cong \g \otimes \C$ and the inclusion $\g \otimes \C \subseteq \g \otimes A$.

From now on, let $\g$ be either $\sl (n,n)$ with $n \ge 2$, or a finite-dimensional simple Lie superalgebra not of type $\lie q(n)$.  Let $\mathcal{C}_A$ denote the category of $\ga$-modules that are finitely semisimple as $\lie r$-modules (see \eqref{eq:reductive.part} for the notation).  By Lemma~\ref{lem:I(a,b)closed}, $\cal C_A = \cal C_{(\ga, \lie r)}$ is an abelian category, closed under taking submodules, quotients, arbitrary direct sums, and finite tensor products.

\begin{lemma}\label{lem-projective}
If $V$ is a finitely-semisimple $\lie r$-module, then $\ind{\lie r}{\g \otimes A} V$ is a projective object in $\mathcal{C}_A$.  Moreover, the category $\cal C_A$ has enough projectives.
\end{lemma}

\begin{proof}
Recall that $\g$ is a finitely-semisimple $\lie r$-module via the adjoint representation.  Since $\ga \cong \g^{\oplus \dim A}$ as $\lie r$-modules, Lemma~\ref{lem:I(a,b)closed} implies that $\ga$ is a finitely-semisimple $\lie r$-module.  Thus, by Lemma~\ref{lem:ind.I(a,b)}, $\ind{\lie r}{\ga} V$ is an object in $\cal C_A$.

To prove that $\ind{\lie r}{\ga} V$ is projective, first recall that an object $M$ of $\cal C_A$ is projective if and only if $\Hom_{\cal C_A} (M, -)$ is an exact functor.  By definition, $\Hom_{\cal C_A} (M, N) = \Hom_{\U \ga} (M, N)$ for all $M, N$ in $\cal C_A$.  By Frobenius Reciprocity, the functor $\Hom_{\U \ga} (\ind{\lie r}{\ga}V, -)$ is naturally isomorphic to $\Hom_{\U{\lie r}} (V, -)$.  Since every object of $\cal C_A$ is assumed to be a direct sum of its finite-dimensional $\lie r$-submodules, that is, every object of $\mathcal C_A$ is completely reducible as an $\lie r$-module, the restriction of $\Hom_{\U{\lie r}} (V, -)$ to $\cal C_A$ is exact.  Thus $\Hom_{\cal C_A} (\ind{\lie r}{\ga}V, -)$ is an exact functor.

Moreover, let $M$ be any object of $\mathcal C_A$.  Since $M$ is assumed to be a finitely-semisimple $\lie r$-module, by the first part of this lemma, $\ind{\lie r}{\g \otimes A} M$ is projective.  Since the unique linear transformation $f \colon \ind{\lie r}{\ga} M \to M$ satisfying $f (u \otimes m) = u m$ for all $u \in \U \ga$ and $m \in M$ is a surjective homomorphism of $\ga$-modules, we conclude that $\cal C_A$ has enough projectives.
\end{proof}

Given a $\g$-module $V$, define $P_A(V)$ to be the $\ga$-module
\begin{equation}\label{cha-fou-equ.1}
P_A(V) = \ind{\g}{\ga} V.
\end{equation}
Notice that, if $V$ is a projective $\lie g$-module, then $P_A(V)$ is a projective $\ga$-module.

The next result, which was proved in \cite[Proposition~3.2]{CLS} for the cases where $\g$ is either basic classical or $\sl (n,n)$ with $n \ge 2$, describes $P_A(K(\lambda))$ by generators and relations.

\begin{prop}\label{cha-fou-pro3}
If $\lambda\in X^+$, then $P_A(K(\lambda))$ is generated as a left $\U \ga$-module, by a vector $p_\lambda\in P_A(K(\lambda))_{\bar0}$ satisfying the following defining relations
\begin{equation}\label{cha-fou-equ.2}
\n^+ p_\lambda=0, \quad
h p_\lambda=\lambda(h) p_\lambda, \quad
(x_\alpha^{-})^{\lambda(h_\alpha)+1} p_\lambda=0, \quad
\text{for all $h \in \h$ and $\alpha \in \dr$}.
\end{equation}
\end{prop}

\begin{proof}
Let $p_\lambda = 1 \otimes k_\lambda \in P_A(K(\lambda))$.  Since $k_\lambda \in K(\lambda)_{\bar0}$ satisfies relations \eqref{kac-mod-relations}, $p_\lambda\in P_A(K(\lambda))_{\bar0}$ satisfies relations \eqref{cha-fou-equ.2}.  The fact that these are defining relations follows from Lemma~\ref{lem:ind.cyclic=cyclic}.
\end{proof}

Given $\lambda \in X^+$ and $M\in {\cal C}_A$, define
\begin{equation} \label{defn:^lambda}
M^\lambda = M \Big/ \sum_{\mu \not\le \lambda} \U\ga M_\mu.
\end{equation}
Notice that, if $\mu$ is a weight of $M^\lambda$, then $\mu \le \lambda$.  Let ${\cal C}_A^\lambda$ denote the full subcategory of ${\cal C}_A$ whose objects are the left $\U \ga$-modules $M \in \cal C_A$ such that $M^\lambda = M$.  (Notice that $\cal C_A^\lambda$ depends on the choice of triagular decomposition $\g = \nm \oplus \h \oplus \np$, even though it is not shown explicitly in its notation.)  The proof of the next result is similar to that of Lemma~\ref{lem-projective}.

\begin{lemma} \label{weyl.projective}
Let $\lambda\in X^+$ and $V$ be a $\g$-module.  If $V$ is finitely semisimple as an $\lie r$-module, then $P_A(V)^\lambda$ is a projective object in $\mathcal{C}_A^\lambda$.  Moreover, the category $\mathcal{C}_A^\lambda$ has enough projectives.
\qed
\end{lemma}

\dproof{
First recall from the construction of $\cal C_A^\lambda$, that
	\[
\Hom_{\cal C_A^\lambda}(P_A(V)^\lambda, M) = \Hom_{\cal C_A} (P_A(V)^\lambda, M)
\quad \textup{for every object $M$ of $\cal C_A^\lambda$}.
	\]
Now, since morphisms in $\cal C_A$ are homomorphisms of $\h$-modules and $\ga$-modules, precomposition with the quotient \eqref{defn:^lambda} gives an isomorphism of vector spaces
\[
\Hom_{\cal C_A} (P_A(V)^\lambda, M) \cong \Hom_{\cal C_A} (P_A(V), M)
\quad \textup{for every object $M$ of $\cal C_A^\lambda$},
\]
which is functorial in $M$.  Finally, recall from Lemma~\ref{lem-projective} that $P_A(V)$ is a projective object of $\cal C_A$.  This shows that $P_A(V)^\lambda$ is a projective object in $\mathcal{C}_A^\lambda$.

Let $M$ be any object of $\mathcal C_A^\lambda$.  Since $M$ is assumed to be a finitely-semisimple $\mathfrak s$-module, by the first part of this lemma, $P_A(M)^\lambda$ is projective and $\Hom_{\cal C_A^\lambda} (P_A(M)^\lambda, M) \cong \Hom_{\cal C_A} (P_A(M), M)$ via precomposition by the (surjective) quotient \eqref{defn:^lambda}.  Thus, there exists a unique surjective homomorphism of $\mathfrak g \otimes A$-modules $P_A(M)^\lambda \to M$ corresponding to the surjective homomorphisms of $\mathfrak g \otimes A$-modules $f \colon P_A(M) \to M$ satisfying $f (u \otimes m) = u m$ for all $u \in \U \ga$ and $m \in M$.
}

\begin{defn}[Global Weyl module] \label{def:global-Weyl}
Let $\lambda \in X^+$.  The \emph{global Weyl module} associated to $\lambda$ is defined to be
\[
W_A(\lambda):=P_A(K(\lambda))^\lambda.
\]
The image of $p_\lambda$ in $W_A(\lambda)$ will be denoted by $w_\lambda$.
\end{defn}

The next result provides a descriptions of global Weyl modules by generators and relations, and as a universal object in $\cal C_A^\lambda$.  Its proof is similar to that of  \cite[Proposition~4]{CFK10}.

\begin{prop} \label{prop:glob.Weyl.genrel+univ}
  For $\lambda\in X^+$, the global Weyl module $W_A(\lambda)$ is generated as a left $\U \ga$-module, by the vector $w_\lambda$, with defining relations
  \begin{equation}\label{cha-fou-equ5}
    (\n^+\otimes A)w_\lambda=0,\quad hw_\lambda=\lambda(h)w_\lambda,\quad (x_\alpha^-)^{\lambda(h_\alpha)+1}w_\lambda=0,\quad \text{for all $h \in \h$ and $\alpha \in \dr$}.
  \end{equation}
Moreover, if the triangular decomposition of $\g$ satisfies \hyperref[C1]{$(\mathfrak C1)$}, then the global Weyl module $W_A(\lambda)$ is the unique object of $\mathcal{C}_A^\lambda$, up to isomorphism, that is generated by a highest-weight vector of weight $\lambda$ and admits a surjective homomorphism onto every object of $\mathcal{C}_A^\lambda$ that is generated by a highest-weight vector of weight $\lambda$.
\end{prop}

\dproof{
Recall that, if $\mu$ is a weight of $W_A(\lambda)$, then $\mu \le \lambda$.  Thus $(\n^+\otimes A)w_\lambda=0$.  The remaining relations are satisfied by $w_\lambda$, since they are satisfied by $p_\lambda$ (see Proposition~\ref{cha-fou-pro3}).  Let $W$ be the module generated by an element $w$ satisfying relations \eqref{cha-fou-equ5}, so that we have an epimorphism $\pi_1 \colon W \twoheadrightarrow W_A(\lambda)$ sending $w$ to $w_\lambda$.  Since the relations \eqref{cha-fou-equ5} imply the relations \eqref{kac-mod-relations}, the vector $w \in W$ generates a $\g$-submodule of $W$ isomorphic to a quotient of $K(\lambda)$.  Thus we have a surjective homomorphism
  \[
    \pi_2 \colon P_A(K(\lambda)) \to W,\quad u_1 \otimes u_2 v_\lambda \mapsto u_1 u_2 w,\quad u_1 \in \bu(\g \otimes A),\ u_2 \in \bu(\g).
  \]
  Since the $\g$-weights of $W$ are bounded above by $\lambda$, it follows that $\pi_2$ induces a map $W_A(\lambda) \to W$ inverse to $\pi_1$.

The \emph{moreover} part was proved in \cite[Proposition~3.5]{CLS} for the case where $\g$ is either basic classical or $\sl (n,n)$ with $n \ge 2$.  The  proofs for the other types are similar.  The fact that $W_A(\lambda)$ is an object of $\cal C_A^\lambda$ generated by a highest-weight vector of weight $\lambda$ follows directly from \eqref{cha-fou-equ5}.  Now, let $V$ be an object of $\mathcal{C}_A^\lambda$ generated by a highest-weight vector $v$ of weight $\lambda$.  Then
\[
(\n^+\otimes A)v=0
\quad \textup{and} \quad
hv=\lambda(h)v \quad \textup{for all $h\in\h$}.
\]
Since $V$ is an object of $\cal C_A^\lambda$, the vector $v$ belongs to a finite direct sum of finite-dimensional irreducible $\lie r$-submodules of $V$.  Thus the $\lie r$-submodule of $V$ generated by $v$ is finite dimensional. Since $\lie r'$ (the semisimple part of $\lie r$) is a finite-dimensional semisimple Lie algebra, $(x_\alpha^-)^{\lambda(h_\alpha)+1}v=0$ for all $\alpha \in \dr$.  Hence $v$ satisfies relations \eqref{cha-fou-equ5}, and one can define a unique surjective homomorphism of $\ga$-modules $W_A(\lambda) \twoheadrightarrow V$ satisfying $w_\lambda \mapsto v$.

Now, in order to prove the uniqueness of $W_A(\lambda)$ up to isomorphism, let $W$ be an object of $\mathcal{C}_A^\lambda$ that is generated by a highest-weight vector $w$ of weight $\lambda$ and that admits a surjective homomorphism onto every object of $\mathcal{C}_A^\lambda$ that is generated by a highest-weight vector of weight $\lambda$.  In particular, we have a surjective homomorphism $\pi_1 \colon W \twoheadrightarrow W_A(\lambda)$.  It follows from the PBW Theorem and the first two relations in \eqref{cha-fou-equ5} that there is an isomorphism of vector spaces $W_A(\lambda)_\lambda \cong \bu(\h \otimes A_+) \otimes \C w_\lambda$.  Since $\h$ is abelian, $\U{\h \otimes A_+}$ is a free commutative algebra.  Thus, all the vectors in $W_A(\lambda)_\lambda$ that generate $W_A(\lambda)$ as a $\ga$-module are in $\C w_\lambda$.  This shows that $\pi_1(w) = c w_\lambda$ for some $c \in \C \setminus \{0\}$.

To finish the proof, first notice that, as above, $w$ satisfies relations \eqref{cha-fou-equ5}.  Thus, since $c \ne 0$, there exists a unique homomorphism of $\ga$-modules $\pi_2 \colon W_A(\lambda) \to W$ satisfying $\pi_2(w_\lambda) = \frac 1 c w$.  By construction, $\pi_1$ and $\pi_2$ are mutually inverse.  Thus $W \cong W_A(\lambda)$.
}

When $A = \C$, the global Weyl module $W_A(\lambda)$ coincides with the generalized Kac module $ K(\lambda)$.  In this case, the \emph{moreover} part of Proposition~\ref{prop:glob.Weyl.genrel+univ} reduces to the universal property given in Proposition~\ref{prop:univ-prop-kac-mod}.

%
\section{Weyl functors} \label{S:super.weyl.functors}
%

Let $A$ be an associative commutative $\C$-algebra with unit, and $\g$ be either $\sl(n,n)$ with $n\geq 2$, or a finite-dimensional simple Lie superalgebra not of type $\lie q(n)$, endowed with a trinagular decomposition satisfying \hyperref[C1]{$(\mathfrak C1)$}.

Let $\lambda \in X^+$.  Recall from Definition~\ref{def:global-Weyl}, that $w_\lambda$ denotes the image of $p_\lambda$ in $W_A(\lambda)$, and set
\begin{gather*}
\Ann_{\g\otimes A}(w_\lambda)=\{u\in \bu(\g\otimes A)\,|\,u w_\lambda=0\}, \\
\Ann_{\h\otimes A}(w_\lambda)=\Ann_{\g\otimes A}(w_\lambda)\cap \bu(\h\otimes A).
\end{gather*}
Notice that $\Ann_{\ga} (w_\lambda)$ is a left ideal of $\U\ga$, and thus, since $\U{\h \otimes A}$ is a commutative algebra, $\Ann_{\h \otimes A}(w_\lambda)$ is an ideal of $\U{\h \otimes A}$.  Define the algebra $\fA_\lambda$ to be the quotient
\[
\fA_\lambda = \bu(\h\otimes A)/\Ann_{\h\otimes A} (w_\lambda).
\]

By Proposition~\ref{prop:glob.Weyl.genrel+univ} and the PBW Theorem, $W_A(\lambda)_\lambda = \U{\h \otimes A}w_\lambda$.  Thus, the unique homomorphism of $\U{\h \otimes A}$-modules satisfying
\[
\phi \colon \U{\h \otimes A} \to W_A(\lambda)_\lambda, \quad
\phi (1) = w_\lambda
\]
induces an isomorphism of $\h \otimes A$-modules between $W_A(\lambda)_\lambda$ and $\U{\h \otimes A} / \Ann_{\h \otimes A} (w_\lambda)$. In other words, $W_A(\lambda)_\lambda\cong\fA_\lambda$ as right $\fA_\lambda$-modules.

\begin{lemma}\label{Ann-1}
For all $\lambda\in X^+$ and $V\in {\cal C}_A^\lambda$, we have $(\Ann_{\h\otimes A}(w_\lambda)) V_\lambda=0$.
\end{lemma}

\begin{proof}
Let $v \in V_\lambda$ and $W = \U\ga v$.  Since $V$ is an object of $\cal C_A^\lambda$, the submodule $W$ is also an object of $\cal C_A^\lambda$ (Lemma~\ref{lem:I(a,b)closed}).  Moreover, since $v \in V_\lambda$, we have $(\n^+ \otimes A) v=0$ and $h v=\lambda(h)v$ for all $h\in \h$.  Thus, by Proposition~\ref{prop:glob.Weyl.genrel+univ}, there exists a unique (surjective) homomorphism of $\ga$-modules $\pi \colon W_A (\lambda) \twoheadrightarrow W$ satisfying $\pi (w_\lambda) = v$.  Since $\pi$ is a homomorphism of $\ga$-modules and $u w_\lambda = 0$ for all $u \in \Ann_{\ha} (w_\lambda)$, we conclude that $u v = \pi (u w_\lambda) = 0$ for all $u \in \Ann_{\ha} (w_\lambda)$.
\end{proof}

Recall that $\U \ha$ is a commutative algebra, so every left $\U\ha$-module is naturally a right $\U\ha$-module.  Given $\lambda \in X^+$, Lemma~\ref{Ann-1} implies that the left action of $\U \ga$ on any object $V$ of $\cal C_A^\lambda$ induces a left (as well as a right) action of $\fA_\lambda$ on $V_\lambda$.  Since $W_A(\lambda)$ is an object of $\cal C_A^\lambda$ generated by $w_\lambda \in W_A(\lambda)_\lambda$ as a left $\U\ga$-module, we have a right action of $\fA_\lambda$ on $W_A(\lambda)$ that commutes with the left $\U\ga$ action; namely,
\begin{equation} \label{eq:right action}
(uw_\lambda)u' = uu'w_\lambda
\quad
\text{for all $u \in \U \ga$ and $u' \in \U \ha$}.
\end{equation}
Thus, with these actions, $W_A(\lambda)$ is a $\left( \U \ga, \fA_\lambda \right)$-bimodule.

In this section we will define \emph{Weyl functors} for Lie superalgebras. These generalize the Weyl functors defined in \cite[p.~525]{CFK10}. Given $\lambda\in X^+$, let $\mod{\fA_\lambda}$ denote the category of left $\fA_\lambda$-modules and let $M\in \mod{\fA_\lambda}$.  Since $W_A(\lambda)$ is a finitely-semisimple $\lie r$-module and the action of $\lie r$ on $W_A(\lambda) \otimes_{\fA_\lambda} M$ is given by left multiplication, we have that $W_A(\lambda) \otimes_{\fA_\lambda} M$ is a finitely semisimple $\lie r$-module.  Since $\id \colon W_A(\lambda) \to W_A(\lambda)$ is an even homomorphism of $\ga$-modules, for every $M, M'\in \mod{\fA_\lambda}$ and $f \in \Hom_{\fA_\lambda} (M,M')$,
	\[
\id \otimes f \colon W_A(\lambda)\otimes_{\fA_\lambda} M \to W_A(\lambda)\otimes_{\fA_\lambda} M'
	\]
is a homomorphism of $\g\otimes A$-modules.

\begin{defn}[Weyl functor] \label{dfn.super.Weyl.func}
Let $\lambda\in X^+$.  The Weyl functor associated to $\lambda$ is defined to be
\[
\fW_A^\lambda \colon \mod{\fA_\lambda} \rightarrow {\cal C}_A^\lambda, \quad
\fW_A^\lambda M = W_A(\lambda) \otimes_{\fA_\lambda} M,\quad
\fW_A^\lambda f = \id \otimes f,
\]
for all $M, M'$ in $\mod{\fA_\lambda}$ and $f\in \Hom_{\fA_\lambda} (M,M')$.
\end{defn}

Given $\lambda \in X^+$, notice that there is an isomorphism of $\ga$-modules $\fW_A^\lambda \fA_\lambda \cong W_A(\lambda)$.  Also notice that, for all $\mu \in \h^*$ and $M$ in $\mod{\fA_\lambda}$, we have
\begin{equation} \label{eq:weight.space.W}
(\fW_A^\lambda M)_\mu = W_A(\lambda)_\mu\otimes_{\fA_\lambda}M.
\end{equation}

Given $\lambda \in X^+$, recall that Lemma~\ref{Ann-1} implies that $W_A (\lambda)$ is a $\left( \U \ga, \fA_\lambda \right)$-bimodule.  This implies in particular, that $\Hom_{\cal C_A^\lambda} ( W_A (\lambda), N )$ can be viewed as an $\fA_\lambda$-module for any object $N$ of $\cal C_A^\lambda$ via
\[
(u \cdot f) (v) = f (v \cdot u) \quad
\textup{for all $u \in \fA_\lambda$, $f \in \Hom_{\cal C_A^\lambda} (W_A (\lambda), N)$ and $v \in W_A(\lambda)$}.
\]
Moreover, Lemma~\ref{Ann-1} also implies that the left action of $\U{\g \otimes A}$ on an object $V$ in ${\cal C}_A^\lambda$ induces a left action of $\fA_\lambda$ on $V_\lambda$.

\begin{lemma} \label{lem:RAlambda.iso}
Let $\lambda \in X^+$.  For every object $N$ of $\cal C_A^\lambda$, the map
\[
\Hom_{\cal C_A^\lambda} (W_A(\lambda), N) \to N_\lambda, \quad
f \mapsto f(w_\lambda)
\]
is an isomorphism of $\fA_\lambda$-modules that is functorial in $N$.
\end{lemma}

\begin{proof}
Fix an object $N$ in $\cal C_A^\lambda$ and a homomorphism $f \in \Hom_{\cal C_A^\lambda} (W_A(\lambda), N)$.  First notice that, since $w_\lambda \in W_A(\lambda)_\lambda$, then $f(w_\lambda) \in N_\lambda$, that is, the map $f \mapsto f(w_\lambda)$ is well-defined.  Now, to show that $f \mapsto f(w_\lambda)$ is a homomorphism of $\fA_\lambda$-modules, notice that
\[
(u \cdot f) (w_\lambda) = f (w_\lambda \cdot u) = f (u w_\lambda) = u \left( f(w_\lambda) \right) \quad
\textup{for all $u \in \fA_\lambda$}.
\]
To show that the map $f \mapsto f(w_\lambda)$ is injective, recall from Proposition~\ref{prop:glob.Weyl.genrel+univ} that $W_A(\lambda)$ is generated as a left $\U\ga$-module, by $w_\lambda$.  Since $f$ is a homomorphism of $\U\ga$-modules, $f$ is thus uniquely determined by $f(w_\lambda)$.

To finish the proof, we will show that the map $f \mapsto f(w_\lambda)$ is surjective.  Let $n \in N_\lambda$.  Recall that $N$ is an object of $\cal C_A^\lambda$, so $(\n^+ \otimes A) n = 0$.  Moreover, by Lemma~\ref{Ann-1}, we also have $\Ann_{\h \otimes A} (w_\lambda) n = 0$.  Furthermore, since $N$ is a finitely-semisimple $\lie r$-module, the $\lie r$-submodule $\U{\lie r} n \subseteq N$ is finite dimensional.  By the representation theory of finite-dimensional semisimple Lie algebras, we thus have that $(x_\alpha^-)^{\lambda (h_\alpha) + 1} n = 0$ for all $\alpha \in \dr$.  Hence, by Proposition~\ref{prop:glob.Weyl.genrel+univ}, there exists a unique homomorphism of $\ga$-modules $f_n \colon W_A(\lambda) \to N$ satisfying $f_n (w_\lambda) = n$.  The result follows.
\end{proof}

Given $\lambda \in X^+$ and an object $M$ of $\cal C_A^\lambda$, consider $M_\lambda$ as an $\fA_\lambda$-module.  Given $\pi \in \Hom_{{\cal C}_A^\lambda} (V,V')$, the restriction of $\pi$ to $V_\lambda$ induces a homomorphism of $\fA_\lambda$-modules $\pi_\lambda \colon V_\lambda \to V'_\lambda $.  We can thus define a functor
\begin{equation} \label{def:RAlambda}
\fR_A^\lambda \colon {\cal C}_A^\lambda \to \mod{\fA_\lambda}, \quad
\fR_A^\lambda V = V_\lambda, \quad
\fR_A^\lambda(\pi) = \pi_\lambda.
\end{equation}
Notice that $\fR_A^\lambda$ is an exact functor, since every object of $\mathcal C_A^\lambda$ is a finitely-semisimple $\lie r$-module, and thus a direct sum of its $\h$-weight spaces, and every morphism of $\mathcal C_A^\lambda$ preserves these weight spaces.

%
\section{The structure of global Weyl modules}
\label{sec:fin.gen.property}
%

Throughout this section, we will assume that $A$ is finitely generated and infinite dimensional.  Recall from \eqref{eq:reductive.part} that $\lie r$ is a finite-dimensional reductive Lie algebra with Cartan subalgebra $\lie h$, that, for every $\alpha \in R_{\lie r}^+$, the subalgebra $\sl_\alpha \subseteq \lie r$, which is generated by $\{ x_\alpha^- , h_\alpha, x_\alpha \}$, is isomorphic to $\lie {sl}(2)$, and recall from Definition~\ref{def:global-Weyl} that, for every $\lambda \in X^+$, the global Weyl module $W_A(\lambda)$ is generated by its highest-weight vector $w_\lambda \in W_A (\lambda)_\lambda$.

\begin{lemma}\label{lem:x.ann.w}
If $\lambda \in X^+$ and $\alpha \in R_{\lie r}^+$, then $(x_\alpha^-)^{\lambda(h_\alpha)+1} w_\lambda = 0$.
\end{lemma}

\begin{proof}
The result follows from the invariance of the weights of $W_A(\lambda)$ under the action of the Weyl group of $\lie r$.
\details{
Recall that the global Weyl module $W_A (\lambda)$ is a direct sum of its irreducible finite-dimensional $\lie r$-submodules.  Thus all the weights of $W_A (\lambda)$ are invariant under the action of the Weyl group of $\lie r$.  Also recall that $W_A (\lambda)$ is a highest-weight module of highest weight $\lambda$.  In particular, if we denote by $s_\alpha$ the reflection in the Weyl group of $\lie r$ associated to the root $\alpha$, then $W_A (\lambda)_{s_\alpha (\lambda + \alpha)} = 0$.  Now notice that $(x_\alpha^-)^{\lambda(h_\alpha)+1} w_\lambda$ has weight $\lambda - (\lambda(h_\alpha)+1)\alpha = s_\alpha (\lambda + \alpha)$.  Thus $(x_\alpha^-)^{\lambda(h_\alpha)+1}w_\lambda = 0$.
\qedhere
}
\end{proof}

Given $a\in A$ and $\alpha \in R_{\lie r}^+$, define a power series in an indeterminate $u$ and with coefficients in $\U{\h\otimes A}$ as follows:
\begin{equation}\label{def-polis}
p(a,\alpha) = \exp \left( -\sum_{i=1}^\infty \frac{h_\alpha\otimes a^i}{i} u^i\right).
\end{equation}
For $i \ge 0$, let $p(a,\alpha)_i$ denote the coefficient of $u^i$ in $p(a,\alpha)$, and notice that $p(a,\alpha)_0 = 1$.  The following lemma was proved by H.~Garland (see \cite[Lemma~7.5]{Gar78}).

\begin{lemma} \label{lemm5-chari-eq1}
Let $m \in \N$, $a\in A$ and $\alpha \in R_{\lie r}^+$.  Then
\[
(x_\alpha\otimes a)^m (x_\alpha^-)^{m+1} - (-1)^m \sum_{i=0}^m (x_\alpha^- \otimes a^{m-i}) p(a,\alpha)_i \in \U{\sl_\alpha \otimes A}(\g_\alpha \otimes A),
\]
where $\U{\sl_\alpha \otimes A} (\g_\alpha \otimes A)$ denotes the left ideal of $\U{\sl_\alpha \otimes A}$ generated by $\g_\alpha \otimes A = \C x_\alpha \otimes A$.
\end{lemma}

The next two lemmas will be used in the proof of our first main result, Theorem~\ref{fin.dim.odd}.  This first one is technical and part of its proof will be used in the proof of Lemma~\ref{lem:n-.fin.gen}

\begin{lemma}\label{lem:fin-gen-even2}
Let $\lambda\in X^+$, $\alpha \in R_{\lie r}^+$, and $a_1, \dotsc, a_t \in A$. Then, for every $m_1, \dotsc, m_t \in \N$, we have:
\begin{equation} \label{eq:falpha-span-inclusion2}
(x^-_\alpha\otimes a_1^{m_1} \dotsm a_t^{m_t}) w_\lambda
\in \vspan_\C\ \{(x^-_\alpha\otimes a_1^{\ell_1}\dotsm a_t^{\ell_t})w_\lambda \fA_\lambda \mid 0\leq \ell_1, \dotsc, \ell_t < \lambda(h_\alpha) \}.
\end{equation}
In particular, $(\lie r \otimes A) w_\lambda$ is a finitely-generated right $\fA_\lambda$-module.
\end{lemma}

\begin{proof}
We will use induction on $t$.  First assume that $t = 1$, and fix $a \in A$.  From Lemma~\ref{lem:x.ann.w}, the first relation in \eqref{prop:glob.Weyl.genrel+univ}, and Lemma~\ref{lemm5-chari-eq1}, we have:
\[
0 = (x_\alpha \otimes a)^m (x^-_\alpha)^{m+1} w_\lambda
= \sum_{i=0}^m (-1)^m (x^-_\alpha\otimes a^{m-i})p(a,\alpha)_i w_\lambda
\qquad \textup{for all $m \ge \lambda(h_\alpha)$}.
\]
Thus, using the fact that $p(a,\alpha)_0=1$ and induction on $m$, we conclude that
\[
(x^-_\alpha\otimes a^{m})w_\lambda\in \vspan_\C\{(x^-_\alpha\otimes a^\ell)w_\lambda\fA_\lambda \mid 0\leq\ell <\lambda(h_\alpha)\}
\qquad \textup{for all $m\in \N$}.
\]
This proves the case $t=1$.

Now, let $s > 1$, assume that \eqref{eq:falpha-span-inclusion2} holds for all $t \le s$, and fix $a_1, \dotsc, a_{s+1} \in A$.  Since
\begin{equation} \label{eq:arg01}
[x_\alpha^-, h_\alpha] = 2 x_\alpha^-
\quad \textup{and} \quad
\textup{$A$ is assumed to be commutative},
\end{equation}
we have:
\[
2(x^-_\alpha \otimes a_1^{m_1} \dotsm a_{s+1}^{m_{s+1}})w_\lambda
= (x^-_\alpha \otimes a_1^{m_1} \dotsm a_s^{m_s})(h_\alpha \otimes a_{s+1}^{m_{s+1}})w_\lambda - (h_\alpha \otimes a_{s+1}^{m_{s+1}})(x^-_\alpha \otimes a_1^{m_1} \dotsm a_s^{m_s} )w_\lambda,
\]
for all $m_1, \dotsc, m_{s+1} \in \N$.  Thus, using the fact that $(h_\alpha \otimes a_{s+1}^{m_{s+1}}) w_\lambda \in w_\lambda\fA_\lambda$ and the induction hypothesis (for $t=s$), we see that
\begin{align*}
(x^-_\alpha \otimes a_1^{m_1} \dotsm a_{s+1}^{m_{s+1}})w_\lambda
\in {}&\vspan_\C \{ (x^-_\alpha \otimes a_1^{\ell_1} \dotsm a_s^{\ell_s}) w_\lambda\fA_\lambda \mid 0 \le \ell_i < \lambda(h_\alpha), i = 1, \dotsc, s\} \\
&{ }+ \vspan_\C \{ (h_\alpha \otimes a_{s+1}^{m_{s+1}}) (x^-_\alpha \otimes a_1^{\ell_1} \dotsm a_s^{\ell_s}) w_\lambda \mid 0 \le \ell_i < \lambda(h_\alpha),  i = 1, \dotsc, s \},
\end{align*}
for all $m_1, \dotsc, m_{s+1} \in \N$.  Using \eqref{eq:arg01} again, we have:
\begin{alignat*}{2}
(h_\alpha \otimes a_{s+1}^{m_{s+1}})
&(x^-_\alpha \otimes a_1^{\ell_1} \dotsm a_s^{\ell_s} )w_\lambda \\
&= (x^-_\alpha \otimes a_1^{\ell_1} \dotsm a_s^{\ell_s})(h_\alpha \otimes a_{s+1}^{m_{s+1}})w_\lambda
&&- 2(x^-_\alpha \otimes a_1^{\ell_1} \dotsm a_s^{\ell_s} a_{s+1}^{m_{s+1}})w_\lambda \\
&=(x^-_\alpha \otimes a_1^{\ell_1} \dotsm a_s^{\ell_s})(h_\alpha \otimes a_{s+1}^{m_{s+1}})w_\lambda
&&+ (h_\alpha \otimes a_1^{\ell_1})(x^-_\alpha \otimes a_2^{\ell_2} \dotsm a_s^{\ell_s} a_{s+1}^{m_{s+1}})w_\lambda \\
&{ }&&- (x^-_\alpha \otimes a_2^{\ell_2} \dotsm a_s^{\ell_s} a_{s+1}^{m_{s+1}})(h_\alpha \otimes a_1^{\ell_1})w_\lambda,
\end{alignat*}
for all $0 \le \ell_1, \dotsc, \ell_s < \lambda(h_\alpha)$ and $m_{s+1} \in \N$.  Thus, using the induction hypothesis again (for $t=s$) on $(h_\alpha \otimes a_1^{\ell_1})(x^-_\alpha \otimes a_2^{\ell_2} \dotsm a_s^{\ell_s} a_{s+1}^{m_{s+1}})w_\lambda$ and $(x^-_\alpha \otimes a_2^{\ell_2} \dotsm a_s^{\ell_s} a_{s+1}^{m_{s+1}})(h_\alpha \otimes a_1^{\ell_1})w_\lambda$, we see that
\begin{align*}
(h_\alpha \otimes a_{s+1}^{m_{s+1}})&(x^-_\alpha \otimes a_1^{\ell_1} \dotsm a_s^{\ell_s} )w_\lambda \\
\in {}&\vspan_\C \{ (x^-_\alpha \otimes a_1^{k_1} \dotsm a_{s+1}^{k_{s+1}}) w_\lambda\fA_\lambda \mid 0 \le k_i < \lambda(h_\alpha), i = 1, \dotsc, s+1\} \\
&{ }+ \vspan_\C \{ (h_\alpha \otimes a_1^{\ell_1}) (x^-_\alpha \otimes a_2^{k_2} \dotsm a_{s+1}^{k_{s+1}}) w_\lambda \mid 0 \le k_i < \lambda(h_\alpha),  i = 2, \dotsc, s+1 \},
\end{align*}
for all $0 \le \ell_1, \dotsc, \ell_s < \lambda(h_\alpha)$ and $m_{s+1} \in \N$.  Finally, using \eqref{eq:arg01} again, we have:
\begin{align*}
(h_\alpha \otimes a_1^{\ell_1})
&(x^-_\alpha \otimes a_2^{k_2} \dotsm a_{s+1}^{k_{s+1}})w_\lambda \\
&= (x^-_\alpha \otimes a_2^{k_2} \dotsm a_{s+1}^{k_{s+1}})(h_\alpha \otimes a_1^{\ell_1})w_\lambda
- 2(x^-_\alpha \otimes a_1^{\ell_1} a_2^{k_2} \dotsm a_{s+1}^{k_{s+1}})w_\lambda \\
&\in \vspan_\C \{ (x^-_\alpha \otimes a_1^{n_1} \dotsm a_{s+1}^{n_{s+1}}) w_\lambda \fA_\lambda \mid 0 \le n_i < \lambda(h_\alpha),  i = 1, \dotsc, s+1 \},
\end{align*}
for all $0 \le \ell_1, k_2, \dotsc, k_{s+1} < \lambda(h_\alpha)$.  Hence, \eqref{eq:falpha-span-inclusion2} follows.

In particular, using \eqref{eq:falpha-span-inclusion2} and the assumptions that $A$ is finitely generated and $\lie r$ is a finite-dimensional Lie algebra, we conclude that $(\lie r \otimes A) w_\lambda$ is a finitely-generated right $\fA_\lambda$-module.
\end{proof}

\begin{lemma}\label{lem:n-.fin.gen}
Let $\lambda\in X^+$, $\alpha \in R_{\lie r}^+$, $x_1, \dotsc, x_k \in \lie n^+$ and $a_1, \dotsc, a_t \in A$.  Then, for all $m_1, \dotsc, m_t \in \N$, the element $([x_1, [x_2, \dotsc [x_k, x^-_\alpha] \dotsc]] \otimes a_1^{m_1} \dotsm a_t^{m_t}) w_\lambda$ is in
\[
\vspan_\C \{ ([x_1, [x_2, \dotsc [x_k, x^-_\alpha] \dotsc]] \otimes a_1^{\ell_1} \dotsm a_t^{\ell_t}) w_\lambda \fA_\lambda \mid  0 \leq \ell_1, \dotsc, \ell_t < \lambda(h_\alpha) \}.
\]
\end{lemma}

\begin{proof}
The proof is by induction on $k$.  First assume that $k=1$ and let $x \in \n^+$.  Using \eqref{eq:falpha-span-inclusion2} and the first relation in \eqref{prop:glob.Weyl.genrel+univ}, for all $m_1,\dotsc,m_t \in \N$, we have:
\begin{align*}
([x, x^-_\alpha] \otimes a_1^{m_1} \dotsm a_t^{m_t}) w_\lambda
&= [x \otimes 1, x^-_\alpha \otimes a_1^{m_1} \dotsm a_t^{m_t}] w_\lambda \\
&= (x \otimes 1) (x^-_\alpha \otimes a_1^{m_1} \dotsm a_t^{m_t}) w_\lambda  \\
&\in \vspan_\C \{ ([x, x^-_\alpha] \otimes a_1^{\ell_1} \dotsm a_t^{\ell_t}) w_\lambda \fA_\lambda \mid 0 \leq \ell_1, \dotsc, \ell_t < \lambda(h_\alpha) \}.
\end{align*}
This proves the case $k = 1$.  Now assume $k > 1$ and let $x_1, \dotsc, x_k \in \n^+$.  Using the first relation in \eqref{prop:glob.Weyl.genrel+univ} and the induction hypothesis, for all $m_1, \dotsc, m_t \in \N$, we have:
\begin{align*}
([x_1, [x_2, \dotsm [x_k, x^-_\alpha] & \dotsm ]] \otimes a_1^{m_1} \dotsm a_t^{m_t}) w_\lambda  \\
{ }&= [(x_1 \otimes 1), ([x_2, [x_3, \dotsm [x_k, x^-_\alpha] \dotsm ]] \otimes a_1^{m_1} \dotsm a_t^{m_t})] w_\lambda \\
{ }&= (x_1 \otimes 1) ([x_2, [x_3, \dotsm [x_k, x^-_\alpha] \dotsm ]] \otimes a_1^{m_1} \dotsm a_t^{m_t}) w_\lambda   \\
&{ }\in \{ ([x_1, [x_2, \dotsm [x_k, x^-_\alpha] \dotsm ]]\otimes a_1^{\ell_1} \dotsm a_t^{\ell_t}) w_\lambda \fA_\lambda \mid 0 \leq \ell_1, \dotsc, \ell_t <\lambda(h_\alpha) \}.
\qedhere
\end{align*}
\end{proof}

Let $\bu(\n^-\otimes A)=\sum_{n\geq 0} \bu_n(\n^-\otimes A)$ be the filtration on $\bu(\n^-\otimes A)$ induced from the usual grading of the tensor algebra $T (\n^- \otimes A) = \bigoplus_{d \ge 0} (\n^- \otimes A)^{\otimes d}$. 

\begin{lemma}\label{lem:fin.filtration}
If $\g$ is a finite-dimensional simple Lie superalgebra, not of type $\lie q(n)$, endowed with a triangular decomposition satisfying \hyperref[C2]{$(\mathfrak C2)$}, then there exists $n_0\in \N$ such that 
\[
\bu_n(\n^-\otimes A) w_\lambda\fA_\lambda = W_A (\lambda),
\quad \textup{for all $n\geq n_0$}.
\]
\end{lemma}

\begin{proof}
First, recall that $W_A(\lambda) = \U{\n^- \otimes A} w_\lambda \fA_\lambda$.  Then, by the PBW Theorem,
\[
W_A(\lambda) = \U{\n^-_{\bar1} \otimes A} \U{\n^-_{\bar0} \otimes A} w_\lambda \fA_\lambda,
\]
where, by abuse of notation, we are denoting by $\bu(\n^-_{\bar 1}\otimes A)$ the subspace of $\bu(\g\otimes A)$ whose basis consists of all the elements in the PBW basis of $\bu(\g\otimes A)$ which have no even component. Since we are assuming that the triangular decomposition $\g = \nm \oplus \h \oplus \np$ satisfies \hyperref[C2]{$(\mathfrak C2)$}, we have $\nm_{\bar0}= \nm_0$.  Hence, $\U{\n^-_{\bar0} \otimes A} w_\lambda$, which is the $(\lie r\otimes A)$-submodule of $W_A(\lambda)$ generated by $w_\lambda$, is a quotient of the Weyl $(\lie r\otimes A)$-module of highest weight $\lambda$.  This is a finitely-generated $\fA_\lambda$-module by \cite[Theorem~2(i)]{CFK10}. Thus $\U{\n^-_{\bar0} \otimes A} w_\lambda$ is a finitely-generated $\fA_\lambda$-module, that is, there exist $f_1, \dotsc, f_k \in \nm_{\bar0} \otimes A$ such that
\[
\U{\n^-_{\bar0} \otimes A} w_\lambda \fA_\lambda
= \sum_{1 \le i_1 \le \dotsb \le i_t \le k} f_{i_1} \dotsm f_{i_t} w_\lambda \fA_\lambda.
\]

Now, recall that $-\theta$ denotes the lowest root of $\g$ and that we have fixed a triangular decomposition of $\g$ satisfying \hyperref[C2]{$(\mathfrak C2)$}.  Hence $\theta \in R^+_{\lie r}$.  Notice that, since $\g$ is assumed to be finite dimensional, there exists $k_0 \in \N$ such that $[x_1, [x_2, \dotsc [x_k, x^-_\theta] \dotsc ]] = 0$, for all $k > k_0$ and $x_1, \dotsc, x_k \in \lie n^+$.  Moreover, since $\g$ is assumed to be simple and $x_\theta^-$ is a lowest-weight vector in the $\g$-module $\g$, we have
\begin{equation} \label{eq:n-.gen.set}
\lie n^- \subseteq \vspan_\C \{ [x_1, [x_2, \dotsm [x_k, x^-_\theta] \dotsm ]] \mid x_1, \dotsc, x_k \in \n^+ \textup{ and } 0 \le k \le k_0 \}.
\end{equation}
Hence, it follows from Lemma~\ref{lem:n-.fin.gen} that, for each $\alpha \in R^+$, the space $(\g_{-\alpha}\otimes A)w_\lambda$ is a finitely-generated $\fA_\lambda$-module. Thus, $(\n_{\bar 1}^-\otimes A)w_\lambda$ is a finitely-generated $\fA_\lambda$-module, that is, there exist $g_1,\dotsc, g_\ell \in \nm_{\bar1} \otimes A$ such that
\[
(\n^-_{\bar1} \otimes A) w_\lambda \fA_\lambda
= \sum_{1 \le j_1 \le \dotsb \le j_s \le \ell} g_{j_1} \dotsm g_{j_s} w_\lambda \fA_\lambda.
\]

Moreover, notice that $[\nm_{\bar0} \otimes A, \nm_{\bar1} \otimes A] \subseteq \nm_{\bar1} \otimes A$.  Then one can use induction on $s$ and $t$ (similar to the proof of Lemma~\ref{lem:fin-gen-even2}) to prove that
\[
\U{\n^-_{\bar1} \otimes A} \U{\n^-_{\bar0} \otimes A} w_\lambda \fA_\lambda
= \sum_{\atop{1 \le i_1 \le \dotsb \le i_t \le k}{1 \le j_1 \le \dotsb \le j_s \le \ell}} g_{j_1} \dotsm g_{j_s} f_{i_1} \dotsm f_{i_t} w_\lambda \fA_\lambda.
\]
The result follows.
\end{proof}

We now state and prove the first main result of the paper.  Notice that the existence of a triangular decomposition satisfying \hyperref[C2]{$(\mathfrak C2)$} will be crucial in the proofs of Theorem~\ref{fin.dim.odd} and Proposition~\ref{prop:fin.gen.A} (see Example~\ref{rmk:global.weyl.not.fg}).

\begin{thm} \label{fin.dim.odd}
Let $\g$ be a finite-dimensional simple Lie superalgebra not of type $\lie q(n)$ endowed with a trinagular decomposition satisfying \hyperref[C2]{$(\mathfrak C2)$}.  For all $\lambda\in X^+$, the global Weyl module $W_A(\lambda)$ is finitely generated as a right $\fA_\lambda$-module
\end{thm}

\begin{proof}
We will show that, for every $n \geq 0$, $\bu_n(\n^-\otimes A) w_\lambda\fA_\lambda$ is a finitely-generated $\fA_\lambda$-module.  Recall that $-\theta$ denotes the lowest root of $\g$.  Also recall that $A$ is assumed to be finitely generated and let $a_1, a_2, \dotsc, a_t$ be generators of $A$.  Denote by $\cal B_{\n^-}$ a (finite) basis of $\n^-$ extracted from the right side of \eqref{eq:n-.gen.set} and let $\cal B_{\n^- \otimes A}$ be the (finite) set
\[
\{ y \otimes a_1^{\ell_1} \dotsm a_t^{\ell_t} \mid y \in \cal B_{\n^-} \textup{ and } 0 \leq \ell_1, \dotsc, \ell_t <\lambda(h_\theta) \}.
\]
We will use induction to prove that, for every $n\in \N_+$,
\[
\bu_n(\n^-\otimes A) w_\lambda \subseteq \vspan_\C \{ Y_1^{n_1} \dotsm Y_t^{n_t} w_\lambda \fA_\lambda \mid t \ge 0, \  Y_1, \dotsc, Y_t \in \cal B_{\lie n^- \otimes A} \textup{ and } n_1 + \dotsb + n_t \le n \}.
\]
For $n=1$, the result follows from Lemma~\ref{lem:n-.fin.gen} and the construction of $\cal B_{\lie n^- \otimes A}$.  Suppose now $n > 1$. Without loss of generality, let $u = u_1 u_{n-1}$ be a monomial, with $u_1 \in \bu_1(\n^-\otimes A)$ and $u_{n-1} \in \bu_{n-1} (\n^-\otimes A)$.  By induction hypothesis, we have:
\[
u w_\lambda
= u_1 u_{n-1} w_\lambda
\in \vspan_\C \{ u_1 Y_1^{n_1} \dotsm Y_t^{n_t} w_\lambda \fA_\lambda \mid t \ge 0, \ Y_1, \dotsc, Y_t \in \cal B_{\lie n^- \otimes A}, \ n_1 + \dotsb + n_t \le n-1 \}.
\]
Let $u'$ be an element in $\vspan_\C \{ Y_1^{n_1} \dotsm Y_t^{n_t} w_\lambda \fA_\lambda \mid t \ge 0, \ Y_1, \dotsc, Y_t \in \cal B_{\lie n^- \otimes A}, \  n_1 + \dotsb + n_t \le n-1 \}$, and without loss of generality assume that $u_1$ and $u'$ are homogeneous.  By induction hypothesis, we have:
\begin{align*}
u_1 u' w_\lambda
& = [u_1, u'] w_\lambda + (-1)^{p(u_1)p(u')} u' u_1 w_\lambda \\
& \in \bu_{n-1} (\n^-\otimes A) w_\lambda \fA_\lambda + \vspan_\C \{ u' Y w_\lambda \fA_\lambda \mid Y \in \cal B_{\lie n^- \otimes A} \} \\
& \subseteq \vspan_\C \{ Y_1^{n_1} \dotsm Y_{t+1}^{n_{t+1}} w_\lambda \fA_\lambda \mid t \ge 0, \  Y_1, \dotsc, Y_{t+1} \in \cal B_{\lie n^- \otimes A} \textup{ and } n_1 + \dotsb + n_{t+1} \le n \}.
\end{align*}
This shows that $\bu_n(\n^-\otimes A) w_\lambda\fA_\lambda$ is a finitely-generated $\fA_\lambda$-module for each $n \geq 0$.  Since there exists $n_0 \in \N_+$ such that $\bu_n(\n^-\otimes A) w_\lambda\fA_\lambda = W_A (\lambda)$ for all $n\geq n_0$ (by Lemma~\ref{lem:fin.filtration}), the result follows.
\end{proof}

In the non-super setting, Theorem~\ref{fin.dim.odd} was proved in \cite[Theorem~2(i)]{CFK10} for the untwisted case, and in \cite[Theorem~5.10]{FMS15} for the twisted case.  Notice that in the non-super setting the analogues of Theorem~\ref{fin.dim.odd} do not depend on the choice of the triangular decomposition of $\g$.

\begin{ex}\label{rmk:global.weyl.not.fg}
Let $\g$ be either a basic classical Lie superalgebra of type I or isomorphic to $\lie p(n)$, $S(n)$, $H(n)$, or $W(n)$, and $A$ be an associative, commutative, finitely-generated infinite-dimensional $\C$-algebra.  We will show that, for all $\lambda \in X^+$, the global Weyl module $W_A (\lambda)$ associated to a parabolic triangular decomposition (recall Definition~\ref{defn.parabolic.td}) $\g=\n^-\oplus \h\oplus \n^+$, is \emph{not} finitely generated.  Notice that a parabolic triangular decomposition cannot satisfy \hyperref[C2]{$(\mathfrak C2)$}, since $\lie r + \np$ is not a subalgebra if the triangular decomposition satisfies \hyperref[C2]{$(\mathfrak C2)$}.

In fact, in this case we have:
	\[
\ga = \left(\g^-\otimes A \right) \oplus \left( (\lie{r} + \np)\otimes A \right),
	\]
where $\lie{r}+\np$ is a subsuperalgebra and $\g^-$ is a nontrivial subspace of $\g$. Thus we can consider the $(\lie{r}+\np)\otimes A$-module $W^{\lir}(\lambda)$ given as the quotient of $\U{(\lir + \np) \otimes A}$ by the left ideal generated by
\[
\g_\alpha\otimes A \quad \textup{for all } \alpha\in R^+, \qquad
h - \lambda(h) \quad \textup{for all } h \in \h, \qquad
(x_\alpha^{-})^{\lambda(h_\alpha)+1} \quad \textup{for all } \alpha \in \dr.
\]
Notice that the image of $1 \in \U{(\lir + \np) \otimes A}$, which we will denote by $u_\lambda$, generates $W^{\lir}(\lambda)$.

Now, let
	\[
\overline{W^\lir}(\lambda) = \ind{(\lie{r}+\np)\otimes A}{\ga} W^\lir (\lambda),
	\]
and notice that $\overline{W^\lir}(\lambda)$ is generated by $1 \otimes u_\lambda$.  Also notice that there exists a unique surjective homomorphism of $\ga$-modules $W_A(\lambda) \to \overline{W^\lir}(\lambda)$ satisfying $w_\lambda \mapsto (1 \otimes u_\lambda)$, thus $\overline{W^\lir}(\lambda)$ admits a structure of right $\fA_\lambda$-module (cf. Lemma~\ref{Ann-1}). Moreover, $(\g^-\otimes A)\otimes W^\lir (\lambda)$ is a right $\fA_\lambda$-submodule of $\overline{W^\lir}(\lambda)$.  Since $A$ is assumed to be infinite dimensional, we have that $\overline{W^\lir}(\lambda)$ is not finitely generated as a right $\fA_\lambda$-module. This concludes that $W_A (\lambda)$ is also not finitely generated as an $\fA_\lambda$-module.
\end{ex}

\begin{prop}\label{prop:fin.gen.A}
Let $\g$ be a finite-dimensional simple Lie superalgebra not of type $\lie q(n)$ endowed with a triangular decomposition that satisfies \hyperref[C2]{$(\mathfrak C2)$}.  For all $\lambda\in X^+$, the algebra $\fA_\lambda$ is finitely generated.
\end{prop}
\begin{proof}
Since $\fA_\lambda$ is defined to be $\U{\ha} / \Ann_{\ha} (w_\lambda)$, to prove that $\fA_\lambda$ is finitely generated is equivalent to proving that there exist finitely many elements $H_1, \dotsc, H_n \in \U{\ha}$ such that
\[
\U{\ha}w_\lambda
= \vspan_\C \{ H_1^{k_1} \dotsm H_n^{k_n} w_\lambda \mid k_1, \dotsc, k_n \ge 0 \}.
\]
Moreover, since $\U{\ha}$ is a commutative algebra generated by $\ha$, this is equivalent to proving that
\begin{equation} \label{eq:equiv.fg}
(\ha)w_\lambda
\subseteq \vspan_\C \{ H_1^{k_1} \dotsm H_n^{k_n} w_\lambda \mid k_1, \dotsc, k_n \ge 0 \}.
\end{equation}

In order to prove \eqref{eq:equiv.fg}, first recall that $A$ is assumed to be finitely generated and let $a_1, a_2, \dotsc, a_t$ be generators of $A$.  Now denote by $-\theta$ the lowest root of $\g$.  Notice that, since we have fixed a triangular decomposition of $\g$ satisfying \hyperref[C2]{$(\mathfrak C2)$}, we have $\theta \in R^+_{\lie r}$.  Also notice that, since $\g$ is assumed to be finite dimensional, there exists $k_0 \in \N$ such that $[x_1, [x_2, \dotsc [x_k, x^-_\theta] \dotsc ]] = 0$, for all $k > k_0$ and $x_1, \dotsc, x_k \in \lie n^+$.  Moreover, since $\g$ is assumed to be simple and $x_\theta^-$ is a lowest-weight vector in the $\g$-module $\g$, we have
\[
\lie h \otimes A \subseteq \vspan_\C \{ [x_1, [x_2, \dotsm [x_k, x^-_\theta] \dotsm ]] \otimes a_1^{m_1} \dotsm a_t^{m_t} \mid x_1, \dotsc, x_k \in \n^+, \ 0 < k \le k_0, \ 0 \le m_1, \dotsc, m_t \}.
\]

Using arguments similar to those used in the proof of Lemma~\ref{lem:fin-gen-even2}, we see that for every $k \in\N_+$ and $x_1, \dotsc, x_k \in \lie n^+$ such that $[x_1, [x_2, \dotsc [x_k, x^-_\theta] \dotsc]] \in \h$, the element $([x_1, [x_2, \dotsc [x_k, x^-_\theta] \dotsc ]] \otimes a_1^{m_1} \dotsm a_t^{m_t}) w_\lambda$ is a linear combination of elements of the form
\begin{equation*}\label{eq1:cartan.fin.gen}
([x_1, [x_2, \dotsc [x_k, x^-_\theta] \dotsc]] \otimes a_1^{\ell_1} \dotsm a_t^{\ell_t}) P(\theta, k_1, \dotsc, k_t) w_\lambda,
\end{equation*}
where $0 \leq \ell_1, \dotsc, \ell_t < \lambda(h_\theta)$, $0 \leq k_1, \dotsc, k_t \le \lambda(h_\theta)$, and $P(\theta, k_1, \dotsc, k_t)$ is a finite product of elements of $\bu(\h\otimes A)$ of the form $(h_\theta \otimes a_1^{k_1} \dotsm a_t^{k_t})$. Thus the result follows.
\end{proof}

\begin{ex}
Given $k > 0$, let $S_k$ denote the symmetric group on $k$ letters, let $(A^{\otimes k})^{S_k}$ denote the subalgebra of $A^{\otimes k}$ consisting of all the fixed points under the natural action of $S_k$ on $A^{\otimes k}$. When $\g$ is either of type II, or isomorphic to $S(n)$ or $H(n)$, $\lie r$ is a finite-dimensional semisimple Lie algebra.  In particular, $X^+\subseteq P_{\lie r}^+$, where $P_{\lie r}^+$ denote the set of dominant integral weights of $\lie r$.  Thus, in these cases (as described in \cite[Theorem~4]{CFK10}), the algebra $\fA_\lambda$ is isomorphic to the algebra $(A^{\otimes r_1})^{S_{r_1}} \otimes \dotsm \otimes (A^{\otimes r_n})^{S_{r_n}}$, where $r_1, \dotsc, r_n$ are unique non-negative integers such that $\lambda = r_1 \omega_1 + \dotsb + r_n \omega_n$, and where $\omega_1,\ldots, \omega_n$ denote the fundamental integral weights of $\lie r$. If $\g$ is either basic classical of type I or isomorphic to $W(n)$, then $\lie{r} = \lie{z} \oplus \lie{r}'$, where $\lie{z}$ is the 1-dimensional center and $\lie{r}'$ is the semisimple part of $\lie r$, and $\h=\lie{z} \oplus \h'$, where $\h'$ is a Cartan subalgebra of $\lie{r}'$. If $\lambda \arrowvert_{\lie z} = 0$, then $\fA_\lambda$ is also isomorphic to the algebra $(A^{\otimes r_1})^{S_{r_1}} \otimes \dotsm \otimes (A^{\otimes r_n})^{S_{r_n}}$. If $\lambda(z)\neq 0$ for some $z\in {\lie z}$, there exist $\Lambda\in P_{\lie r'}^+$ and $\eta\in {\lie z}^*$ such that $\lambda(z,h)=\eta(z)+\Lambda(h)$ for every $z\in {\lie z}$, $h\in \h'$. Then by the proof of Proposition~\ref{prop:fin.gen.A} we have that
	\[
\bu({\lie z}\otimes A)w_\lambda \subseteq \C[z\otimes a_i^{\ell_i}] \bu(\h'\otimes A)w_\lambda,
	\]
where $0 \leq \ell_i < \lambda(h_\theta)$, for all $i=1,\ldots, t$. In particular, since $\bu(\h\otimes A)\cong \bu(\lie z\otimes A)\otimes \bu(\h'\otimes A)$, this yields a surjective homomorphism of algebras
	\[
\C[z\otimes a_i^{\ell_i}] \otimes (A^{\otimes r_1})^{S_{r_1}} \otimes \dotsm \otimes (A^{\otimes r_n})^{S_{r_n}} \twoheadrightarrow \fA_\lambda,
	\]
where now $r_1, \dotsc, r_n$ are the unique non-negative integers such that $\Lambda = r_1 \omega_1 + \dotsb + r_n \omega_n$.
\end{ex}

The next result follows directly from Theorem~\ref{fin.dim.odd}.

\begin{cor}\label{cor5.12-FMS}
Let $\g$ be a finite-dimensional simple Lie superalgebra not of type $\lie q(n)$ endowed with a triangular decomposition that satisfies \hyperref[C2]{$(\mathfrak C2)$}.  If $M$ is a finitely-generated $\fA_\lambda$-module (resp. finite dimensional), then $\fW_A^\lambda M$ is a finitely-generated $\g\otimes A$-module (resp. finite dimensional).  \qed
\end{cor}

%
\section{Local Weyl modules} \label{sec:local-Weyl}
%

In this section we will assume that $\g$ is either $\sl(n,n)$ with $n\geq 2$, or a finite-dimensional simple Lie superalgebra not of type $\lie q(n)$, and that $A$ is an associative commutative finitely-generated $\C$-algebra with unit.

\begin{defn}[Local Weyl module] \label{def:local-Weyl}
Assume that $\psi\in (\h\otimes A)^*$ and $\psi \arrowvert_{\h} \in X^+$.  The \emph{local Weyl module} $W_A^\loc(\psi)$ associated to $\psi$ is defined to be the cyclic $\g\otimes A$-module given as the quotient of $\U\ga$ by the left ideal generated by
\[
\np \otimes A, \quad
h - \psi(h), \quad
(x_\alpha^-)^{\psi(h_\alpha)+1},\quad
\textup{for all $h \in \ha$ and $\alpha \in \dr$}.
\]
\end{defn}

Denote the image of $1 \in \U\ga$ in $W_A^\loc (\psi)$ by $w_\psi$, and notice that as a $\ga$-module, $W_A^\loc (\psi)$ is generated by the vector $w_\psi$, satisfying the following defining relations:
\begin{equation}\label{relations-local}
(\n^+\otimes A)w_\psi = 0, \quad
h w_\psi = \psi(h) w_\psi, \quad
(x_\alpha^-)^{\psi(h_\alpha)+1} w_\psi=0,
\quad \textup{for all $h \in \h \otimes A$ and $\alpha \in \dr$}.
\end{equation}

The next result describes local Weyl modules as universal objects.  Its proof is similar to that of \cite[Proposition~4.13]{CLS}.

\begin{prop} \label{prop:local.weyl.univ}
Let $\psi\in (\h\otimes A)^*$ such that $\psi \arrowvert_\h = \lambda \in X^+$. Assume that $W\in {\cal C}_A^\lambda$ is a finite-dimensional $\ga$-module that is generated by a highest weight vector $w\in W$ such that $xv=\psi(x)v$, for all $x\in \h\otimes A$. Then there exists a surjective homomorphism from $W_A^\loc(\psi)$ to $W$ sending $w_\psi$ to $w$. Moreover, if the triangular decomposition of $\g$ satisfies \hyperref[C1]{$(\mathfrak C1)$}, then $W_A^\loc(\psi)$ is the unique object in ${\cal C}_A^\lambda$ with this property.
\qed
\end{prop}

\dproof{
First notice that by hypothesis, $w$ satisfies the first two sets of relations in \eqref{relations-local}.  Moreover, since $W$ is assumed to be finite dimensional, the third set of relations in \eqref{relations-local} is also satisfied by $w$.  Therefore there exists a surjective homomorphism of $\ga$-modules $W_A^\loc(\psi) \twoheadrightarrow W$ such that $w_\psi \mapsto w$.

Suppose now that $U$ is another finite-dimensional object in ${\cal C}_A^\lambda$ such that there exists a surjective homomorphism from $U \twoheadrightarrow W$ for every $W \in {\cal C}_A^\lambda$ that is a finite-dimensional $\ga$-module generated by a highest weight vector $w\in W$ such that $xv=\psi(x)v$, for all $x\in \h\otimes A$.  Then $W_A^\loc(\psi)$ is a quotient of $U$ and vice versa.  Since $U$ and $W_A^\loc(\psi)$ are assumed to be finite dimensional, we conclude that $U \cong W_A^\loc(\psi)$.
\qedhere
}

Notice that, since $\fA_\lambda$ is a commutative algebra, every irreducible finite-dimensional $\fA_\lambda$-module is one-dimensional. For $\psi\in (\h\otimes A)^*$ such that $\psi |_{\h} \in X^+$, let $\C_\psi$ denote the one-dimensional irreducible $\fA_\lambda$-module, where $x v=\psi(x)v$ for all $x\in \fA_\lambda$ and $v\in \C_\psi$.

\begin{rmk} \label{rmk:local.props}
Recall that $W_A^\loc(\psi)$ is generated by $w_\psi$, that is, $W_A^\loc(\psi) = \U{\ga}w_\psi$.  Thus, since $w_\psi$ satisfies $(\n^+\otimes A)w_\psi = 0$ and $h w_\psi = \psi(h) w_\psi$ for all $h \in \ha$, we have $\fR_A^{\psi \arrowvert_\h} W_A^\loc(\psi) = \C w_\psi$.  Moreover, notice that $\C w_\psi$ is isomorphic to $\C_\psi$ as a $\fA_\lambda$-module.
\end{rmk}

For the remainder of this section we fix a triangular decomposition of $\g$ satisfying \hyperref[C2]{$(\mathfrak C2)$}.  The next result describes local Weyl modules via Weyl functors.

\begin{thm}\label{thm:recovering.local.Weyl.module}
Assume that $\g$ is a finite-dimensional simple Lie superalgebra not isomorphic to $\lie q(n)$.  Let $\psi\in (\h\otimes A)^*$ such that $\psi \arrowvert_\h = \lambda \in X^+$. Then $\fW_A^\lambda \C_\psi \cong W_A^\loc(\psi)$.
\end{thm}

\begin{proof}
First recall from Remark~\ref{rmk:local.props} that $W_A^\loc(\psi)=\bu(\g\otimes A)w_\psi$ and $\fR_A^\lambda W_A^\loc(\psi) = \C w_\psi$.  Thus, there exists a unique homomorphism of $\ga$-modules $\epsilon_{W_A^\loc(\psi)} \colon \fW_A^\lambda\fR_A^\lambda W_A^\loc(\psi) \to W_A^\loc(\psi)$ satisfying
\[
\epsilon_{W_A^\loc(\psi)} (u \otimes w_\psi) = uw_\psi
\qquad \textup{for all $u \in \U\ga$}.
\]
Moreover, $\epsilon_{W_A^\loc(\psi)}$ is surjective.

Now, notice that $\fW_A^\lambda \fR_A^\lambda W_A^\loc(\psi)$ is a $\ga$-module generated by the highest-weight vector $1 \otimes w_\psi$ (see Remark~\ref{rmk:local.props}).  Moreover, by Corollary~\ref{cor5.12-FMS}, $\fW_A^\lambda \fR_A^\lambda W_A^\loc(\psi)$ is finite dimensional.  Thus, $1 \otimes w_\psi$ satisfies all the relations \eqref{cha-fou-equ5}.  This implies that we have a unique homomorphism of $\ga$-modules $\eta \colon W_A^\loc(\psi) \to \fW_A^\lambda\fR_A^\lambda W_A^\loc(\psi)$ satisfying $\eta(w_\psi) = 1 \otimes w_\psi$.  Moreover, $\eta$ is surjective, $\eta \circ \epsilon_{W_A^\loc(\psi)} = \id_{W_A^\loc(\psi)}$, and $\epsilon_{W_A^\loc(\psi)} \circ \eta = \id_{\fW_A^\lambda\fR_A^\lambda W_A^\loc(\psi)}$.  The result follows.
\end{proof}

The next result follows directly from Theorems~\ref{thm:tensor-property2} and \ref{thm:recovering.local.Weyl.module}.

\begin{cor}\label{cor:tensor.product.local}
Let $A$ and $B$ be finite-dimensional commutative, associative $\C$-algebras with unit, $\psi\in(\h\otimes A)^*$ such that $\psi|_{\h}=\lambda\in X^+$ and $\varphi\in(\h\otimes B)^*$ such that $\varphi|_\h=\mu\in X^+$ and $\lambda+\mu \in X^+$. Then
  \[
    \fW_{A\oplus B}^{\lambda+\mu}\left(\Delta_{\lambda,\mu}^* (\C_{\psi+\varphi})\right) \cong \pi_A^*(W_A^\loc(\psi))\otimes \pi_B^*(W_B(\varphi))
  \]
as $\g\otimes (A\oplus B)$-modules.  \qed
\end{cor}

The next result gives a homological characterization of local Weyl modules,
and its proof is similar to that of \cite[Lemma~7.5]{FMS15}.

\begin{cor}\label{cor:caracterizing-local-wey}
Let $\psi\in (\h\otimes A)^*$ such that $\psi \arrowvert_\h = \lambda \in X^+$.  A $\g\otimes A$-module $V$ is isomorphic to the local Weyl module $W_A^\loc(\psi)$ if and only if it satisfies all of the following conditions:
\begin{enumerate}[leftmargin=*]
\item $V\in {\cal C}_A^\lambda$;

\item $\fR_A^\lambda V\cong \C_\psi$;

\item $\Hom_{{\cal C}_A^\lambda}(V,U)=0\text{ and } \Ext_{{\cal C}_A^\lambda}^1(V,U)=0$, for all finite-dimensional irreducible $U\in {\cal C}_A^\lambda$ with $U_\lambda=0$.
\qed
\end{enumerate}
\end{cor}

\dproof{
If $V$ satisfies all the conditions, then it follows from Lemma~\ref{prop8} and Theorem~\ref{thm:recovering.local.Weyl.module} that
\[
V
\cong \fW_A^\lambda \fR_A^\lambda V
\cong \fW_A^\lambda \C_\psi
\cong W_A^\loc(\psi).
\]
Conversely, recall that $W_A^\loc(\lambda) \in {\cal C}_A^\lambda$ and $\fR_A^\lambda W_A^\loc(\lambda) \cong \C_\psi$ (see Remark~\ref{rmk:local.props}).  Hence, by Theorem~\ref{thm:recovering.local.Weyl.module}, we have that $W_A^\loc(\psi) \cong  \fW_A^\lambda \C_\psi \cong \fW_A^\lambda\fR_A^\lambda W_A^\loc(\psi)$.  Thus, by Lemma~\ref{prop8} (or Theorem~\ref{thm:hom-ext1-vanish}), $\Hom_{{\cal C}_A^\lambda}(W_A^\loc(\psi),U) = \Ext_{{\cal C}_A^\lambda}^1(V,U)=0$, for all finite-dimensional irreducible $U \in {\cal C}_A^\lambda$ with $U_\lambda=0$.
\qedhere
}

Now we give necessary and sufficient conditions for local Weyl modules to be finite dimensional.  We begin by giving a sufficient condition.  In the non-super setting, this result was proved in \cite[Theorem~1]{CP01} for $A = \C[t^{\pm1}]$, and in \cite[Theorem~1]{FL04}, for the case where $A$ is the algebra of functions on a complex affine variety.  For the case where $\g$ is either basic classical or $\sl (n,n)$ with $n \ge 2$, and $A$ is finitely generated, it was proved in \cite[Theorem~4.12]{CLS}.  In our curent setting, the result is a direct consequence of Corollary~\ref{cor5.12-FMS} and Theorem~\ref{thm:recovering.local.Weyl.module}.

\begin{thm}\label{thm:local-weyl-fd}
Let $\psi \in (\h\otimes A)^*$ with $\psi \arrowvert_\h \in X^+$.  If $\g$ is either isomorphic to $\sl(n,n)$ with $n\ge 2$, or a finite-dimensional simple Lie superalgebra not of type $\lie q(n)$, and the triangular decomposition $\g = \nm \oplus \h \oplus \np$ satisfies \hyperref[C2]{$(\mathfrak C2)$}, then the local Weyl module $W_A^\loc(\psi)$ is finite dimensional.  \qed
\end{thm}

In what follows we give a necessary condition over triangular decompositions of $\g$ for local Weyl modules to be finite dimensional.  We begin with a technical lemma.

\begin{lemma}\label{trivial-weyl}
If $\psi \in (\ha)^*$ is such that $\psi \arrowvert_\h = \lambda \in X^+$, then there exists a finite-codimensional ideal $I \subseteq A$ such that $(\nm_0 \otimes I) W_A^\loc (\psi)_\lambda = 0$.
\end{lemma}

\begin{proof}
Let $\alpha \in R_{\lie r}^+$ and let $I_{\alpha}$ be the kernel of the linear map
\begin{align*}
A &\longrightarrow \Hom_\C \left( \g_{-\alpha} \otimes W_A^\loc(\psi)_\lambda, (\g_{-\alpha}\otimes A)w_\psi \right) \\
a &\longmapsto [u \otimes v \mapsto (u\otimes a)v],
\quad
a\in A, \ u\in \g_{-\alpha}, \ v\in W_A^\loc(\psi)_\lambda.
\end{align*}
Since $\g_{-\alpha}$ is finite dimensional for all $\alpha \in R^+_{\lie r}$, and since $W_A^\loc(\psi)_\lambda = \U\ha w_\psi = \C w_\psi$, Lemma~\ref{lem:fin-gen-even2} implies that $(\g_{-\alpha}\otimes A)w_\psi$ is finite dimensional.  Thus, $I_{\alpha}$ is a finite-codimensional linear subspace of $A$.  We claim that $I_{\alpha}$ is, in fact, an ideal of $A$.  Indeed, since $\alpha\neq 0$, we can fix $h\in\h$ such that $\alpha(h)\neq 0$.  Then, for all $a \in A$, $b \in I_{\alpha}$, $v\in W_A^\loc(\psi)_{\lambda}$, and $y \in \g_{-\alpha}$, we have
  \[
    0 = (h\otimes a)(y\otimes b)v
    =  [h\otimes a, y\otimes b]v + (y\otimes b)(h\otimes a)v
    =  -\alpha(h)(y\otimes ab)v + (y\otimes b)(h\otimes a)v.
  \]
  Since $(h\otimes a)v\in W_A^\loc(\psi)_{\lambda}$ and $b\in I_{\alpha}$, we have $(y \otimes b) (h \otimes a) v = 0$; and since we have assumed that $\alpha(h)$ is nonzero, this implies that $(y \otimes ab)v=0$.  As this holds for all $v\in W_A^\loc(\psi)_{\lambda}$ and $y\in \g_{-\alpha}$, we have that $ab \in I_{\alpha}$. Hence $I_{\alpha}$ is an ideal of $A$.

Let $I = \bigcap_{\alpha \in R^+_{\lie r}} I_{\alpha}$, and notice that $(\nm_0 \otimes I)W_A^\loc(\psi)_{\lambda}=0$.  Since $R^+_{\lie r}$ is a finite set, $I$ is an intersection of finitely many finite-codimensional ideals, and thus $I$ is also a finite-codimensional ideal of $A$.
\end{proof}

\begin{defn} \label{defn:I_psi}
For $\psi \in (\h \otimes A)^*$ with $\psi |_\h \in X^+$, let $I_\psi$ be the sum of all ideals $I\subseteq A$ such that $(\nm_0 \otimes I)w_\psi=0$. \end{defn}

\begin{rmk}\label{rem:ann-even-root}
It follows from Lemma~\ref{trivial-weyl}, that $I_\psi$ is a finite-codimensional ideal of $A$ and from Definitions~\ref{def:local-Weyl} and \ref{defn:I_psi}, that $((\nm_0 \oplus \n^+) \otimes I_\psi)w_\psi=0$. Furthermore, since $I_\psi$ has finite codimension and $A$ is assumed to be finitely generated, we have that $I_{\psi}^n$ has finite codimension, for all $n \in \N$ (see \cite[Lemma~2.1(a),(b)]{CLS}).
\end{rmk}

Given $\psi \in (\ha)^*$ such that $\psi \arrowvert_\h \in X^+$, let $w_\psi^\g$ denote the set
\[
\{ x \in \g \mid (x \otimes a) w_\psi = 0 \textup{ for all $a \in A$}\}.
\]
Notice that $w_\psi^\g$ is a subalgebra of $\g$ and $\np \subseteq w_\psi^\g$ (by Definition~\ref{def:local-Weyl}).

\begin{lemma} \label{lem:x_theta.ann.w_psi}
Let $\g$ be a finite-dimensional simple Lie superalgebra not of type $\q(n)$, and $\psi \in (\ha)^*$ be such that $\psi \arrowvert_\h = \lambda \in X^+$.  If $x^-_\theta$ is in the $\lie{r}$-submodule of $\g$ generated by $w_\psi^\g$, then there exists $n_\psi \in \N$ such that $\left( \g \otimes I_\psi^{n_\psi} \right) w_\psi = 0$.
\end{lemma}

\begin{proof}
Assume that $x^-_\theta$ is in the $\lie{r}$-submodule of $\g$ generated by $w_\psi^\g$.  Since $\lie{r}$ is a reductive Lie algebra, $x^-_\theta$ is in
\[
\vspan_\C \{ [x_{\beta_1}^-, [ \dots, [x_{\beta_k}^-, [x_{\gamma_1}, [ \dots, [x_{\gamma_\ell}, z] \dots ]]] \dots ]] \mid k, \ell \in \N, \beta_1, \dotsc, \beta_k, \gamma_1, \dotsc, \gamma_\ell \in R_\lie{r}^+, z \in w_\psi^\g \}.
\]
Now, since $x_\gamma \in w_\psi^\g$ for all $\gamma \in R^+_{\lie r}$, $w_\psi^\g$ is a subalgebra of $\g$ and $\lie r$ is finite dimensional, there exists $N \in \N$ such that
\[
x^-_\theta \in \vspan_\C \{ [x_{\beta_1}^-, [ \dots, [x_{\beta_k}^-, z] \dots ]] \mid k \le N, \beta_1, \dotsc, \beta_k \in R_\lie{r}^+, z \in w_\psi^\g \}.
\]
Thus, since $(x_\beta^ -\otimes b) w_\psi = 0$ for all $\beta \in R^+_\lie{r}$ and $b \in I_\psi$ (by Definition~\ref{defn:I_psi}), we see that $(x_\theta^- \otimes a) w_\psi = 0$ for all $a \in I_\psi^N$.  Hence, since $\g = \vspan_\C \{ [x_{\alpha_1}, [ \dotsm, [x_{\alpha_n}, x^-_{\theta}] \dotsm]] \mid n \in \N, \alpha_1, \dotsc, \alpha_n \in R^+ \}$, it follows that $(\g \otimes I_\psi^N) w_\psi = 0$.
\end{proof}

\begin{lemma} \label{lem:basic.x_theta}
Let $\g$ be a finite-dimensional simple Lie superalgebra not of type $\q(n)$, and $\psi \in (\ha)^*$ be such that $\psi \arrowvert_\h = \lambda \in X^+$.

\begin{enumerate}[leftmargin=*]
\item \label{x_theta.II}  If $\g$ is basic classical of type II, then, for every choice of triangular decomposition $\g = \nm \oplus \h \oplus \np$, $x^-_\theta$ is in the $\lie{r}$-submodule of $\g$ generated by $w_\psi^\g$.

\item \label{x_theta.I}  If $\g$ is basic classical of type I, then $x^-_\theta$ is in the $\lie{r}$-submodule of $\g$ generated by $w_\psi^\g$ if and only if the triangular decomposition $\g = \nm \oplus \h \oplus \np$ is not a parabolic one.
\end{enumerate}
\end{lemma}

\begin{proof}
Assume first that $\g$ is basic classical of type II.  Since $(\n^+ \otimes A) w_\psi=0$, we have $(z' \otimes a) w_\psi = 0$ for all $z' \in \g_\alpha$, $\alpha \in R^+$ and $a \in A$.  Since $\lie r = \g_{\bar0}$ is a finite-dimensional simple Lie algebra and $\g_{\bar1}$ is an irreducible $\lie{r}$-module, then there exist $\alpha \in R^+$ and $z' \in \g_\alpha$ such that
\[
x_\theta^- \in \vspan_\C \{ [x_{\beta_1}^-, [\dots, [x_{\beta_k}^-,[x_{\gamma_1},[\dots, [x_{\gamma_\ell}, z'] \dots ]]] \dots ]] \mid k, \ell \in \N, \ \beta_1, \dotsc, \beta_k, \gamma_1, \dotsc, \gamma_\ell \in R^+_{\lie r} \}.
\]
Let $z = [x_{\gamma_1}, [\dots ,[x_{\gamma_\ell}, z'] \dots ]]$, and notice that $(z \otimes a) w_\psi = 0$ for all $a \in A$.

Now assume that $\g$ is basic classical of type I.  Recall that in these cases $\g$ admits a $\Z$-grading $\g_{-1}\oplus\g_0\oplus \g_{1}$, $\g_{\bar 0}=\g_0=\lie{r}$, and $\g_{\bar 1} = \g_{-1}\oplus \g_1$, where $\g_1$ and $\g_{-1}$ are irreducible $\lie{r}$-modules.  If we choose a triangular decomposition $\g = \nm \oplus \h \oplus \np$ that is not a parabolic one, for each $i \in \{-1, 0, 1\}$, there exists $\beta_i \in R^+$ such that $\g_{\beta_i} \subseteq \g_i \cap \n^+$.  In particular, $\g_{\beta_{-1}}, \g_{\beta_0}, \g_{\beta_1} \subseteq w_\psi^\g$.  Since $x^-_\theta$ is not in the center of $\g$, and since $\g_{-1}$, $\lie r'$ and $\g_1$ are irreducible $\lie{r}$-modules, it follows that $x^-_\theta$ is in the $\lie r$-submodule of $\g$ generated by $\g_{\beta_{-1}} \oplus \g_{\beta_0} \oplus \g_{\beta_1} \subseteq w_\psi^\g$.

Conversely, if $\g=\n^-\oplus \h\oplus \n^+$ is a parabolic triangular decomposition, then $x_\theta^-$ belongs to either $\g_{-1}$ or $\g_1$. In any case, we have that $x_\theta^-$ is not in the $\lie{r}$-submodule of $\g$ generated by $w_\psi^\g$ (in fact, if $x_\theta^-\in \g_{\pm 1}$, then the $\lie{r}$-submodule of $\g$ generated by $w_\psi^\g$ is $\g_0\oplus \g_{\mp 1}$).
\end{proof}

This next result gives, for basic classical Lie superalgebras, a necessary and sufficient condition for a local Weyl module to be finite dimensional, and when $\g$ is either $\lie{p}(n)$ or of Cartan type, a sufficient condition and a necessary condition for a local Weyl module to be finite dimensional.

\begin{thm} \label{thm:loc.weyl.fd}
Let $\g$ be a finite-dimensional simple Lie superalgebra not of type $\q(n)$ with a triangular decomposition $\g = \nm \oplus \h \oplus \np$, let $\psi \in (\ha)^*$ be such that $\psi \arrowvert_\h \in X^+$, and let $A$ be infinite dimensional.

\begin{enumerate}[leftmargin=*]
\item \label{loc.weyl.fd.II}
If $\g$ is basic classical of type II, then $W_A^\loc(\psi)$ is finite dimensional (for every triangular decomposition).

\item \label{loc.weyl.fd.I}
If $\g$ is basic classical of type I, then $W_A^\loc(\psi)$ is finite dimensional if and only if the triangular decomposition is not a parabolic one.

\item \label{loc.weyl.fd.cartan}
If $\g$ is either of type $\lie{p}(n)$ or of Cartan type, and the $x^-_\theta$ is in the $\lie{r}$-submodule of $\g$ generated by $w_\psi^\g$, then $W_A^\loc(\psi)$ is finite-dimensional.

\item \label{loc.weyl.infd.cartan}
If $\g$ is either of type $\lie{p}(n)$ or of Cartan type, and the triangular decomposition of $\g$ is parabolic, then $W_A^\loc(\psi)$ is infinite-dimensional.
\end{enumerate}
\end{thm}

\begin{proof}
The proofs that $W_A^\loc (\psi)$ are finite dimensional (that is, item~\eqref{loc.weyl.fd.II}, the \emph{if} part of item~\eqref{loc.weyl.fd.I} and item~\eqref{loc.weyl.fd.cartan}) follow from Lemmas~\ref{lem:x_theta.ann.w_psi}~and~\ref{lem:basic.x_theta} using standard arguments.

To prove the \emph{only if} part of item~\eqref{loc.weyl.fd.I} and item~\eqref{loc.weyl.infd.cartan}, suppose that $\g$ is either basic classical of type I, or of type $\lie p(n)$, or of Cartan type, and that the triangular decomposition $\g = \nm \oplus \h \oplus \np$ is parabolic, that is, $(\lir + \np)$ is a Lie subsuperalgebra of $\g$ and there exists a nontrivial subspace $\g^- \subseteq \g$ such that $\g = \g^-\oplus (\lie{r} + \np)$.  Thus, we can consider the $(\lie{r}+\np)\otimes A$-module $W^{\lir}(\psi)$ defined to be the quotient of $\U{(\lir + \np) \otimes A}$ by the left ideal generated by
\[
\g_\alpha\otimes A \quad \textup{for all } \alpha\in R^+, \qquad
h - \psi(h) \quad \textup{for all } h \in \ha, \qquad
(x_\alpha^{-})^{\psi(h_\alpha)+1} \quad \textup{for all } \alpha \in \dr.
\]
Notice that the image of $1 \in \U{(\lir + \np) \otimes A}$, which we will denote by $u_\psi$, generates $W^{\lir}(\psi)$.

Now, let
	\[
\overline{W^\lir}(\psi) = \ind{(\lie{r}+\np)\otimes A}{\ga} W^\lir (\psi),
	\]
and notice that $\overline{W^\lir}(\psi)$ is generated by $1 \otimes u_\psi$.  Moreover, since $A$ is assumed to be infinite dimensional, we have that $\g^- \otimes A$ is infinite dimensional, which implies that $\overline{W^\lir}(\psi)$ is infinite dimensional.  Finally, notice that there exists a unique surjective homomorphism of $\ga$-modules $W_A^\loc (\psi) \to \overline{W^\lir}(\psi)$ satisfying $w_\psi \mapsto (1 \otimes u_\psi)$.  Since $\overline{W^\lir}(\psi)$ is infinite dimensional, we conclude that $W_A^\loc (\psi)$ is also infinite dimensional.
\end{proof}

\begin{ex}
Let $\g$ be a simple Lie superalgebra of Cartan type, and let $\psi \in (\ha)^*$ be such that $\psi \arrowvert_\h \in X^+$.  If one chooses either the maximal or the minimal triangular decomposition (see Subsection~\ref{tri.dec.cartan}), then $x_\theta^-$ is not in the $\lie r$-submodule generated by $w_\psi^\g$.  In these cases, $W_A^\loc (\psi)$ will \emph{not} be finite dimensional.  On the other hand, if one chooses a triangular decomposition of $\g$ satisfying \hyperref[C2]{$(\mathfrak C2)$} (see Proposition~\ref{prop:nice-system-exists-cartan}), then $W_A^\loc (\psi)$ is finite dimensional.
\end{ex}

\begin{rmk}
Traditionally, local Weyl modules are universal objects in certain categories of finite-dimensional modules (see, for instance, \cite[Proposition~2.1(iii)]{CP01}, \cite[Theorem~5]{FL04}, \cite[Proposition~5]{CFK10}, \cite[Corollary~4.6]{FKKS12} and \cite[Proposition~4.13]{CLS}).  In the current setting, we have proved in Proposition~\ref{prop:local.weyl.univ} that the local Weyl module $W_A^\loc (\psi)$ is a universal object in the category $\cal C_A^\lambda$ ($\lambda = \psi \arrowvert_\h$).  However, if $\g$ is a Lie superalgebra of type I, $\mathfrak p(n)$ or Cartan, with a parabolic triangular decomposition, then local Weyl modules are infinite-dimensional (see Theorem~\ref{thm:loc.weyl.fd}\eqref{loc.weyl.fd.I},~\eqref{loc.weyl.infd.cartan}).  In these cases, there is no finite dimensional $\ga$-module of highest-weight $\lambda$ of which every finite dimensional $\ga$-module of highest-weight $\lambda$ is a quotient.

We illustrate this claim with a concrete example.  Let $A = \C[t]$, $\g$ be a Lie superalgebra of type I with a distinguished triangular decomposition (which is parabolic), and let $\psi = 0$.  Notice that, in this case, $W_{\C[t]}^\loc (\psi)$ is free as a left $\U{\n_{\bar1}^- \otimes \C[t]}$-module.  Now, for every $k > 0$, consider the $\g \otimes \C[t]$-module $W_k$ given as the quotient of $\U{\g \otimes \C[t]}$ by the left ideal generated by
\[
\np \otimes \C[t], \quad
\h \otimes \C[t], \quad
x_\alpha^-,\quad
y \otimes t^k, \quad
\textup{for all $\alpha \in \dr$ and $y \in \n_{\bar1}^-$}.
\]
Notice that $W_k \in \cal C^\lambda_{\C[t]}$ is a quotient of $W_{\C[t]}^\loc(\psi)$ and that $\dim W_k = 2^{k \dim \nm_{\bar1}}$ for all $k \ge 0$.  Since there is no upper bound for $k$, we see that there is no finite-dimensional $\g \otimes \C[t]$-module of highest-weight $\lambda$ that projects onto $W_k$ for all $k \ge 0$.
\end{rmk}

\begin{cor}
Let $\g$ be either of type $\lie{p}(n)$ or a basic classical Lie superalgebra, and let ${\cal L} (\h\otimes A)=\{\psi\in (\h\otimes A)^* \mid \psi(\h\otimes I)=0 \text{ for some finite-codimensional ideal $I$ of $A$}\}$.

\begin{enumerate}[leftmargin=*]
\item If $\g$ is basic classical of type II, then $W_A^\loc(\psi)=0$ if $\psi\notin {\cal L} (\ha)$.

\item If $\g$ is either of type $\lie{p}(n)$ or basic classical of type I, and the triangular decomposition is not a parabolic one, then $W_A^\loc(\psi)=0$ if $\psi\notin {\cal L} (\ha)$.

\item If $\g$ is of Cartan type and $x^-_\theta$ is in the $\lie{r}$-submodule of $\g$ generated by $w_\psi^\g$, then $W_A^\loc(\psi) = 0$ if $\psi \not\in \cal L(\ha)$.
\end{enumerate}
\end{cor}

\begin{proof}
In each one of these cases, $W_A^\loc(\psi)$ is finite dimensional.  Thus, there exists a finite codimensional ideal $I$ of $A$ such that $(\g\otimes I)w_\psi=0$.  In particular, $(h\otimes a)w_\psi = \psi(h\otimes a)w_\psi = 0$ for all $a \in I$.  If $\psi \notin {\cal L} (\h\otimes A)$, then there exists $a\in I$ such that $\psi(h\otimes a) \neq 0$.   In this case, $w_\psi = 0$.  Thus $W_A^\loc(\psi)=\bu(\n^-\otimes A)w_\psi=0$.
\end{proof}

We finish this section with two results regarding tensor products of local Weyl modules.  They generalize well-known results.

\begin{lemma}\label{lem:fin-cod.ann}
Let $J_\psi$ be the sum of all ideals $I \subseteq A$ such that $(\g\otimes I)W(\psi)=0$.  Then $J_\psi$ is a finite-codimensional ideal of $A$.
\end{lemma}

\begin{proof}
By \cite[Proposition~8.1]{Sav14}, all ideals of $\g\otimes A$ are of the form $\g\otimes I$, where $I$ is an ideal of $A$. In particular, the annihilator of the action of $\g\otimes A$ on $W_A^\loc(\psi)$ is of the form $\g\otimes I$, for some ideal $I$ of $A$. Since $W_A^\loc(\psi)$ is finite dimensional and $(\g\otimes A)/(\g\otimes I)\cong \g\otimes A/I$, we see that $I$ must be a finite-codimensional ideal of $A$. Now the result follows from the fact that $I\subseteq J_\psi$.
\end{proof}

Given an ideal $I$ of $A$, we define its support to be the set $\Supp(I)=\{\mathsf{m}\in\MaxSpec(A)\mid I\subseteq \mathsf{m}\}$. The next result generalizes \cite[Theorem~4.15]{CLS}.

\begin{prop} \label{prop:local.weyl.tensor}
Let $\psi,\ \varphi\in(\h \otimes A)^*$, $\psi |_\h=\lambda,\ \varphi |_\h=\mu$, and suppose that $\lambda,\ \mu\in X^+$ are such that $\lambda+\mu \in X^+$. If $\Supp(J_\psi)\cap \Supp(J_\varphi)=\emptyset$, then (omitting the pull back maps) we have
	\[
W_A^\loc(\psi+\varphi)\cong W_A^\loc(\psi)\otimes W_A^\loc(\varphi),
	\]
as $\ga$-modules.
\end{prop}

\begin{proof}
Using the fact that $\Supp(J_\psi)\cap \Supp(J_\varphi)=\emptyset$, one can prove that the action of $\ga$ on the tensor product $W_A^\loc(\psi)\otimes W_A^\loc(\varphi)$ descends to an action of $\g\otimes (A/J_\psi\oplus A/J_\varphi)$ on $W_A^\loc(\psi)\otimes W_A^\loc(\varphi)$. By Lemma~\ref{lem:fin-cod.ann}, both algebras $A/J_\psi$ and $A/J_\varphi$ are finite dimensional. Thus we can use Corollary~\ref{cor:tensor.product.local}.  The result follows from Theorem~\ref{thm:tensor-property2}.
\end{proof}

%
\appendix
\section{Homological properties of Weyl functors}
%

The results of this section show that the Weyl functors defined in the current paper satisfy properties similar to the ones satisfied by Weyl functors defined in the non-super setting. Since the proofs of ther results of this appendix are very similar to those in the non-super setting, we refer to \cite[\S3.7]{CFK10} and \cite[\S4]{FMS15} for the details.

Throughout this appendix, we assume that $\g$ is either $\sl(n,n)$ with $n\geq 2$, or a finite-dimensional simple Lie superalgebra not of type $\lie q(n)$, endowed with a triangular decomposition satisfying \hyperref[C1]{$(\mathfrak C1)$}.  We will also assume that $A$, $B$ and $C$ are associative, commutative $\C$-algebras with unit.  

\begin{prop}\label{prop:relating-functors}
Let $\lambda \in X^+$.
\begin{enumerate}[leftmargin=*]
\item \label{prop-item:iso-iden}
For every ${\fA_\lambda}$-module $M$, there is an isomorphism of $\fA_\lambda$-modules $\fR_A^\lambda \fW_A^\lambda M \cong M$ that is functorial in $M$.

\item \label{prop-item:left-adjoint}
$\fW_A^\lambda \colon \mod{\fA_\lambda} \to \cal C_A^\lambda$ is left adjoint to $\fR_A^\lambda \colon \cal C_A^\lambda \to \mod{\fA_\lambda}$.

\item \label{weyl.func.full.faith}
$\fW_A^\lambda$ is fully faithful.

\item \label{prop-item:proj-to-proj}
If $M$ is a projective $\fA_\lambda$-module, then $\fW_A^\lambda M$ is a projective object in $\cal C_A^\lambda$.
\end{enumerate}
\end{prop}

\dproof{
First recall from Definition~\ref{dfn.super.Weyl.func}, \eqref{eq:weight.space.W} and \eqref{def:RAlambda} that
\[
\fR_A^\lambda\fW_A^\lambda M
= \fR_A^\lambda \left( W_A(\lambda) \otimes_{\fA_\lambda} M \right)
= \left( W_A(\lambda) \otimes_{\fA_\lambda} M \right)_\lambda
= W_A(\lambda)_\lambda \otimes_{\fA_\lambda} M.
\]
Now, there is an isomorphism of $\fA_\lambda$-modules $\phi \colon \fA_\lambda \to W_A(\lambda)_\lambda$; namely, $\phi(u) = uw_\lambda$ for all $u \in \fA_\lambda$.  Thus $\phi \otimes \id_M \colon  \fA_\lambda \otimes_{\fA_\lambda} M \to W_A(\lambda)_\lambda \otimes_{\fA_\lambda} M$ is an isomorphism of $\fA_\lambda$-modules that is functorial in $M$.  Finally, notice that the unique homomorphism of $\fA_\lambda$-modules satisfying
\[
\fA_\lambda \otimes_{\fA_\lambda} M \to M, \quad
u \otimes m \mapsto um, \quad
\textup{for all $u \in \fA_\lambda$ and $m \in M$},
\]
is in fact an isomorphism that is functorial in $M$. Hence the result of part~\eqref{prop-item:iso-iden} follows.

To prove part~\eqref{prop-item:left-adjoint}, let $M$ be an $\fA_\lambda$-module and $N$ be an object of $\cal C_A^\lambda$.  By  Definition~\ref{dfn.super.Weyl.func}, the tensor-hom adjunction, and Lemma~\ref{lem:RAlambda.iso}, we have
\begin{align*}
\Hom_{\cal C_A^\lambda} \left( \fW_A^\lambda M, N \right)
&= \Hom_{\U\ga} \left( W_A (\lambda) \otimes_{\fA_\lambda} M, N \right) \\
&\cong \Hom_{\fA_\lambda} \left( M, \Hom_{\U\ga} \left( W_A (\lambda), N \right) \right) \\
&\cong \Hom_{\fA_\lambda} \left( M, N_\lambda \right) \\
&= \Hom_{\fA_\lambda} \left( M, \fR_A^\lambda N \right),
\end{align*}
where the isomorphisms are functorial in $M$ and $N$.

Part~\eqref{weyl.func.full.faith} follows from part~\eqref{prop-item:left-adjoint} and the fact that $\fR_A^\lambda \fW_A^\lambda N \cong N$ for every $\fA_\lambda$-module $N$.  In fact,
\[
\Hom_{\fA_\lambda} (M, N)
\cong \Hom_{\fA_\lambda} (M, \fR_A^\lambda \fW_A^\lambda N)
\cong \Hom_{\cal C_A^\lambda} (\fW_A^\lambda M, \fW_A^\lambda N)
\]
for all $\fA_\lambda$-modules $M, N$.

Part~\eqref{prop-item:proj-to-proj} follows from part~\eqref{prop-item:left-adjoint}, since $\fW_A^\lambda$ is left adjoint to an exact functor.
}

\begin{cor}
For each $\lambda \in X^+$, the module $W_A(\lambda)$ is projective in $\cal C_A^\lambda$ and the module $K(\lambda)$ is projective in $\cal C_\C^\lambda$.  Moreover, there is an isomorphism of algebras $\Hom_{\cal C_A^\lambda} \left( W_A (\lambda), W_A (\lambda) \right) \cong \fA_\lambda$.
\end{cor}

\dproof{
The first statement follows from Proposition~\ref{prop:relating-functors}\eqref{prop-item:proj-to-proj} and the fact that $\fA_\lambda$ is a free (therefore projective) $\fA_\lambda$-module.  The \emph{moreover} statement follows from Lemma~\ref{lem:RAlambda.iso}.
}

Notice that, despite $\fW_A^\lambda$ being a fully faithful functor (by Proposition~\ref{prop:relating-functors}\eqref{weyl.func.full.faith}), it is \emph{not} an equivalence of categories, as it is not essentially surjective.  In fact, if $\mu < \lambda$, then $W_A (\mu)$ is an object of $\cal C_A^\lambda$ for which there exists no $\fA_\lambda$-module $N$ satisfying $\fW_A^\lambda N \cong W_A(\mu)$.  (If $\fW_A^\lambda N \cong W_A(\mu)$, then $N \cong \fR_A^\lambda \fW_A^\lambda N \cong \fR_A^\lambda W_A(\mu) = 0$.)  Theorem~\ref{thm:hom-ext1-vanish} describes for which objects $M$ of $\cal C_A^\lambda$ there exists an $\fA_\lambda$-module $N$ satisfying $\fW_A^\lambda N \cong M$.

\details{
\begin{lemma} \label{lem:hom-vanish}
An object $M \in \cal C_A^\lambda$ satisfies $M = \U\ga M_\lambda$ if and only if, for each object $N$ of $\cal C_A^\lambda$ that satisfies $N_\lambda = 0$, we have $\Hom_{\cal C_A^\lambda} \left( M, N \right) = 0$.
\end{lemma}

\begin{proof}
First assume that, for each object $N$ of $\cal C_A^\lambda$ that satisfies $N_\lambda = 0$, we have $\Hom_{\cal C_A^\lambda} \left( M, N \right) = 0$.  Denote the submodule $\U\ga M_\lambda \subseteq M$ by $M'$, and consider the short exact sequence
\[
0 \to M' \inj M \to M / M' \to 0.
\]
Since $\fR_A^\lambda$ is an exact functor, $0 \to M'_\lambda \inj M_\lambda \to (M / M')_\lambda \to 0$ is also an exact sequence.  Now notice that by construction, $M'_\lambda = M_\lambda$.  Hence $(M / M')_\lambda = 0$.  This implies, by hypothesis, that $\Hom_{\cal C_A^\lambda} \left( M, M/M' \right) = 0$.  Thus $M' = M$.

Now assume that $\U\ga M_\lambda = M$ and let $N$ be an object $N$ of $\cal C_A^\lambda$ that satisfies $N_\lambda = 0$.  Given a homomorphism of $\ga$-modules $\phi \colon M \to N$, consider the short exact sequence
\[
0 \to \ker \phi \to M \pra{\phi} \im \phi \to 0.
\]
Since $\fR_A^\lambda$ is an exact functor, $0 \to (\ker \phi)_\lambda \to M_\lambda \pra{\phi_\lambda} (\im \phi)_\lambda \to 0$ is also an exact sequence.  Since $N_\lambda$ is assumed to be zero, we have that $(\im \phi)_\lambda = 0$ and $\phi_\lambda = 0$.  Since $M = \U\ga M_\lambda$, the homomorphism $\phi$ is uniquely determined by $\phi_\lambda$.  Thus $\phi = 0$, and the result follows.
\end{proof}

\begin{rmk} \label{rmk:hom.vanish}
Let $V$ be an $\fA_\lambda$-module.  First notice that, by Proposition~\ref{prop:relating-functors}\eqref{prop-item:iso-iden}, we have $\fW_A^\lambda V \cong \fW_A^\lambda \fR_A^\lambda \fW_A^\lambda V$.  Moreover, notice that $\fW_A^\lambda V = \U\ga \left( \fW_A^\lambda V \right)_\lambda$.  Thus, by Lemma~\ref{lem:hom-vanish}, for every object $N$ of $\cal C_A^\lambda$ that satisfies $N_\lambda = 0$, we have $\Hom_{\cal C_A^\lambda} \left( \fW_A^\lambda V, N \right) = 0$.  This observation will be used in the proofs of the next two results.
\end{rmk}}

\begin{thm} \label{thm:hom-ext1-vanish}
Let $M$ be an object of $\cal C_A^\lambda$. Then $M \cong \fW_A^\lambda \fR_A^\lambda M$ if and only if, for each object $N$ of ${\cal C}_A^\lambda$ that satisfies $N_\lambda = 0$, we have
\[
\Hom_{\cal C_A^\lambda} \left( M, N \right) = \Ext_{{\cal C}_A^\lambda}^1 \left( M, N \right) = 0.
\]
\end{thm}

\dproof{
First assume that $M$ is an object of ${\cal C}_A^\lambda$ such that, for each $N \in {\cal C}_A^\lambda$ that satisfies $N_\lambda = 0$, we have $\Hom_{\cal C_A^\lambda} \left( M, N \right) = \Ext_{{\cal C}_A^\lambda}^1 (M, N) = 0$.  Using Proposition~\ref{prop:glob.Weyl.genrel+univ},  we see that there is a unique homomorphism of $\ga$-modules $\epsilon_M \colon \fW_A^\lambda \fR_A^\lambda M \to M$ satisfying
\[
\epsilon_M (u w_\lambda \otimes m_\lambda) = u m_\lambda
\quad \textup{for all $u \in \U\ga$ and $m_\lambda \in M_\lambda$}.
\]
Since $\Hom_{\cal C_A^\lambda} \left( M, N \right) = 0$ for every object $N$ of ${\cal C}_A^\lambda$ that satisfies $N_\lambda = 0$, by Lemma~\ref{lem:hom-vanish}, we have that $M = \U\ga M_\lambda$.  Hence, the homomorphism $\epsilon_M$ is surjective, and we have a short exact sequence
\[
0 \to \ker \epsilon_M \to \fW_A^\lambda \fR_A^\lambda M \to M \to 0.
\]
Consider the associated long exact sequence on $\Ext^\bullet_{\cal C_A^\lambda} \left( - , \ker\epsilon_M \right)$.  In particular, we have:
\[
\cdots \to
\Hom_{\cal C_A^\lambda} \left( \fW_A^\lambda \fR_A^\lambda M, \ker \epsilon_M \right) \to
\Hom_{\cal C_A^\lambda} \left( \ker \epsilon_M, \ker \epsilon_M \right) \to
\Ext^1_{\cal C_A^\lambda} \left( M, \ker \epsilon_M \right) \to \cdots
\]
Now, notice that $(\ker \epsilon_M)_\lambda = 0$, as $\epsilon_M$ maps $\left( \fW_A^\lambda \fR_A^\lambda M \right)_\lambda$ isomorphically onto $M_\lambda$.  This implies by Remark~\ref{rmk:hom.vanish}, that $\Hom_{\cal C_A^\lambda} \left( \fW_A^\lambda \fR_A^\lambda M , \ker \epsilon_M \right) = 0$, and by hypothesis, that $\Ext^1_{\cal C_A^\lambda} \left( M , \ker \epsilon_M \right) = 0$.  Hence, $\Hom_{{\cal C}_A^\lambda} (\ker \epsilon_M, \ker \epsilon_M) = 0$.  Thus $\ker \epsilon_M = 0$, and $\epsilon_M$ is an isomorphism between $\fW_A^\lambda \fR_A^\lambda M$ and $M$.

Now, to prove the converse, assume that $M \cong \fW_A^\lambda \fR_A^\lambda M$.  It follows from Remark~\ref{rmk:hom.vanish} that, for each object $N$ of ${\cal C}_A^\lambda$ that satisfies $N_\lambda = 0$, we have $\Hom_{\cal C_A^\lambda} \left( M, N \right) = 0$.  Now, recall that the category $\mod{\fA_\lambda}$ has enough projectives, that is, there exist a projective ${\fA_\lambda}$-module $P$ and a surjective homomorphism of ${\fA_\lambda}$-modules $\pi \colon P \to M_\lambda$.  Since the functor $\fW_A^\lambda$ is right exact, $\fW_A^\lambda \pi$ is an even homomorphism, and $M \cong \fW_A^\lambda \fR_A^\lambda M$, we have a short exact sequence of of $\ga$-modules
\[
0 \to
\ker (\fW_A^\lambda \pi) \to
\fW_A^\lambda P \pra{\fW_A^\lambda \pi}
M \to 0,
\]
where $\ker (\fW_A^\lambda \pi)$ is the image of $\fW_A^\lambda (\ker \pi)$ inside $\fW_A^\lambda P$.  Let $N$ be an object of ${\cal C}_A^\lambda$ that satisfies $N_\lambda = 0$, and consider the associated long exact sequence on $\Ext^\bullet_{\cal C_A^\lambda} \left( - , N \right)$.  In particular, we have:
\[
\cdots \to
\Hom_{\cal C_A^\lambda} \left( \ker (\fW_A^\lambda \pi), N \right) \to
\Ext^1_{\cal C_A^\lambda} \left( M, N \right) \to
\Ext^1_{\cal C_A^\lambda} \left( \fW_A^\lambda P, N \right) \to \cdots
\]
By Proposition~\ref{prop:relating-functors}\eqref{prop-item:proj-to-proj}, $\fW_A^\lambda P$ is a projective object in $\cal C_A^\lambda$.  Hence, $\Ext^1_{\cal C_A^\lambda} \left( \fW_A^\lambda P, N \right) = 0$.  Since $\ker (\fW_A^\lambda \pi)$ is the image of $\fW_A^\lambda (\ker \pi)$ inside $\fW_A^\lambda P$, we have that $\ker (\fW_A^\lambda \pi) = \U\ga \ker (\fW_A^\lambda \pi)_\lambda$.  Hence, by Remark~\ref{rmk:hom.vanish}, we also have that $\Hom_{\cal C_A^\lambda} \left( \ker (\fW_A^\lambda \pi), N \right) = 0$.  The result follows.
}

\details{
Recall that $\fW_A^\lambda \colon \mod{\fA_\lambda} \to \cal C_A^\lambda$ is a right exact functor, that $\mod{\fA_\lambda}$ has enough projectives, and that projective $\fA_\lambda$-modules are left acyclic for $\fW_A^\lambda$.  Thus, we can consider the left derived functors $\L_n \fW_A^\lambda$, $n \ge 0$.  In particular, $\L_0 \fW_A^\lambda = \fW_A^\lambda$ and $\fW_A^\lambda$ is exact if and only if $\L_1 \fW_A^\lambda = 0$.
}

\begin{cor} \label{cor:eq:vanishing.condition.2}
The functor $\fW_A^\lambda$ is exact if and only if, for each object $N$ of ${\cal C}_A^\lambda$ that satisfies $N_\lambda = 0$, we have
\[
\Ext^2_{{\cal C}_A^\lambda} (\fW_A^\lambda -, N) = 0.
\]
\end{cor}

\dproof{
Assume that $\Ext^2_{{\cal C}_A^\lambda} (\fW_A^\lambda -, N) = 0$, and let $V$ be an $\fA_\lambda$-module.  Recall that the category $\mod{\fA_\lambda}$ has enough projectives, that is, there exist a projective ${\fA_\lambda}$-module $P$ and an exact sequence of ${\fA_\lambda}$-modules $0 \to \ker \pi \pra{i} P \pra{\pi} V \to 0$.  Thus, by the definition of $\L_1\fW_A^\lambda$, we have an exact sequence of $\ga$-modules
\[
0 \to
\L_1 \fW_A^\lambda (V) \to
\fW_A^\lambda (\ker \pi) \pra{\fW_A^\lambda i}
\fW_A^\lambda P \pra{\fW_A^\lambda \pi}
\fW_A^\lambda V \to 0.
\]
Denote $\ker (\fW_A^\lambda \pi) = \im (\fW_A^\lambda i)$ by $K$, consider the associated short exact sequences:
\begin{gather}
0 \to \L_1 \fW_A^\lambda (V) \to \fW_A^\lambda (\ker \pi) \pra{\fW_A^\lambda i} K \to 0, \label{eq:ass.ses01} \\
0 \to K \inj \fW_A^\lambda P \pra{\fW_A^\lambda \pi} \fW_A^\lambda V \to 0, \label{eq:ass.ses02}
\end{gather}
and notice, by applying the exact functor $\fR_A^\lambda$ to \eqref{eq:ass.ses02}, that $K_\lambda \cong \ker \pi$.

Let $N$ be an object of ${\cal C}_A^\lambda$ that satisfies $N_\lambda = 0$.  The long exact sequence on $\Ext_{{\cal C}_A^\lambda}^\bullet (-, N)$ associated to \eqref{eq:ass.ses02} gives, in particular,
\begin{alignat*}{3}
\cdots
&\to \Hom_{\cal C_A^\lambda} \left( \fW_A^\lambda P, N \right)
&&\to \Hom_{\cal C_A^\lambda} \left( K, N \right)
&&\to \Ext^1_{\cal C_A^\lambda} \left( \fW_A^\lambda V, N \right) \to \\
&\to \Ext^1_{\cal C_A^\lambda} \left( \fW_A^\lambda P, N \right)
&&\to \Ext^1_{\cal C_A^\lambda} \left( K, N \right)
&&\to \Ext^2_{\cal C_A^\lambda} \left( \fW_A^\lambda V, N \right) \to \cdots
\end{alignat*}
By Lemma~\ref{lem:hom-vanish} and Theorem~\ref{thm:hom-ext1-vanish}, we have that $\Hom_{\cal C_A^\lambda} \left( \fW_A^\lambda P, N \right) = \Ext^1_{\cal C_A^\lambda} \left( \fW_A^\lambda V, N \right) = 0$.  Hence $\Hom_{\cal C_A^\lambda} \left( K, N \right) = 0$.  By Proposition~\ref{prop:relating-functors}\eqref{prop-item:proj-to-proj} and the hypothesis, we have that $\Ext^1_{\cal C_A^\lambda} \left( \fW_A^\lambda P, N \right) = \Ext^2_{\cal C_A^\lambda} \left( \fW_A^\lambda V, N \right) = 0$.  Hence, $\Ext^1_{\cal C_A^\lambda} \left( K, N \right) = 0$.  Thus, by Theorem~\ref{thm:hom-ext1-vanish}, we have that
\[
K \cong \fW_A^\lambda \fR_A^\lambda K \cong \fW_A^\lambda (\ker \pi).
\]
That is, $\fW_A^\lambda i \colon \fW_A^\lambda (\ker \pi) \to K$ is an isomorphism, and using \eqref{eq:ass.ses01}, we see that $\L_1\fW_A^\lambda (V) = 0$.

Now, to prove the converse, assume that $\fW_A^\lambda$ is exact, and let $V$ be an $\fA_\lambda$-module.  Since the category $\mod{\fA_\lambda}$ has enough projectives, there exist a projective ${\fA_\lambda}$-module $P$ and a surjective homomorphism of ${\fA_\lambda}$-modules $\pi \colon P \to V$.  Since the functor $\fW_A^\lambda$ is assumed to be exact and $\fW_A^\lambda \pi$ is an even homomorphism we have a short exact sequence of of $\ga$-modules
\[
0 \to
\fW_A^\lambda (\ker \pi) \to
\fW_A^\lambda P \to
\fW_A^\lambda V \to 0.
\]
Now let $N$ be an object of ${\cal C}_A^\lambda$ that satisfies $N_\lambda = 0$, and consider the associated long exact sequence on $\Ext^\bullet_{\cal C_A^\lambda} \left( - , N \right)$.  In particular, we have:
\[
\cdots \to
\Ext^1_{\cal C_A^\lambda} \left( \fW_A^\lambda (\ker \pi), N \right) \to
\Ext^2_{\cal C_A^\lambda} \left( \fW_A^\lambda V, N \right) \to
\Ext^2_{\cal C_A^\lambda} \left( \fW_A^\lambda P, N \right) \to \cdots
\]
By Theorem~\ref{thm:hom-ext1-vanish}, $\Ext^1_{\cal C_A^\lambda} \left( \fW_A^\lambda (\ker \pi), N \right) = 0$, and by Proposition~\ref{prop:relating-functors}\eqref{prop-item:proj-to-proj}, $\Ext^2_{\cal C_A^\lambda} \left( \fW_A^\lambda P, N \right) = 0$.  Thus, as we wanted to prove, $\Ext_{{\cal C}_A^\lambda}^2 \left( \fW_A^\lambda V, N \right) = 0$.
}

Similar to the definition of $\fA_\lambda$ in Section \ref{S:super.weyl.functors}, for each $\lambda \in X^+$, let
\[
\fB_\lambda = \bu(\h \otimes B)/\Ann_{\h\otimes B} (w_\lambda)
\quad \textup{ and } \quad
\fC_\lambda = \bu(\h\otimes C)/\Ann_{\h\otimes C} (w_\lambda).
\]

\begin{rmk} \ \label{rmk:actions}
\begin{enumerate}[leftmargin=*, itemsep=.75ex]
\item \label{rmk:actions.pi}
Every homomorphism of $\C$-algebras $\pi \colon C \to A$ induces a unique (even) homomorphism of Lie superalgebras (which we denote by the same symbol) $\pi \colon \g \otimes C \to \g \otimes A$ satisfying
\[
\pi (x \otimes c) = x \otimes \pi(c)
\quad \textup{ for all } x \in \g \textup{ and } c \in C.
\]
This latter homomorphism induces an action of $\g\otimes C$ on any $\g\otimes A$-module $M$ via the pull-back.  Let $\pi^*M$ denote such a $\g\otimes C$-module.

\item \label{rmk:actions.ann}
Let $\lambda \in X^+$ and $\pi \colon C \to A$ be a homomorphism of $\C$-algebras.  Using item~\eqref{rmk:actions.pi}, we see that $\pi$ also induces a homomorphism of associative superalgebras (which we  keep denoting by the same symbol), $\pi \colon \U{\g \otimes C} \to \U{\g \otimes A}$.  Notice that by construction, $\pi(\n^+ \otimes C) \subseteq \n^+ \otimes A$, $\pi(h) - \lambda(h) = h - \lambda(h)$ for all $h \in \h$, and $\pi(x_\alpha^-)^k = (x_\alpha^-)^k$ for all $\alpha \in \dr$ and $k \ge 0$.  Hence
\[
\pi \left( \Ann_{\g \otimes C} (w_\lambda) \right) \subseteq \Ann_{\g \otimes A} (w_\lambda)
\quad \textup{ and } \quad
\pi \left( \Ann_{\h \otimes C} (w_\lambda) \right) \subseteq \Ann_{\h \otimes A} (w_\lambda).
\]
Thus $\pi$ induces a homomorphism of $\C$-algebras $\overline\pi \colon \fC_\lambda \to \fA_\lambda$, and every $\fA_\lambda$-module $V$ admits a structure of $\fC_\lambda$-module via the pull-back along $\overline\pi$.  Denote this $\fC_\lambda$-module by $\overline\pi^*V$.

\item \label{rmk:actions.delta}
Let $\lambda, \mu \in X^+$ be such that $\lambda+\mu\in X^+$, and recall that the action of the superalgebra $\U{\g \otimes A}$ on $W_A(\lambda) \otimes W_A (\mu)$ is induced by the comultiplication $\Delta \colon \bu(\g \otimes A) \to \bu(\g \otimes A) \otimes \bu(\g \otimes A)$.  In particular, we have $x (w_\lambda \otimes w_\mu) = (x w_\lambda) \otimes w_\mu + w_\lambda \otimes (x w_\mu)$ for all $x \in \g \otimes A$, and thus $w_\lambda \otimes w_\mu$ is a highest-weight vector in $W_A(\lambda) \otimes W_A(\mu)$.  Hence, there exists a unique surjective homomorphism of $\g \otimes A$-modules $\xi \colon W_A(\lambda + \mu) \twoheadrightarrow W_A(\lambda) \otimes W_A(\mu)$ satisfying $\xi( w_{\lambda+\mu} ) = w_\lambda \otimes w_\mu$.  Now, notice that $\fR_A^{\lambda+\mu} \xi$ is a surjective homomorphism of $\U{\h \otimes A}$-modules:
\[
\fR_A^{\lambda+\mu} \xi \colon \U{\h \otimes A} w_{\lambda+\mu} \twoheadrightarrow \U{\h \otimes A}w_\lambda \otimes \U{\h \otimes A}w_\mu.
\]
Moreover, since $\U{\h \otimes A} w_\nu \cong \fA_\nu$ for all $\nu \in X^+$, $\fR_A^{\lambda+\mu} \xi$ induces a homomorphism of commutative $\C$-algebras $\Delta_{\lambda, \mu} \colon \fA_{\lambda+\mu} \to \fA_\lambda \otimes \fA_\mu$.  Thus, given an $\fA_\lambda$-module $M$ and an $\fA_\mu$-module $N$, their tensor product $M \otimes N$ admits an $\fA_{\lambda + \mu}$-module structure via the pull-back along $\Delta_{\lambda, \mu}$.  Denote this $\fA_{\lambda+\mu}$-module by $\Delta_{\lambda,\mu}^* (M\otimes N)$.
\end{enumerate}
\end{rmk}

\details{
Using the remarks above and arguments similar to those used in  \cite[Corollary~5]{CFK10}, one can prove this next result.

\begin{prop}\label{prop:epi-tensor-property}
Let $\lambda,\mu\in X^+$ be such that $\lambda+\mu\in X^+$, $C = A \oplus B$, and $\pi_A \colon C \twoheadrightarrow A$, $\pi_B\colon C \twoheadrightarrow B$ denote the canonical surjective homomorphisms of $\C$-algebras.  If $M \in \mod{\fA_\lambda}$ and $N \in \mod{\fB_\mu}$, then there exists a surjective homomorphism of $\g\otimes C$-modules
\[
\tau \colon \fW_C^{\lambda+\mu}\left(\Delta_{\lambda,\mu}^* (M\otimes N)\right) \twoheadrightarrow \pi_A^*(\fW_A^\lambda M)\otimes \pi_B^*(\fW_B^\mu N).
\]
\end{prop}

\dproof{
Since $C$ is assumed to be equal to $A \oplus B$, we have that $\U{\g \otimes C} \cong \U{\g \otimes A} \otimes \U{\g \otimes B}$ and that $\pi_A^* W_A(\lambda) \otimes \pi_B^* W_B(\mu)$ is a $\U{\g \otimes C}$-module with (see Remark~\ref{rmk:actions}\eqref{rmk:actions.pi})
\[
(u_a u_b) (w_1 \otimes w_2) = (u_a w_1) \otimes (u_b w_2),
\quad u_a \in \U{\g \otimes A},\ u_b \in \U{\g \otimes B},\ w_1 \in W_A (\lambda),\ w_2 \in W_B (\mu).
\]
Thus, notice that $w_\lambda \otimes w_\mu$ generates $\pi_A^* W_A(\lambda) \otimes \pi_B^* W_B(\mu)$ and satisfies \eqref{cha-fou-equ5}.  Hence, there exists a unique surjective homomorphism of $\U{\g \otimes C}$-modules
\[
\phi \colon W_C(\lambda+\mu) \to \pi_A^*W_A(\lambda) \otimes \pi_B^*W_B(\mu)
\quad \textup{ such that} \quad
\phi (w_{\lambda+\mu}) = w_\lambda \otimes w_\mu.
\]

Now, recall from \eqref{eq:right action} that $W_A(\lambda)$ (resp. $W_B(\mu)$) is a (right) $\fA_\lambda$-module (resp. right $\fB_\mu$-module), with the (right) action of $\fA_\lambda$ (resp. $\fB_\mu$) being induced from the (left) action of $\U\ga$ (resp. $\U{\g \otimes B}$).  Hence, by Remark~\ref{rmk:actions}\eqref{rmk:actions.ann}, $\overline\pi_A^*W_A(\lambda)$ is a $\fC_\lambda$-module and $\overline\pi_B^*W_B(\mu)$ is a $\fC_\mu$-module, by Remark~\ref{rmk:actions}\eqref{rmk:actions.pi}, $\Delta_{\lambda,\mu}^* \left( \overline\pi_A^*W_A(\lambda) \otimes \overline\pi_B^*W_B(\mu) \right)$ is a $\fC_{\lambda+\mu}$-module, and by \eqref{eq:right action}, $\phi$ induces is a surjective homomorphism of $\fC_{\lambda+\mu}$-modules (which we denote by the same symbol)
\[
\phi \colon W_C(\lambda+\mu) \to \Delta_{\lambda,\mu}^* \left( \overline\pi_A^*W_A(\lambda) \otimes \overline\pi_B^*W_B(\mu) \right)
\quad \textup{ such that} \quad
\phi (w_{\lambda+\mu}) = w_\lambda \otimes w_\mu.
\]
Notice that as a vector space $\pi_A^*W_A(\lambda) \otimes \pi_B^*W_B(\mu)$ is isomorphic to $\Delta_{\lambda,\mu}^* \left( \overline\pi_A^*W_A(\lambda) \otimes \overline\pi_B^*W_B(\mu) \right)$.  So, for the rest of this proof, we will abuse notation and denote by $W_A(\lambda) \otimes W_B(\mu)$ the (left) $\g \otimes C$-module $\pi_A^*W_A(\lambda) \otimes \pi_B^*W_B(\mu)$ and the (right) $\fC_\lambda$-module $\Delta_{\lambda,\mu}^* \left( \overline\pi_A^*W_A(\lambda) \otimes \overline\pi_B^*W_B(\mu) \right)$.

Since $- \otimes_{\fC_{\lambda+\mu}} \Delta_{\lambda,\mu}^* \left( M \otimes N \right)$ is a right exact functor, the map
\[
\phi \otimes \id \colon
W_C (\lambda + \mu) \otimes_{\fC_{\lambda+\mu}} \Delta_{\lambda,\mu}^* \left( M \otimes N \right) \to \left( W_A(\lambda) \otimes W_B(\mu) \right) \otimes_{\fC_{\lambda+\mu}} \Delta_{\lambda,\mu}^* \left( M \otimes N \right)
\]
is in fact a surjective homomorphism of $(\g\otimes C, \fC_{\lambda+\mu})$-bimodules.

Finally, using \cite[Lemma~3.1.7(2)]{Kum02}, we see that the map
\[
\begin{array}{rcl}
\psi \colon \left( W_A(\lambda) \otimes W_B(\mu) \right) \otimes_{\fC_{\lambda+\mu}} \Delta_{\lambda, \mu}^* \left( M \otimes N \right) &\longrightarrow& \pi_A^*(\fW_A^\lambda M) \otimes \pi_B^*(\fW_B^\mu N) \\
(w \otimes w') \otimes (m \otimes n) & \longmapsto & (w \otimes m)\otimes (w' \otimes n)
\end{array}
\]
is in fact an isomorphism of $\g\otimes C$-modules.  Thus $\tau = \psi \circ (\phi \otimes \id)$ is a surjective homomorphism of $\g \otimes C$-modules.
}

When $M$, $N$, $A$, $B$ are finite dimensional and $\g$ is simple, the homomorphism $\tau$ given in Proposition~\ref{prop:epi-tensor-property} is in fact an isomorphism.  In order to prove this fact, we will prove a finite-dimensional version of Theorem~\ref{thm:hom-ext1-vanish}.  We begin with some homological results.

\begin{lemma} \label{lem:fd.ext.vanish}
Let $\lambda \in X^+$, $M \in \cal C_A^\lambda$ and $m \ge 0$.  If $\Ext^m_{\cal C_A^\lambda} ( M, N ) = 0$ for all finite-dimensional irreducible $N \in \cal C_A^\lambda$ satisfying $N_\lambda = 0$, then $\Ext^m_{\cal C_A^\lambda} ( M, N ) = 0$ for all finite-dimensional $N \in \cal C_A^\lambda$ satisfying $N_\lambda = 0$.
\end{lemma}

\begin{proof}
If $N$ is irreducible, then the result follows directly from the hypothesis.  Now, suppose that $N$ is not irreducible and consider its composition series
		\[
N\supsetneq N_1 \supsetneq N_2 \supsetneq \dotsb \supsetneq N_m \supsetneq 0.
		\]
The short exact sequence $0 \to N_1 \to N \to N/N_1 \to 0$ yields a long exact sequence
	\[
\cdots \to \Ext^m_{\cal C_A^\lambda} (M,N_1) \to \Ext^m_{\cal C_A^\lambda} (M,N) \to \Ext^m_{\cal C_A^\lambda} (M,N/N_1) \to \cdots
	\]
Since the length of $N_1$ is less than the length of $N$ and $(N_1)_\lambda\subseteq N_\lambda=0$, it follows by induction that $\Ext^m_{\cal C_A^\lambda} ( M, N_1 ) = 0$. Moreover, since $N/N_1$ is an irreducible finite-dimensional object of $\cal C_A^\lambda$ satisfying $(N/N_1)_\lambda=0$, we have $\Ext^m_{\cal C_A^\lambda} ( M, N/N_1 ) = 0$.  The result follows.
\qedhere
\end{proof}

\begin{lemma}\label{prop8}
Let $\lambda \in X^+$.  A finite-dimensional module $M \in \cal C_A^\lambda$ satisfies $\fW_A^\lambda\fR_A^\lambda M \cong M$, if and only if,  for every finite-dimensional, irreducible $N \in {\cal C}_A^\lambda$ satisfying $N_\lambda = 0$, we have
\begin{equation}\label{eq:10-prop8-CFK}
\Hom_{\cal C_A^\lambda} ( M, N )
= \Ext_{\cal C_A^\lambda}^1 ( M, N ) = 0.
\end{equation}
\end{lemma}

\begin{proof}
The \emph{only if} part is an immediate consequence of Theorem~\ref{thm:hom-ext1-vanish}.  The proof of the \emph{if} part is similar to the proof the \emph{if} part of Theorem~\ref{thm:hom-ext1-vanish}, but uses Lemma~\ref{lem:fd.ext.vanish} and the finite dimensionality of $M$ and $N$.  Namely, assume that $M$ is a finite-dimensional object of ${\cal C}_A^\lambda$ such that, for each irreducible finite-dimensional $N \in {\cal C}_A^\lambda$ that satisfies $N_\lambda = 0$, we have $\Hom_{\cal C_A^\lambda} \left( M, N \right) = \Ext_{{\cal C}_A^\lambda}^1 (M, N) = 0$.  Using Proposition~\ref{prop:glob.Weyl.genrel+univ},  we see that there is a unique homomorphism of $\ga$-modules $\epsilon_M \colon \fW_A^\lambda \fR_A^\lambda M \to M$ satisfying
\[
\epsilon_M (u w_\lambda \otimes m_\lambda) = u m_\lambda
\quad \textup{for all $u \in \U\ga$ and $m_\lambda \in M_\lambda$}.
\]

Denote the submodule $\U\ga M_\lambda \subseteq M$ by $M'$, and consider the short exact sequence
\[
0 \to M' \inj M \to M / M' \to 0.
\]
Since $\fR_A^\lambda$ is an exact functor, $0 \to M'_\lambda \inj M_\lambda \to (M / M')_\lambda \to 0$ is also an exact sequence.  Now notice that by construction, $M'_\lambda = M_\lambda$.  Hence $M/M'$ is finite dimensional and $(M / M')_\lambda = 0$.  This implies, by hypothesis, that $\Hom_{\cal C_A^\lambda} \left( M, M/M' \right) = 0$.  Thus $M = M' = \U\ga M_\lambda$.  Hence, the homomorphism $\epsilon_M$ is surjective, and we have a short exact sequence
\[
0 \to \ker \epsilon_M \to \fW_A^\lambda \fR_A^\lambda M \to M \to 0.
\]
Consider the associated long exact sequence on $\Ext^\bullet_{\cal C_A^\lambda} \left( - , \ker\epsilon_M \right)$.  In particular, we have:
\[
\cdots \to
\Hom_{\cal C_A^\lambda} \left( \fW_A^\lambda \fR_A^\lambda M, \ker \epsilon_M \right) \to
\Hom_{\cal C_A^\lambda} \left( \ker \epsilon_M, \ker \epsilon_M \right) \to
\Ext^1_{\cal C_A^\lambda} \left( M, \ker \epsilon_M \right) \to \cdots
\]
Now, notice that $(\ker \epsilon_M)_\lambda = 0$, as $\epsilon_M$ maps $\left( \fW_A^\lambda \fR_A^\lambda M \right)_\lambda$ isomorphically onto $M_\lambda$.  This implies by Remark~\ref{rmk:hom.vanish}, that $\Hom_{\cal C_A^\lambda} \left( \fW_A^\lambda \fR_A^\lambda M , \ker \epsilon_M \right) = 0$, and by hypothesis, that $\Ext^1_{\cal C_A^\lambda} \left( M , \ker \epsilon_M \right) = 0$.  Hence, $\Hom_{{\cal C}_A^\lambda} (\ker \epsilon_M, \ker \epsilon_M) = 0$.  Thus $\ker \epsilon_M = 0$, and $\epsilon_M$ is an isomorphism between $\fW_A^\lambda \fR_A^\lambda M$ and $M$.\qedhere
\end{proof}
}

\begin{thm} \label{thm:tensor-property2}
Let $\g$ be a finite-dimensional simple Lie superalgebra not of type $\lie q(n)$, with a fixed triangular decomposition satisfying \hyperref[C2]{$(\mathfrak C2)$}. Suppose also that $A$ and $B$ are finite-dimensional commutative, associative $\C$-algebras with unit and let $\pi_A\colon A\oplus B\twoheadrightarrow A$ and $\pi_B\colon A\oplus B\twoheadrightarrow B$ be the canonical projections. Let $\lambda,\mu\in X^+$ be such that $\lambda+\mu\in X^+$.  If $M \in \mod{\fA_\lambda}$, $N \in \mod{\fB_\mu}$ are finite dimensional, then there is an isomorphism of $\g\otimes (A\oplus B)$-modules
\[
\fW_{A\oplus B}^{\lambda+\mu}\left(\Delta_{\lambda,\mu}^* (M\otimes N)\right) \cong \pi_A^*(\fW_A^\lambda M)\otimes \pi_B^*(\fW_B^\mu N).
\]
\end{thm}

\dproof{
Recall from Remark~\ref{rmk:actions}, that as vector spaces, $\pi_A^* \fW_A^\lambda M$ is isomorphic to $\fW_A^\lambda M$, $\pi_B^* \fW_B^\mu N$ is isomorphic to $\fW_B^\mu N$, and $\Delta_{\lambda, \mu}^* (M \otimes N)$ is isomorphic to $(M \otimes N)$.  So, during this proof, we will abuse notation and omit the pull-backs.

Now, notice that
\begin{align*}
\fR_{A \oplus B}^{\lambda + \mu} \left( \fW_A^\lambda M \otimes \fW_B^\mu N \right)
&\cong \sum_{\xi + \eta = \lambda + \mu} (\fW_A^\lambda M)_\xi \otimes (\fW_B^\mu N)_\eta \\
&\cong (\fW_A^\lambda M)_\lambda \otimes (\fW_B^\mu N)_\mu \\
& \cong M \otimes N.
\end{align*}
Thus, $\fW_{A \oplus B}^{\lambda + \mu} \fR_{A \oplus B}^{\lambda + \mu} \left( \fW_A^\lambda M \otimes \fW_B^\mu N \right) \cong \fW_{A \oplus B}^{\lambda + \mu} (M \otimes N)$.  Since $M, N, A$ and $B$ are assumed to be finite dimensional, by Corollary~\ref{cor5.12-FMS}, $\fW_A^\lambda M\otimes \fW_B^\mu N$ is finite dimensional.  Thus, by Lemmas~\ref{lem:fd.ext.vanish} and \ref{prop8}, to prove that $\fW_{A \oplus B}^{\lambda + \mu} \fR_{A \oplus B}^{\lambda + \mu} (\fW_A^\lambda M \otimes \fW_B^\mu N) \cong (\fW_A^\lambda M \otimes \fW_B^\mu N)$ is equivalent to proving that
\[
\Hom_{\cal C_{A\oplus B}^{\lambda+\mu}} (\fW_A^\lambda M \otimes \fW_B^\mu N, U) = \Ext_{\cal C_{A\oplus B}^{\lambda+\mu}}^1 (\fW_A^\lambda M \otimes \fW_B^\mu N, U) = 0
\]
for all finite-dimensional irreducible $U \in {\cal C}_{A\oplus B}^{\lambda+\mu}$ with $U_{\lambda+\mu} = 0$.

Let $U$ be a finite-dimensional irreducible $U \in {\cal C}_{A\oplus B}^{\lambda+\mu}$.  Since $A$ and $B$ are finite dimensional, there exist $\nu_A, \nu_B \in X^+$ such that $U_A$ is an irreducible $\ga$-module of highest-weight $\nu_A$, $U_B$ is an irreducible $\g \otimes B$-module of highest-weight $\nu_B$, and $U \cong U_A \otimes U_B$ (see \cite[Proposition~8.4]{che95}).  Moreover, since $\nu_A+\nu_B$ is in the set of weights of $U$, we have $\nu_A + \nu_B \le \lambda + \mu$.  Now, using the K\"unneth formula, we have:
\begin{align*}
\Hom_{\cal C_{A\oplus B}^{\lambda+\mu}} (\fW_A^\lambda M \otimes \fW_B^\mu N, U) \cong
{ }&{ }\Hom_{\cal C_A^\lambda} (\fW_A^\lambda M, U_A) \otimes \Hom_{\cal C_B^\mu} (\fW_B^\mu N, U_B), \\
\Ext_{\cal C_{A\oplus B}^{\lambda+\mu}}^1 (\fW_A^\lambda M \otimes \fW_B^\mu N, U) \cong
{ }&{ }\Ext_{\cal C_A^\lambda}^1 (\fW_A^\lambda M, U_A) \otimes \Hom_{\cal C_B^\mu} (\fW_B^\mu N, U_B) \\
&\oplus \Hom_{\cal C_A^\lambda} (\fW_A^\mu M, U_A) \otimes \Ext_{\cal C_B^\mu}^1 (\fW_B^\mu N, U_B).
\end{align*}
Thus, if we prove that either $(U_A)_\lambda = 0$ or $(U_B)_\mu = 0$, by Theorem~\ref{thm:hom-ext1-vanish}, we will have finished our proof.

First, assume that either $\nu_A \le \lambda$ or $\nu_B \le \mu$, and recall that $\nu_A + \nu_B \le \lambda + \mu$ by construction.  If $U_{\lambda+\mu} = 0$, then either $(U_A)_\lambda = 0$ or $(U_B)_\mu = 0$.  Now, assume that $\nu_A \not\le \lambda$ and $\nu_B \not\le \mu$.  In this case, if $U_{\lambda+\mu} = 0$ and $\lambda \le \nu_A$, then $\lambda + \nu_B \le \nu_A + \nu_B < \lambda + \mu$, which contradicts the fact that $\nu_B \not\le \mu$.  Thus, since $\lambda \not\le \nu_A$, $\nu_A \not\le \lambda$ and $U_A$ is a highest-weight module of highest weight $\nu_A$, we conclude that $(U_A)_\lambda = 0$.
}


\end{document}